\documentclass[12pt]{amsart}

\usepackage[margin=1in]{geometry}

\usepackage[bookmarksnumbered, colorlinks, plainpages,linkcolor=blue,anchorcolor=blue,citecolor=blue,urlcolor=blue]{hyperref}

\usepackage{amsmath, amssymb, amsthm, mathrsfs, mathtools, enumitem, bbm, color}

\newtheorem{theorem}{Theorem}[section]
\newtheorem{corollary}[theorem]{Corollary}
\newtheorem{lemma}[theorem]{Lemma}
\newtheorem{proposition}[theorem]{Proposition}
\theoremstyle{definition}
\newtheorem{definition}[theorem]{Definition}
\newtheorem{remark}[theorem]{Remark}
\newtheorem{assumption}[theorem]{Assumption}

\numberwithin{equation}{section}

\newcommand{\C}{\mathbb{C}}  
\newcommand{\R}{\mathbb{R}} 
 
\newcommand{\Z}{\mathbb{Z}} 
\newcommand{\N}{\mathbb{N}} 

\newcommand{\Sc}{\mathcal{S}}
\newcommand{\cF}{\mathcal{F}}
\newcommand{\cL}{\mathcal{L}}

\newcommand{\cB}{\mathcal{B}}

\newcommand{\sB}{\mathsf{B}}
\newcommand{\sR}{\mathsf{R}}

\newcommand{\rH}{\mathrm{H}}
\newcommand{\bM}{\mathbf{M}}
\newcommand{\SL}{\Sc_{\sqrL}}

\newcommand{\Dom}{\mathsf{Dom}}

\newcommand{\rL}{\mathrm{L}}
\newcommand{\sqrL}{\ssqrt{\mathbf{L}}}
\newcommand{\id}{{\mathrm{Id}}}
\newcommand{\bL}{\mathbf{L}}

\newcommand{\eps}{\varepsilon}
\newcommand{\ee}{\mathrm{e}}
\newcommand{\ii}{\mathrm{i}}

\newcommand{\dd}{\, \mathrm{d}}
\newcommand{\supp}[1]{\mathrm{supp}(#1)}
\newcommand{\loc}{\mathrm{loc}}
\newcommand{\diag}{\mathrm{diag}}

\newcommand{\jb}{\langle\, \cdot \,\rangle}

\newcommand{\ssqrt}[1]{\sqrt{\smash[b]{#1}}}

\usepackage{color}

\definecolor{gr}{rgb}   {0.,   0.8,   0. } 
\definecolor{bl}{rgb}   {0.,   0.5,   1. } 
\definecolor{mg}{rgb}   {0.7,  0.,    0.7}

\newcommand{\Bk}{\color{black}}

\title[Global Strichartz Estimates ]{Global Strichartz Estimates for Wave Equations with time-dependent structured Lipschitz coefficients}

\author{Dorothee Frey}
\address{Karlsruhe Institute of Technology, Department of Mathematics, Englerstr. 2, 76131 Karlsruhe, Germany}
\email{dorothee.frey@kit.edu}

\author{Yonas Mesfun}
\address{Karlsruhe Institute of Technology, Department of Mathematics, Englerstr. 2, 76131 Karlsruhe, Germany}
\email{yonas.mesfun@kit.edu}

\date{June 24, 2026}
\thanks{Funded by the Deutsche Forschungsgemeinschaft (DFG, German Research Foundation) – Project-ID 258734477 – SFB 1173}

\begin{document}

		\begin{abstract}
		We establish global-in-time Strichartz estimates without loss of derivatives for wave equations with  time-dependent  Lipschitz coefficients, which satisfy an additional structural assumption.  Our approach is based on a parametrix construction through the Phillips functional calculus.  
		 We furthermore obtain the well-posedness of such wave equations with  Lipschitz coefficients in $H^1$. 
	\end{abstract}

		\maketitle

	\section{Introduction and Main Results}
		 In this article, we prove  global-in-time  Strichartz estimates for  variable-coefficient wave equations with  time-dependent  Lipschitz-regular coefficients  under a structural condition.
		 More precisely, we consider wave equations in space-time $\R\times \R^d=\{(t,x)\colon t\in \R, x\in \R^d\}$ with spatial dimension $d\geq 2$, which are of the form
		 \begin{equation}\label{eq:WaveEquationWithTimeDependentLipschitzMetrics}
		 	(D_t^2 -P(t,x, D_x))u(t,x)=F(t,x),\quad u(0,x)=g(x),\quad D_t u(0,x)=h(x).
		 \end{equation}
		 Here, $D_t\coloneqq -\ii\partial_t$, $D_x\coloneqq (D_{x_1},\dots, D_{x_d})\coloneqq -\ii\nabla_x$ and $P(t,x, D_x)$ is an elliptic second order differential operator of the form
		 \begin{equation*}
		 	P(t,x, D_x)\coloneqq \sum_{j=1}^d D_{x_j} c_{j}(t,x)D_{x_j}\quad \text{or}\quad P(t,x, D_x)\coloneqq \sum_{j=1}^d c_{j}(t,x)D_{x_j}^2.
		 \end{equation*}
		 The coefficients $c_1,\dots, c_d$ are merely assumed to be Lipschitz-regular and to satisfy some structural and an asymptotic flatness condition (see Assumption~\ref{ass:MainThm} below).
		 
		 Strichartz estimates control space-time $\rL_t^p\rL_x^q$-norms of the solution of a linear dispersive equation in terms of ($\rL^2$-based) Sobolev norms of its initial data. These estimates turn out to be particularly useful in establishing the well-posedness of nonlinear dispersive equations (see e.g. \cite{Sogge1995}, \cite{Tao2006}, \cite{BahouriDanchinChemin2011} for an account of nonlinear wave equations). For a Schwartz solution $u\in \Sc(\R\times \R^d)$ of the classical wave equation in spatial dimension $d\geq 2$ (which corresponds to $c_j=1$ for all $j\in \{1,\dots, d\}$, i.e., $P=-\Delta_x$ in \eqref{eq:WaveEquationWithTimeDependentLipschitzMetrics}), 
		\begin{equation}\label{eq:ClassicalWaveEquation}
		(D_t^2+\Delta_x) u =F \quad \text{in }\R\times \R^d, \quad u(0,\cdot)=g, \quad D_tu(0,\cdot )=h,
		\end{equation}
		they read
		\begin{equation}\label{eq:StrichartzEstimatesForClassicalWaveEquation}
			\||D_x|^{1-\alpha}u\|_{\rL_t^p(\R;\rL_x^q(\R^d))}\lesssim \|g\|_{\dot{\rH}^{1}(\R^d)}+\|h\|_{\rL^2(\R^d)}+\||D_x|^{\tilde{\alpha}}F\|_{\rL_t^{\tilde{p}'}(\R;\rL_x^{\tilde{q}'}(\R^d))},
		\end{equation}
		where $|D_x|^\beta$ is the Fourier multiplier operator associated with the symbol $|\xi|^\beta$ for $\beta\in \{1-\alpha, \tilde{\alpha}\}$ and $\|g\|_{\dot{\rH}^1(\R^d)}\coloneqq \|\nabla_x g\|_{\rL^2(\R^d)}$ denotes the homogeneous Sobolev norm. The integrability exponents $\tilde{p}'$ and $\tilde{q}'$ in \eqref{eq:StrichartzEstimatesForClassicalWaveEquation} are the Hölder conjugate exponents of $\tilde{p}$ and $\tilde{q}$, respectively, and the triples $(p,q,\alpha), (\tilde{p},\tilde{q},\tilde{\alpha})\in [0,\infty]^3$ satisfy the admissibility conditions
		\begin{equation}\label{eq:StrichartzTriple}
			2\leq p\leq \infty,\,2\leq q < \infty,\quad \frac{1}{p}+\frac{d}{q}=\frac{d}{2}-\alpha, \quad \frac{2}{p}+\frac{d-1}{q}\leq \frac{d-1}{2}.
		\end{equation}
		Triples satisfying \eqref{eq:StrichartzTriple} are called \textit{(wave-admissible) Strichartz triples} and they are referred to as \textit{strict} if the fourth condition in \eqref{eq:StrichartzTriple} holds with an equality. Time-translation invariance, scaling symmetry and the Knapp example show that the conditions in \eqref{eq:StrichartzTriple} are in fact necessary for \eqref{eq:StrichartzEstimatesForClassicalWaveEquation} to hold (see \cite{Tao2006}). We preclude $q=\infty$ in \eqref{eq:StrichartzTriple} because in general, \eqref{eq:StrichartzEstimatesForClassicalWaveEquation} fails to hold in this case (see \cite{FangWang2006}, \cite{Guo2018}). 
		Observe that strict Strichartz triples with $p=\infty$, i.e., $(p,q,\alpha)=(\tilde{p},\tilde{q},\tilde{\alpha})=(\infty,2,0)$ correspond to the well-known energy inequality (see e.g. \cite[p.13]{Sogge1995}). At the other extreme, strict Strichartz triples with $p=2$, i.e., $(2, \frac{2(d-1)}{d-3}, \frac{d+1}{2(d-1)})$ for $d\geq 4$, are called \textit{endpoint} Strichartz triples. Strichartz estimates for non-strict Strichartz triples follow from those for strict ones by Sobolev embedding.
	
		Strichartz estimates for the wave equation have a long history. In his seminal work, Strichartz \cite{Strichartz1977} proved \eqref{eq:StrichartzEstimatesForClassicalWaveEquation} in the case $p=q$ by means of Fourier restriction. Later, Ginibre--Velo \cite{GinibreVelo1995} and Lindblad--Sogge \cite{LindbladSogge1995} independently established \eqref{eq:StrichartzEstimatesForClassicalWaveEquation} for non-endpoint Strichartz triples. Finally, Keel--Tao \cite{KeelTao1998} developed an abstract approach to Strichartz inequalities, thereby settling the subtle endpoint case. 
	
		The primary goal of this article is to obtain  sharp global-in-time  Strichartz estimates for wave equations as in \eqref{eq:WaveEquationWithTimeDependentLipschitzMetrics}, where the coefficients $c_j$ are merely Lipschitz continuous. To this end, we impose the following assumptions on the coefficients $c_j$.
		\begin{assumption}\label{ass:MainThm} \leavevmode 
			We have 
			\begin{equation}\label{eq:StructuralAssumption}
			c_{j}(t,x)\coloneqq  b_j(t)a_j(x_j), \quad (t,x)\in \R\times \R^d,\, j\in \{1,\dots,d\},
			\end{equation}
			where $a_1, \dots, a_d$ and $b_1,\dots, b_d$ are Lipschitz functions from $\R$ to $\R$ with the following properties.
			\begin{enumerate}[label=$\bullet$]
				\item There exist constants $0<m_1\leq m_2 <\infty$ such that
				\begin{equation}\label{eq:Ellipticity_a}
				m_1 \leq a_j(x)\leq m_2 \quad \text{for all } x\in \R \text{ and } j\in \{1,\dots,d \}.
				\end{equation}
				Moreover, we assume that 
				\begin{equation}\label{eq:SizeCondition_a}
				m_3\coloneqq\max_{1\leq j\leq d}\big\|\tfrac{\dd}{\dd x}\log(a_j)\big\|_{\rL^1(\R)}<4.
				\end{equation}
				\item For some sufficiently small $\eps_0\in (0,\tfrac{1}{2})$, we have
				\begin{equation}\label{eq:Ellipticity_b}
				1-\eps_0 \leq b_j(t)\leq 1+\eps_0 \quad \text{for all } t\in \R \text{ and } j\in \{1,\dots,d \}.
				\end{equation}
				We set $m_4\coloneqq \max_{1\leq j\leq d}\|b_j'\|_{\rL^\infty(\R)}<\infty$. Moreover, we assume that there exists some sufficiently small $\eps_1=\eps_1(m_1,m_2,m_4)>0$ such that
				\begin{equation}\label{eq:SizeCondition_b}
				\max_{1\leq j\leq d}\|b_j'\|_{\rL^1(\R)}\leq \eps_1.
				\end{equation}
			\end{enumerate}
		\end{assumption}
		Thus, $P(t,x,D_x)$ is of the form
		\begin{equation}\label{eq:DefOfP(t)}
			P(t,x,D_x)=\sum_{j=1}^d b_j(t) D_{x_j}a_j(x_j)D_{x_j}\quad \text{or}\quad P(t,x,D_x)=\sum_{j=1}^d  b_j(t)a_j(x_j)D_{x_j}^2.
		\end{equation}
		Due to the low regularity of the coefficients, it is %(to the best of the authors' knowledge) 
		a priori not clear if \eqref{eq:WaveEquationWithTimeDependentLipschitzMetrics} admits a distributional solution even for very regular initial data (cf. for instance \cite[Theorem~10]{ColombiniFerruccioSpagnolo1979} in the case of Hölder continuous coefficients). Therefore, in a first step, we address the well-posedness of \eqref{eq:WaveEquationWithTimeDependentLipschitzMetrics} in $\rH^\alpha\coloneqq\rH^\alpha(\R^d)$, the latter denoting the $\rL^2$-based (inhomogeneous) Sobolev space of order $\alpha\in \R$. To this end, we view \eqref{eq:WaveEquationWithTimeDependentLipschitzMetrics} as a vector-valued Cauchy problem in $ \rH^\alpha$. It is convenient to write $P(t)$ instead of $P(t,x,D_x)$, and we think of $P\coloneqq (P(t))_{t\in \R}$ as a family of differential operators acting on functions of $x$. We use the following notion of a weak solution.
		\begin{definition}[Weak Solutions in $\rH^\alpha$]\label{def:WeakSolutions}Let $\alpha\in \R$ and suppose that $g\in \rH^\alpha, h\in \rH^{\alpha-1}$ and $F\in \rL^1(\R;\rH^{\alpha-1})$. Then, a function $u\in  C(\R; \rH^{\alpha})\cap C^1(\R; \rH^{\alpha-1})\cap W^{2,1}_\loc(\R;\rH^{\alpha-2})$ is called a \textit{weak solution} of \eqref{eq:WaveEquationWithTimeDependentLipschitzMetrics} (in $\rH^\alpha$) if 
		\begin{align*}
			\left\{
			\begin{aligned}
				D_t^2 u(t) &= P(t)u(t)+F(t) && \text{in } \rH^{\alpha-2} \quad \text{for a.e. $t \in \R$}, \\
				u(0)      &= g, \\
				D_t u(0)  &= h.
			\end{aligned}
			\right.
		\end{align*}	
		\end{definition}
		Then, we obtain the following well-posedness result.
		\begin{theorem}[Well-posedness in $\rH^\alpha$]\label{thm:wellposedness}
		Let $\alpha\in[-1,2]$ and suppose that $g\in \rH^{\alpha}$, $h\in \rH^{\alpha-1}$ and $F\in \rL^1(\R;\rH^{\alpha-1})$. Then, there exists a unique weak solution $u$ to \eqref{eq:WaveEquationWithTimeDependentLipschitzMetrics}. Moreover, for each bounded interval $J\subseteq \R$ with $0\in J$, the map 
		\begin{equation*}
				\Phi\colon \rH^\alpha \times \rH^{\alpha-1}\times \rL^1(J;\rH^{\alpha-1})\to C(\overline{J}; \rH^{\alpha}),\quad (g,h,F)\mapsto u
			\end{equation*}
		is Lipschitz continuous.
		\end{theorem}
		We note that for Theorem~\ref{thm:wellposedness} we do not need the smallness of $\eps_0$ as stated in \eqref{eq:Ellipticity_b}. In fact, it is enough that the coefficients $b_j$ are just bounded from above and below by positive constants, see Remark~\ref{rem:RelaxAssForWellposedness} below.
		
		For the weak solution from Theorem~\ref{thm:wellposedness}, we are able to derive a representation formula, see Theorem~\ref{thm:WellposednessInAdaptedSpaces} below. This representation then allows us to prove global-in-time Strichartz estimates, which is our main result:
		\begin{theorem}[Global-In-Time Strichartz Estimates for Weak Solutions in $\rH^1$]\label{mainthm:Strichartz} Let $(p,q,\alpha)$ be a wave-admissible Strichartz triple and $\alpha\in [0,2]$. Suppose that $g\in \rH^{1}$, $h\in \rL^{2}$ and $F\in \rL^1(\R;\rL^{2})$. Then, the weak solution to the wave equation \eqref{eq:WaveEquationWithTimeDependentLipschitzMetrics} satisfies the global-in-time Strichartz estimate
		\begin{equation}\label{eq:HomogeneousStrichartzEstimatesForTimeDependentLipschitzmetrics}
			\||D_x|^{1-\alpha}u\|_{\rL_t^p(\R;\rL_x^q(\R^d))}\lesssim \|g\|_{\rH^1}+\|h\|_{\rL^2}+\|F\|_{\rL^1(\R;\rL^2)}.
		\end{equation}
		\end{theorem}
 In 	\cite[Theorem~2]{FreySchippa2023}, Schippa and the first author	  obtained Theorem~\ref{mainthm:Strichartz} in the special case where $P$ as in \eqref{eq:DefOfP(t)} is time-independent, i.e., $b_j=1$ for all $j\in \{1\dots,d\}$.  Theorem~\ref{mainthm:Strichartz} extends this result to the case of time-dependent coefficients of multiplicative form  \eqref{eq:DefOfP(t)}. 
		% that we can allow for 'multiplicative' time-dependence. 
	
		One striking feature of Theorem~\ref{mainthm:Strichartz} is the mere Lipschitz regularity of the coefficients. Indeed, Strichartz estimates for wave equations with low-regularity coefficients have been intensively studied in the literature, leading to the insight that $C^2$- or at least $C^{1,1}$-regularity of the coefficients is the minimal requirement for Strichartz estimates to be valid in general. Indeed, for more general wave equations
		\begin{equation}\label{eq:GeneralWaveEquation}
		(D_t^2 -\tilde{P}(t,x, D_x))u(t,x)=F(t,x),\quad u(0,x)=g(x),\quad D_t u(0,x)=h(x),%\quad (t,x)\in \R\times \R^d,
		\end{equation}
		in $\R\times \R^d$ with elliptic differential operator $\tilde{P}(t,x,D_x)=\sum_{j,k=1}^d D_{x_j} c_{jk}(t,x)D_{x_k}$, Smith \cite{SmithParametrix1998} first obtained local(-in-time)\footnote{By this, we mean that \eqref{eq:StrichartzEstimatesForClassicalWaveEquation} holds with $\||D_x|^{1-\alpha}u\|_{\rL_t^p(\R;\rL_x^q(\R^d))}$ replaced by $\||D_x|^{1-\alpha}u\|_{\rL_t^p(I;\rL_x^q(\R^d))}$ or $\|\langle D_x\rangle^{1-\alpha}u\|_{\rL_t^p(I;\rL_x^q(\R^d))}$ for some finite interval $I\subseteq \R$ containing $0$.} Strichartz estimates in spatial dimensions $d=2,3$ in the homogeneous case ($F=0$), provided that $c_{jk}\in C^{1,1}$. Later, Tataru \cite{Tataru2001} established local Strichartz estimates under a $C^2$-assumption in all dimensions $d\geq 2$, and Metcalfe--Tataru \cite{MetcalfeTataru2012} proved that these estimates are global-in-time under an additional asymptotic flatness condition. For $C^r$-coefficients with $r\in [0,2)$, Strichartz estimates hold true with loss of derivatives \cite{Tataru2001}, in which case the regularity parameter $\alpha$ of the Strichartz triple $(p,q,\alpha)$ in \eqref{eq:StrichartzTriple} has to be replaced by $\alpha+\frac{\sigma_r}{p}$ with $\sigma_r\coloneqq \frac{2-r}{2+r}$. Moreover, the counterexamples constructed in \cite{SmithTataru2002} show that this loss is generally unavoidable. In light of these results, Theorem~\ref{mainthm:Strichartz} shows that for the wave equations under consideration in this article, Strichartz estimates in fact do hold without loss of derivatives. The reason for this essentially lies in the crucial structural assumption \eqref{eq:StructuralAssumption}, see also Subsection~\ref{section:Methods} below for a further discussion.
	
	\subsection{Methods}\label{section:Methods}
	In our approach, we make use of  the abstract result by  Keel--Tao \cite{KeelTao1998},  which states  that Strichartz estimates follow from uniform $\rL^2$-bounds and dispersive estimates for the solution operators to \eqref{eq:GeneralWaveEquation}. However, one of the the main difficulties  in our situation  is to find 	 	a good enough approximation of the solution operator, i.e., a {parametrix}. 
	 In case of the classical wave equation, the solution operator to \eqref{eq:ClassicalWaveEquation} can be represented in terms of Fourier multiplier operators of the form
	\begin{equation}\label{eq:SolutionOperatorsForClassicalWaveEquation}
		\big[\tilde{T}(t)f\big](x)\coloneqq \big[\ee^{\ii t|D_x|}\psi(D_x)f\big](x)=\frac{1}{(2\pi)^d}\int_{\R^d} \ee^{\ii (x\xi+t|\xi|)}\psi(\xi) \widehat{f}(\xi)\dd \xi
	\end{equation}
	with $0\neq \psi\in C_c^\infty(\R^d)$ supported away from the origin. While in this case the uniform $\rL^2$-bounds are an immediate consequence of Plancherel's Theorem, the dispersive estimates are considerably more delicate. In order to derive them, we may rewrite \eqref{eq:SolutionOperatorsForClassicalWaveEquation} as a convolution according to
	\begin{equation}\label{eq:SolutionOperatorsForClassicalWaveEquation2}
		\big[\tilde{T}(t)f\big](x)=\int_{\R^d} \tilde{K}_t(y) f(x-y)\dd y=\int_{\R^d} \tilde{K}_t(y) (\ee^{-\ii y\cdot D_x}f)(x)\dd y,
	\end{equation}
	where $\tilde{K}_t(y)\coloneqq \cF^{-1}(\ee^{\ii t |\cdot|}\psi)(y)$
	 and $(\ee^{\ii y\cdot D_x})_{y\in \R^d}$ denotes the translation group acting on functions of $x$.
	The crucial dispersive estimate for the operator $\tilde{T}(t)$,
	\begin{equation}
		\|\tilde{T}(t)f\|_{\rL^\infty(\R^d)}\lesssim (1+|t|)^{-\frac{d-1}{2}}\|f\|_{\rL^1(\R^d)},
	\end{equation}
	then follows from the bound $|\tilde{K}_t(y)|\lesssim (1+|t|)^{-\frac{d-1}{2}}$, which in turn follows from the theory of oscillatory integrals (see e.g. \cite[Chapter~8]{Stein1993}, \cite[Section~7.7]{Hoermander1990}).
	
	For variable-coefficient wave equations, 
	one has to rely on more sophisticated tools from microlocal analysis, namely {Fourier integral operators} (FIOs) in the case of smooth coefficients (see \cite{Kapitanski1989_1}, \cite{MockenhauptSeegerSogge1993}), or, in the case of coefficients with limited regularity, on phase space methods such as the {FBI transform} (see \cite{Tataru2001}, \cite{Tataru2002}). Those latter methods  are not readily available  if the regularity of the coefficients of $P$ is lower than $C^2$.
	
	The assumption that allows us to fall below the '$C^2$-barrier' is the  structural assumption \eqref{eq:StructuralAssumption}, which goes back to \cite{FrePortal2020} on fixed-time $\rL^p$-estimates for rough wave equations   and which opens the door to a construction of a parametrix via tools from functional calculus. Indeed, \eqref{eq:StructuralAssumption} means that $P(t)=\sum_{j=1}^d b_j(t)L_j$, where $L_1,\dots, L_d$ are commuting differential operators given by
	\begin{equation*}
		L_j\coloneqq D_{x_j}a_j(x_j)D_{x_j}\quad\text{or}\quad L_j\coloneqq a_j(x_j)D^2_{x_j}.
	\end{equation*}
	It turns out that $\ii\sqrL=\ii(\sqrt{L_1},\dots,\sqrt{L_d})$ is the generator of a $d$-parameter $C_0$-group $(\ee^{\ii y\cdot \sqrL})_{y\in \R^d}$ on $\rL^p\coloneqq \rL^p(\R^d)$, $p\in (1,\infty)$, and, in analogy to \eqref{eq:SolutionOperatorsForClassicalWaveEquation2}, we are able to define a parametrix for \eqref{eq:WaveEquationWithTimeDependentLipschitzMetrics} in terms of operators of the form 
	 \begin{equation}
	 	T(t)f\coloneqq (\ee^{\ii \varphi_t}\psi)(\sqrL)f\coloneqq \int_{\R^d} K_t(y)\ee^{-\ii y\cdot \sqrL}f \dd y
	 \end{equation}
	 with a suitable one-homogeneous phase function $\varphi_t\colon \R^d\to \R$ and the kernel $K_t$ given by $K_t(y)=\cF^{-1}(\ee^{\ii \varphi_t}\psi)(y)\coloneqq$$ \frac{1}{(2\pi)^d}\int_{\R^d} \ee^{\ii(y\cdot \xi +\varphi_t(\xi))}\psi(\xi) \dd \xi$. 
 Let us emphasize that this representation cannot be achieved through a simple bi-Lipschitz change of variables, see \cite[Remark 4.5]{FrePortal2020}. 
	 The representation then allows to reduce the crucial dispersive estimate for $T(t)f$ to a corresponding oscillatory integral estimate for $K_t$, namely
	 \begin{equation}\label{eq:DispersiveEstimate}
	 	|K_t(y)|\lesssim (1+|t|)^{-\frac{d-1}{2}}.
	 \end{equation}
	 We deduce \eqref{eq:DispersiveEstimate} from decay estimates for the Fourier transform of a surface-carried measure. This is where \eqref{eq:Ellipticity_b} enters the picture. Indeed, the smallness of $\eps_0$ as in \eqref{eq:Ellipticity_b} guarantees that the hypersurface $\mathbb{S}\coloneqq \{\xi\in \R^d\colon \varphi(\xi)=1\}$ is a small perturbation of the unit sphere, and as such inherits its nonvanishing Gaussian curvature. An application of stationary phase then yields \eqref{eq:DispersiveEstimate}.
	 
	\subsection{Outline} We briefly describe the structure of this article. In Section~\ref{section:Preliminaries}, we recall the operator-theoretic framework, which was introduced in \cite{FrePortal2020}. In particular, we shortly review the Phillips functional calculus for $\ssqrt{\bL}$ and recall basic facts and tools of this calculus. In Section~\ref{section:ParametrixConstruction}, we use these tools in order to construct parametrices, i.e., approximate solution operators for \eqref{eq:WaveEquationWithTimeDependentLipschitzMetrics}. Following  e.g.   \cite{SmithParametrix1998} (see also \cite[Section~4]{Kapitanski1989_2}), we then use these parametrices in Section~\ref{section:ExistenceOfWeakSolutions} to obtain a weak solution to \eqref{eq:WaveEquationWithTimeDependentLipschitzMetrics} by an iterative procedure. Uniqueness is shown by energy methods in Section~\ref{section:UniquenessOfWeakSolutions}. Finally, we prove in Section~\ref{section:GlobalStrichartzEstimates} our main result, Theorem~\ref{mainthm:Strichartz}. Our proof is based on the Keel--Tao argument \cite{KeelTao1998} and closely follows the proof of \cite[Theorem~2]{FreySchippa2023}. Here, the crucial ingredient is the dispersive estimate \eqref{eq:DispersiveEstimate}, the proof of which we have postponed to Section~\ref{section:OscillatoryIntegralEstimate}.

	\subsection{Notation}

 Since the underlying first order operators are Dirac operators as in \cite{FrePortal2020} (cf. \cite{FreyMcPortal2018} and the references therein), we define	the second order elliptic operator $P(t)$ from \eqref{eq:DefOfP(t)} as the matrix-valued operator
	\begin{equation*}
	P(t)=\begin{pmatrix}
		\sum_{j=1}^d b_j(t) D_{x_j}a_j(x_j)D_{x_j} & 0\\
		 0 & \sum_{j=1}^d  b_j(t)a_j(x_j)D_{x_j}^2
	\end{pmatrix}
	\end{equation*}	
	and implicitly work with $\C^2$-valued functions.  This allows to treat the divergence-form and standard-form versions of $P(t)$ simultaneously. 
	To ease noation, we suppress the target space and write $\rL^p\coloneqq \rL^p(\R^d;\C^2)$ for $p\in [1,\infty]$, $\rH^\alpha\coloneqq \rH^\alpha(\R^d;\C^2)$, etc. We denote the Schwartz space by $\Sc(\R^d)$ and the Fourier transform on $\Sc(\R^d)$ by $\cF$, where we use the convention 
	\begin{align*}
	\cF f(\xi)\coloneqq \int_{\R^d} \ee^{-\ii y\cdot \xi} f(y) \dd y\quad \text{ and }\quad \cF^{-1}g(y)=\frac{1}{(2\pi)^d} \int_{\R^d} \ee^{\ii y\cdot \xi} g(\xi) \dd \xi 
	\end{align*}
	for $f,g\in \Sc(\R^d)$. The extension of $\cF$ onto $\Sc'(\R^d)$, the space of tempered distributions, is still denoted by $\cF$. We set $\cF \rL^1\coloneqq \cF(\rL^1(\R^d))$ and $\cF\bM\coloneqq \cF(\bM(\R^d))$, where $\bM(\R^d)$ denotes the space of complex-valued Borel measures on $\R^d$ with finite variation norm; we put $\|\varphi\|_{\cF\bM}\coloneqq \|\cF^{-1}\varphi\|_{\bM(\R^d)}$ for $\varphi\in \cF\bM$. For $\lambda>0$ and $\xi\in \R^d$, we set $\langle \lambda \rangle \coloneqq (1+\lambda^2)^{1/2}$ and $\langle \xi \rangle \coloneqq \langle |\xi| \rangle $, respectively. By a slight abuse of notation, we denote both the dual operator (in Banach spaces) and the adjoint operator (in Hilbert spaces) of $T$ by $T^\ast$. We reserve the letter $\lambda>0$ to denote either the size $\xi \in \R^d$ in frequency space or a scaling parameter in physical space. Finally, $A\lesssim_M B$ means $A\leq C B$ for some constant $C>0$ depending on a set $M$ of parameters, and $A\simeq_M B$ denotes the relations $A\lesssim_M B$ and $B\lesssim_M A$.

	\section{Preliminaries}\label{section:Preliminaries}
	In this section, we recall some properties of the Phillips functional calculus used in this article. For more information on this calculus, we refer the reader to e.g. \cite[Section~3.3]{Haase2006}, \cite[Section~10.7]{HytonenVanNeervenVeraarWeis2017}.

	\subsection{Functional Calculus} For each $j\in \{1,\dots, d\}$ and $p\in (1,\infty)$, we define the operator $L_{j,p}$ in $\rL^p=\rL^p(\R^d;\C^2)$ by
	\begin{align*}\label{eq:DiracOperators}
	L_{j,p}\coloneqq{}& \begin{pmatrix}
		D_{x_j}a_j(x_j)D_{x_j} & 0\\
		0& a_{j}(x_j)D_{x_j}^2 &
	\end{pmatrix},\\[0.2cm]
	\Dom(L_{j,p})\coloneqq{}&\{f\in \rL^p\mid D_{x_j}^k u\in \rL^p \text{ for }k\in \{0,1,2\}\}.
	\end{align*}
	The operators $L_{j,p}$ are sectorial operators in $\rL^p$ by \cite[Proposition~1.1]{McIntoshNahmod2000} and \cite[Section 7.3, Theorem~3.2]{Pazy2011} (the latter theorem is only formulated for bounded domains $\Omega$ but can be modified for $\Omega=\R^d$, see \cite[Proposition~2.2.2]{Mesfun2026}). By the holomorphic functional calculus for sectorial operators, the square roots $\ssqrt{L_{j,p}}$ are well-defined and we write $\ssqrt{\mathbf{L}_p}\coloneqq (\ssqrt{L_{1,p}},\dots, \ssqrt{L_{d,p}})$. For the proof of the following proposition, we refer to \cite[Proposition~2]{FreySchippa2023} (see also \cite[Theorem~2.2.8 and Remark~2.2.9~(3)]{Mesfun2026}).
	\begin{proposition}[Generation of $d$-parameter $C_0$-Group] Let $p\in (1,\infty)$. Then, the operator $\ii\ssqrt{\bL_p}$ generates a bounded $d$-parameter $C_0$-group $\big(\ee^{\ii y\cdot \ssqrt{\bL_p}}\big)_{y\in \R^d}$ on $\rL^p$. Moreover, $\ee^{\ii y\cdot \ssqrt{\bL_p}}f=\ee^{\ii y\cdot \ssqrt{\bL_2}}f$ for all $f\in \rL^p\cap \rL^2$.
	\end{proposition}
	We set $M_p\coloneqq \sup_{y\in \R^d} \|\ee^{\ii y\cdot \ssqrt{\bL_p}}\|_{\cL(\rL^p)}.$ Since $\ee^{\ii y\cdot \ssqrt{\mathbf{L}_p}}f=\ee^{\ii y\cdot \ssqrt{\mathbf{L}_2}}f$ for $f\in \rL^p\cap \rL^2$, there is no danger of ambiguity and we will just write $\ee^{\ii y\cdot \ssqrt{\mathbf{L}}}$ instead of $\ee^{\ii y\cdot \ssqrt{\mathbf{L}_p}}$. The bounded $d$-parameter $C_0$-group  $(\ee^{\ii y\cdot \ssqrt{\bL}})_{y\in \R^d}$ gives rise to the so-called \textit{Phillips functional calculus} for $\sqrL$.
	\begin{definition}[Phillips Functional Calculus for $\sqrL$] Let $p\in (1,\infty)$. For $\varphi\in \cF \bM$, we define $\varphi(\sqrL)\in \cL(\rL^p)$ by
		\begin{equation}\label{eq:DefinitionPhillipsCalculus}
			\varphi(\sqrL)\colon \rL^p\to \rL^p, \quad \varphi(\sqrL)f\coloneqq \int_{\R^d} \ee^{-\ii y\cdot \ssqrt{\mathbf{L}}}f \dd\mu_\varphi (y),
		\end{equation}
	where $\mu_\varphi\coloneqq\mathcal{F}^{-1}\varphi\in \bM(\R^d)$.
	\end{definition}
	We will mostly work with $\varphi\in \cF\rL^1\subseteq \cF\bM$, for which \eqref{eq:DefinitionPhillipsCalculus} is just
	\begin{equation}\label{eq:DefinitionPhillipsCalculus2}
	 \varphi(\sqrL)f= \int_{\R^d} (\mathcal{F}^{-1}\varphi)(y)\ee^{-\ii y\cdot \ssqrt{\mathbf{L}}}f \dd y\quad (f\in \rL^p).
	\end{equation}
	
	\begin{proposition}[Properties of the Phillips Functional Calculus]\label{prop:PropertiesPhillipsCalculus}Let $p\in (1,\infty)$. The following statements hold true.
		\begin{enumerate}[label=(\alph*)]
			\item The map $\Phi\colon 	\cF\bM \to \cL(\rL^p),\: \varphi \mapsto \varphi(\ssqrt{\bL})$, defines a unit-preserving, bounded homomorphism of Banach algebras. Moreover, $(\ee^{\ii te_j \cdot })(\ssqrt{\bL})=\ee^{\ii t\ssqrt{L_j}}$ for $t\in \R$ and $j\in \{1,\dots,d\}$. 
			\item Let $f\in \rL^p$ and $\varphi\in C_c^\infty(\R^d)$.Then, for all $\alpha\in \N_0^d$, we have $\varphi(\ssqrt{\bL})f\in \Dom(\ssqrt{\bL}^\alpha)$ and
				\begin{equation}\label{eq:PhillipsCalculusDerivativeAndMultiplication}
					\ssqrt{\bL}^\alpha \varphi(\ssqrt{\bL})f=\varphi_\alpha(\ssqrt{\bL})f,
				\end{equation}
				where $\varphi_\alpha(\xi)\coloneqq \xi^\alpha \varphi(\xi)$, $\xi \in \R^d$.
				Moreover, $\ssqrt{\bL}^\alpha \varphi(\ssqrt{\bL})f=\varphi(\ssqrt{\bL})\ssqrt{\bL}^\alpha f$ if $f\in \Dom(\ssqrt{\bL}^\alpha)$.
				\item Let $\varphi\in \cF \rL^1$. Then,
				\begin{equation*}
					\|\varphi(\ssqrt{\bL})\|_{\cL(\rL^p)}\lesssim_{M_p} \|\varphi(D_x)\|_{\cL(\rL^p)},
				\end{equation*}
				where $\varphi(D_x)\coloneqq \mathcal{F}^{-1}M_\varphi \mathcal{F}$ denotes the Fourier multiplier operator associated with the symbol $\varphi$. In particular, we have
				\begin{align*}%\label{eq:Transference_l2}
					\|\varphi(\ssqrt{\bL})\|_{\cL(\rL^2)}&\lesssim_{M_2} \|\varphi\|_{\rL^\infty},\\[0.2cm] \label{eq:Transference_lp}
					\|\varphi(\sqrL)\|_{\cL(\rL^p)}&\lesssim_{M_p,p,d} \sup_{|\alpha|\leq [d/2]+1}\big\|\, |\cdot |^{|\alpha|}  \partial_{\xi}^{\alpha}\varphi(\cdot )\big\|_{\rL^\infty}, \quad p\in(1,\infty).
				\end{align*}
			\end{enumerate}
		\end{proposition}
		\begin{proof} See \cite[Proposition~2.1.10]{Mesfun2026}.	
		\end{proof}
	\begin{remark}\label{rem:PhillipsCalculusForSymmetricFunctions} If $\varphi\in \cF \rL^1$ is even in every component, i.e., if
		\begin{equation*}
			\varphi(\xi)=\varphi^{(\mathrm{e})}(\xi)\coloneqq \frac{1}{2^d}\sum_{(\delta_j)_{j=1}^d\in \{-1,1\}^d} \varphi(\delta_1\xi_1,\dots, \delta_d \xi_d)\quad (\xi \in \R^d),
		\end{equation*}
		then the same is true for $\rho\coloneqq \cF^{-1}\varphi$. In this case, we have by Fubini's theorem
		\begin{align*}
			\varphi(\sqrL)f={}& \int_{\R^{d-1}}\bigg(\int_\R  \rho^{(\ee)}(y)\frac{1}{2}\big(\ee^{-\ii y_1 \ssqrt{L_1}}+\ee^{-\ii y_1 \ssqrt{L_1}}\big)  \dd y_1\bigg)\ee^{-\ii y'\cdot \ssqrt{\mathbf{L'}}}f\dd y'\\
			={}&\int_{\R^{d-1}}\bigg(\int_\R  \rho^{(\ee)}(y)\frac{1}{2}\big(\ee^{\ii y_1 \ssqrt{L_1}}+\ee^{-\ii y_1 \ssqrt{L_1}}\big)  \dd y_1\bigg)\ee^{-\ii y'\cdot \ssqrt{\mathbf{L'}}}f\dd y'\\
			\eqqcolon {}&  \int_{\R^{d-1}}\bigg(\int_\R  \rho^{(\ee)}(y)\mathrm{Cos}(y_1\ssqrt{L_1})\big)  \dd y_1\bigg)\ee^{-\ii y'\cdot \ssqrt{\mathbf{L'}}}f\dd y'\\
			={}&\cdots= \int_{\R^d}\rho^{(\ee)}(y)\mathrm{Cos}(y\sqrL)f \dd y=\int_{\R^d}\rho(y)\mathrm{Cos}(y\sqrL)f
		\end{align*}
		with the $d$-parameter cosine function $\mathrm{Cos}(y\ssqrt{\bL})\coloneqq \prod_{j=1}^d  \mathrm{Cos}(y_j\ssqrt{L_j})$.
	\end{remark}
	Thus for such functions $\varphi$, we are free to replace $\ee^{-\ii y\cdot \sqrL}$ by $\mathrm{Cos}(y \sqrL)$ in the integral which defines $\varphi(\sqrL)$. This observation turns out to be useful when combined with the following proposition. 
	\begin{proposition}[$\rL^\infty_x\rL^1_y$-Estimate]\label{prop:L^inftyL^1-Estimate} Let $p\in (1,\infty)$. 
		\begin{enumerate}[label=(\alph*)]
			\item We have \[\|\mathrm{Cos}(y\sqrL)f(x)\|_{\rL^\infty_x(\R;\rL^1_y(\R^d))}\lesssim \|f\|_1\quad (f\in \rL^1\cap \rL^p),\] 
			with an implicit constant depending on $m_1,m_2,m_3$ from Assumption~\ref{ass:MainThm}. 
			\item If $\varphi\in \cF(\rL^1\cap \rL^\infty)$ with $\varphi=\varphi^{(\mathrm{e})}$, then
			 \[\|\varphi(\sqrL)f\|_{\infty}\lesssim \|\cF^{-1}\varphi\|_\infty \|f\|_1\quad (f\in \rL^1\cap \rL^p).\] 
		\end{enumerate}
	\end{proposition}
	\begin{proof} For a proof of (a), we refer the reader to \cite[Theorem~1.1]{BeliIgnatZuaZua2016} (see also \cite[Theorem~2.2.8 and Remark~2.2.9~(3)]{Mesfun2026}). To prove (b), let $\varphi\in \cF(\rL^1\cap \rL^\infty)$ with $\varphi=\varphi^{(\mathrm{e})}$. Then, by Remark~\ref{rem:PhillipsCalculusForSymmetricFunctions}, Hölder's inequality and part~(a), we have for $f\in \rL^1\cap \rL^p$ and a.e. $x\in \R^d$
	\begin{align*}
		|\varphi(\sqrL)f(x)| \leq{}& \int_{\R^d} |(\cF^{-1}\varphi)(y)| |\mathrm{Cos}(y\sqrL)f(x)| \dd y\\
		\leq{}& \|\cF^{-1}\varphi\|_{\infty}	\|\mathrm{Cos}(y\sqrL)f(x)\|_{\rL^\infty_x(\R;\rL^1_y(\R^d))}\lesssim \|\cF^{-1}\varphi\|_{\infty} \|f\|_1.
	\end{align*}
	\end{proof}
	\subsection{An Adapted Scale of Sobolev spaces} Fourier multiplier operators can be viewed as the Phillips functional calculus for $D_x$. These operators can be efficiently studied in the scale of $\rL^2$-based Sobolev spaces $\rH^\alpha$, using Plancherel's Theorem and the characterization of these spaces within this calculus, i.e., $\rH^\alpha=\{u\in \Sc'\colon \langle D_x \rangle^\alpha u \in \rL^2 \}$ $(\alpha\in \R)$. As we replace the Fourier multiplier calculus (for $D_x$) by the Phillips functional calculus for $\sqrL$, it is more convenient to replace the classical Sobolev spaces $\rH^\alpha$ by a scale of spaces which is better adapted to $\sqrL$. To this end, we consider the operator $L$ in $\rL^2= \rL^2(\R^d;\C^2)$ defined by
	\begin{equation}\label{eq:OperatorL}
		L\coloneqq L_{1,2}+\cdots +L_{d,2}=\begin{pmatrix}
			\sum_{j=1}^d D_{x_j} a_j(x_j)D_{x_j} & 0\\
			0 &  \sum_{j=1}^d a_{j}(x_j)D_{x_j}^2
		\end{pmatrix}
	\end{equation}
	with domain $\rH^2=\rH^2(\R^d;\C^2)$. Note that $L$ is self-adjoint with respect to the equivalent scalar product
	\begin{equation}\label{eq:EquivalentScalarProduct}
		\langle u,v \rangle_{A}\coloneqq \langle A^{-1}u,v\rangle_{\rL^2}, \quad \text{with}\quad A=A(x)=\begin{pmatrix}
			1 & 0\\
			0 & a(x)
		\end{pmatrix}, \quad a(x)\coloneqq a_1(x_1)\cdots a_d(x_d).
	\end{equation}
	Moreover, $L$ is injective with spectrum contained in $[0,\infty)$. In particular, the fractional powers $L^\alpha$ are defined for each $\alpha\in \R$ (see e.g. \cite[Section~3.2]{Haase2006}). As in \cite[Appendix~E]{KunstmannWeis2004}, we may therefore use the following definition.
	\begin{definition}[The Spaces $\rH^\alpha_L$] Let $\alpha\in \R$. We define the \textit{$L$-adapted ($\rL^2$-based) Sobolev space of order $\alpha$} by 
	\begin{equation*}
		(\rH_L^\alpha, \|\cdot\|_{\rH^{\alpha}_L})\coloneqq 
		\begin{cases}
			(\Dom(L^{\frac{\alpha}{2}}), \|(\id+L)^{\frac{\alpha}{2}} \cdot\|_2), & \alpha\geq 0,\\
			(\rL^2,\|(\id+L)^{\frac{\alpha}{2}} \cdot\|_2)^{\sim}, & \alpha<0.
		\end{cases}
	\end{equation*}	
	Here, $(X,\|\cdot\|_X)^\sim$  denotes the completion of a normed space $(X,\|\cdot\|_X)$. We also set $\rH^\infty_L\coloneqq \bigcap_{\alpha\in \R} \rH^\alpha_L$ and $\rH^{-\infty}_L\coloneqq \bigcup_{\alpha\in \R} \rH^\alpha_L$. 
	\end{definition}
	Note that $\rH^\alpha_L\subseteq \rH^\beta_L$ for $-\infty\leq\beta\leq \alpha\leq\infty$. To ease notation, we will frequently write $\|\cdot\|_\alpha$ instead of $\|\cdot\|_{\rH^{\alpha}_L}$. The notation $\rH^\alpha_L$ is motivated by the observation that $\rH^\alpha_{-\Delta_x}=\rH^\alpha$ for all $\alpha\in \R$. Let $\beta\in\R$.
	It follows from \cite[ Proposition~15.23]{KunstmannWeis2004} that the fractional power $(\id+L)^\frac{\beta}{2}$ extends to an operator  $ (\id+\mathscr{L})^{\frac{\beta}{2}}$ on $\rH_L^{-\infty}$ such that $(\id+\mathscr{L})^{\frac{\beta}{2}}\colon \rH_L^\alpha\to \rH^{\alpha-\beta}_L$ is a bounded isomorphism for all $\alpha\in \R$. We write $\langle D_L\rangle^{\beta}\coloneqq (\id+\mathscr{L})^{\frac{\beta}{2}}$. We need the following lemma to extend the Phillips functional calculus to the spaces $\rH^\alpha_L$, at least for compactly supported smooth functions.
	\begin{lemma}\label{lemma:ExtensionOfPhillipsFunctionalCalculus} Let $\varphi\in C_c^\infty(\R^d)$ and $\psi\in C_c^\infty(\R^d)$ with $\supp{\psi}\subseteq \R^d\setminus\{0\}$. Let further $f\in \rL^2$, and $\alpha,\beta\in \R$. Then:
	\begin{enumerate}[label=(\alph*)]
			\item $(\id+L)^{\frac{\alpha}{2}}\varphi(\sqrL)f=(\jb^{\alpha}\varphi)(\sqrL)f$, and if $f\in \Dom(L^{\frac{\alpha}{2}})$, then $(\id+L)^{\frac{\alpha}{2}}\varphi(\sqrL)f=\varphi(\sqrL)(\id+L)^{\frac{\alpha}{2}}f$,
			\item  $L^{\frac{\beta}{2}}\psi(\sqrL)f=(| \cdot |^{\beta}\psi)(\sqrL)f$. Moreover, if $f\in \Dom(L^\frac{\beta}{2})$, then $L^{\frac{\beta}{2}}\psi(\sqrL)f=\psi(\sqrL)L^{\frac{\beta}{2}}f$.
	\end{enumerate}
	\end{lemma}
	\begin{proof} It is straightforward to verify that $\ee^{-tL}=G_t(\sqrL)$, where $G_t(\xi)\coloneqq \ee^{-t|\xi|^2}$. Then, one may use subordination formulas such as \cite[Proposition~3.3.5]{Haase2006} to prove (a). Part (b) can be proved similarly. 
	\end{proof}
	For $\varphi\in C_c^\infty(\R^d)$ and $f\in \rH^\alpha_L$, the Phillips functional calculus for $\sqrL$ extends to $\rH^\alpha_L$ by letting
	\begin{equation}\label{eq:ExtensionOfPhillipsFunctionalCalculus}
		\varphi(\sqrL)f\coloneqq (\jb^{-\alpha} \varphi)(\sqrL)g,\quad g\coloneqq \langle D_L\rangle^\alpha f\in \rL^2.
	\end{equation}
	By Lemma~\ref{lemma:ExtensionOfPhillipsFunctionalCalculus}~(a), this definition does not depend on $\alpha$ and is therefore reasonable. In view of  Lemma~\ref{lemma:ExtensionOfPhillipsFunctionalCalculus}~(b) and \eqref{eq:ExtensionOfPhillipsFunctionalCalculus}, we see that 
	\begin{equation}\label{eq:RuleForFractionalPowers}
		L^{\frac{\beta}{2}} \psi(\sqrL)f=(|\cdot|^{\beta}\psi)(\sqrL)f\quad (f\in \rH^\alpha_L)
	\end{equation}
	for all $\alpha,\beta\in \R$ and all smooth $\psi\in C_c^\infty(\R^d)$ supported away from the origin. It follows similarly that
	\begin{equation}\label{eq:RuleForD_L}
		\langle D_L\rangle^\beta \varphi(\sqrL)f=(\jb^{\beta}\varphi)(\sqrL)f\quad (f\in \rH^\alpha_L)
	\end{equation}
	for all $\alpha,\beta\in \R$ and $\varphi\in C_c^\infty(\R^d)$. We use the following result in order to switch from the $L$-adapted scale to the classical one.
 The $L$-adapted $\rL^q$-based homogeneous Sobolev space $\dot{W}^{\beta,q}_L(\R^d)$ is defined as in \cite[Definition~15.21]{KunstmannWeis2004} with $A$ replaced by the closure of the $\rL^q$-part of $\ssqrt{L}$ and $\alpha$ replaced by $\beta$ therein. 
	
	\begin{proposition}\label{prop:Kato}\begin{enumerate}[label=(\alph*)]
			\item For $\alpha\in [-2,2],$ we have $\rH^\alpha_L=\rH^\alpha$ and
			\begin{equation*}
				\|\langle D_L\rangle^\alpha u\|_2\simeq \|\langle D_x\rangle^\alpha u\|_2 \quad (u\in \rH^\alpha).
			\end{equation*}
			\item For all $\beta\in [-1,1]$ and $q\in (1,\infty)$, we have $\dot{W}^{\beta,q}(\R^d)=\dot{W}^{\beta,q}_L(\R^d)$ with equivalent norms $\|L^{\frac{\beta}{2}} u\|_q\simeq \||D_x|^\beta u \|_q$. 
		\end{enumerate}
	\end{proposition}

	\begin{proof} Consider first the case $\alpha=2$. Then, $\rH_L^2=\Dom(L)=\rH^2$ by definition. The estimate $\|u\|_{\rH_L^2}\lesssim \|u\|_{\rH^2}$ is trivial (by the Lipschitz continuity of the $a_j$), and the reverse one then follows from the open mapping theorem. Thus, $\rH^2_L=\rH^2$ with equivalent norms. For $\alpha=0$, there is nothing to prove. For intermediate values $\alpha\in (0,2)$, one argues by interpolation, as both $(\rH^\beta_L)_{\beta\in \R}$ and $(\rH^\beta)_{\beta\in \R}$ form complex interpolation scales \cite[Theorem~4.17]{Lunardi2009}. The case $\alpha\in [-2,0)$ follows by duality.  The second assertion  for $q\in(1,\infty)$ and $\beta=1$ follows from the Kato square root problem,  see \cite[Corollary~5.19]{AuscherMcInstoshTchamitchian1998} for the case of Lipschitz coefficients.  	
	\end{proof}
	
	\subsection{Littlewood--Paley-type Estimates}
	In order to prove Theorem~\ref{thm:wellposedness}, we rely on results of Littlewood--Paley type for the Phillips functional calculus for $\sqrL$. Proofs can be found in \cite[Section~2]{Mesfun2026}. In order to state the results, we fix some standard (homogeneous) Littlewood--Paley partition of unity, i.e., some radially symmetric $\psi\in C_c^\infty(\R^d)$ supported in the annulus $\{\xi \in \R^d\colon \frac{1}{2}<|\xi|<2\}$ such that we have
	\begin{equation}\label{eq:LittlewoodPaleyPartitionOfUnity}
		\sum_{\lambda \in 2^\Z} \psi_\lambda(\xi)=1\quad \text{for all} \quad \xi \neq 0,\quad \text{where }  \psi_\lambda(\xi)\coloneqq \psi\bigg(\frac{\xi}{\lambda}\bigg).
	\end{equation}

	\begin{proposition}[Calderón Reproducing Formula in $\rH^\alpha_L$]\label{prop:CalderonReproducingFormulaOnInhomogeneousSobolevSpaces} 
	Let $\alpha\in \R$ and $f\in \rH^\alpha_L$. Then,
		\begin{equation*}
			\sum_{\lambda\in 2^\Z} \psi_\lambda(\sqrL)f =f.
		\end{equation*}
	In particular, 
		\begin{equation}\label{eq:DenseSpace}
			\Sc_{\sqrL}\coloneqq  \mathrm{span}\{\psi(\sqrL)f\mid \psi,f\in C_c^\infty(\R^d), \supp{\psi}\in \R^d\setminus\{0\} \}
		\end{equation}
	belongs to $\bigcap_{\beta\in \R}\Dom(L^\beta)\cap \bigcap_{p\in (1,\infty)}\rL^p$ and is dense in $\rH^\alpha_L$. 
	\end{proposition}
	
	\begin{proposition}[$\rL^2$-Boundedness of Almost Orthogonal Operators]\label{prop:L^2BoundednessOfAlmostOrthogonalOperators} 
			Let $\alpha,\gamma\in \R$. Assume that $h\coloneqq (h_\lambda)_{\lambda\in 2^\Z}$ is a sequence of functions on $\R^d$ such that $h_\lambda$ is smooth on an open neighborhood of $K_\lambda\coloneqq \supp{\psi_\lambda}$ for all $\lambda\in 2^\Z$ and $\|h\|_\infty\coloneqq \sup_{\lambda\in 2^\Z} \|h_\lambda\|_{\rL^\infty(K_\lambda)}<\infty$. Then,
			\begin{equation*}
				T\colon \rH^\alpha_L\to \rH_L^{\alpha-\gamma}, \quad Tf\coloneqq \sum_{\lambda\in 2^\Z} \langle\lambda\rangle ^\gamma (h_\lambda\psi_\lambda)(\sqrL)f
			\end{equation*}
			is a well-defined linear bounded operator with $\|T\|\lesssim \|h\|_\infty$. Moreover, for $f\in \rH^\alpha_L$, there holds
			\begin{equation*}
				\|Tf\|_{\rH^{\alpha-\gamma}_L}^2\lesssim \sum_{\lambda\in 2^\Z} \|(h_\lambda \psi_\lambda)(\sqrL)g\|_2^2\lesssim \sum_{\lambda\in 2^\Z} \|\psi_\lambda(\sqrL)g\|_2^2\simeq \|f\|^2_{\rH^\alpha_L}, 
			\end{equation*}
			where $g\coloneqq \langle D_L \rangle^\alpha f\in \rL^2$. 
	\end{proposition}
	 
	\begin{lemma}[Differentiation under the Sum]\label{lemma:DifferentiationUnderTheSum} Let $\alpha\in \R$, $f\in \rH^{\alpha}_L$ and $n\in \N_0$. Suppose that $\{h_{t,\lambda} \mid \lambda\in 2^\Z, t\in \R\}$ is a set of smooth functions on $\R^d\setminus\{0\}$ which satisfy the following assumptions:
	\begin{enumerate}[label=(\roman*)]
		\item For all $\lambda\in 2^\Z$ and each $\xi \in K_\lambda\coloneqq \supp{\psi_\lambda}$, the map $t\mapsto h_{t,\lambda}(\xi)$ belongs to $C^n(\R;\C)$ and
		\begin{equation*}
			|\partial_t^k h_{t,\lambda}(\xi)|\leq C(t)\langle \xi \rangle^{k} \quad (0\leq k \leq n)
		\end{equation*}
		with some locally bounded
		%\footnote{By this, we mean that $\sup_{t\in I}C(t)<\infty$ for any bounded Interval $I\subseteq \R$.} 
		function $t\mapsto C(t)$.
		\item For each fixed $t_0\in \R, \lambda\in 2^\Z$, we have 
		\begin{equation*}
			\big\|\partial_t^n h_{t,\lambda}-\partial_t^n h_{t_0,\lambda}\big\|_{L_\xi^\infty(K_\lambda)}\rightarrow 0 \quad (t\to t_0).
		\end{equation*}
	\end{enumerate}
	Then, for all $k\in \{0,\dots,n\}$, the map \begin{equation*}
		u\colon t\mapsto T(t)f\coloneqq \sum_{\lambda\in 2^\Z}  (h_{t,\lambda}\psi_\lambda)(\sqrL)f
	\end{equation*}
	belongs to $C^k(\R;\rH_L^{\alpha-k})$ and it holds
	\begin{equation*}
		u^{(k)}(t)=\sum_{\lambda\in 2^\Z}  (\partial_t^k h_{t,\lambda}\psi_\lambda)(\sqrL)f,\quad \|u^{(k)}(t)\|_{\alpha-k}\lesssim C(t)\|f\|_{\alpha}\quad (t\in \R).
	\end{equation*}
	\end{lemma}
	\begin{proof} This is a straightforward consequence of Proposition~\ref{prop:L^2BoundednessOfAlmostOrthogonalOperators} and dominated convergence in $\ell^1(2^\Z)$.
	\end{proof}
	
	\begin{remark}\label{rem:DifferentiationUnderTheSum} We also have a straightforward modification of Lemma~\ref{lemma:DifferentiationUnderTheSum} in the case when $t\mapsto h_{t,\lambda}$ is only Lipschitz continuous (i.e., only differentiable a.e.). More precisely, let $n\in \N_0, \alpha\in \R$, $f\in \rH^{\alpha}_L$ and suppose that $\{h_{t,\lambda} \mid \lambda\in 2^\Z, t\in \R\}$ is a set of smooth functions on $\R^d\setminus\{0\}$ which satisfy: For all $\lambda\in 2^\Z$ and each $\xi \in K_\lambda\coloneqq \supp{\psi_\lambda}$, the map $t\mapsto h_{t,\lambda}(\xi)$ is Lipschitz continuous and
			\begin{equation*}
				|\partial_t h_{t,\lambda}(\xi)|\leq C(t)\langle \lambda\rangle^{n}\quad \text{for a.e. } t\in \R,
			\end{equation*}
			where $t\mapsto C(t)$ is some locally bounded function.
			
	Then, the map
		\begin{equation*}
			u\colon t\mapsto T(t)f\coloneqq \sum_{\lambda\in 2^\Z}  (h_{t,\lambda}\psi_\lambda)(\sqrL)f
		\end{equation*}
		belongs to $W^{1,\infty}(\R;\rH_L^{\alpha-n})$ and it holds
		\begin{equation*}
			u'(t)=\sum_{\lambda\in 2^\Z}  (\partial_t h_{t,\lambda}\psi_\lambda)(\sqrL)f,\quad \|u'(t)\|_{\alpha-n}\lesssim C(t)\|f\|_{\alpha}\quad \text{for a.e. $t\in \R$}.
		\end{equation*}
	\end{remark}
	\begin{proposition}[Squarefunction Characterization of the $\rL^p$-norm]\label{prop:SquareFunctionCharacterizationOfL^pNorm} Let $1<p<\infty$. Then,
		\begin{equation*}
			\left\|\bigg(\sum_{\lambda\in 2^{\Z}} \big|\psi_\lambda(\sqrL)f(x)\big|^2 \bigg)^{1/2}\right\|_p\simeq \|f\|_p \quad \text{for all } f\in \rL^p.
		\end{equation*}	
	\end{proposition}
	\begin{proof} See \cite[Proposition~5]{FreySchippa2023}.
	\end{proof}
	
	\begin{lemma}[Extension of the Operator $L_j$ to $\rH^\alpha_L$]\label{lemma:ExtensionOfL_jToHomogeneousSobolevSpaces} Let $\alpha\in \R$, $n\in \N_0$ and $j\in \{1,\dots,d\}$. Then, the operator $(L_j)^{n}\colon \SL\to \SL$ extends uniquely to a bounded operator from $\rH^{\alpha}_L$ to $\rH^{\alpha-n}_L$. In particular, for each $t\in \R$, we have $P(t)\colon \rH_L^\alpha\to \rH^{\alpha-2}_L$.
	\end{lemma}
	\begin{proof} Uniqueness is clear, since $\SL$ is dense in $\rH^\alpha_L$ by Proposition~\ref{prop:CalderonReproducingFormulaOnInhomogeneousSobolevSpaces}. Now, just note that the operator
		\begin{equation*}
			T\colon \rH_L^\alpha\to \rH^{\alpha-n}_L,\quad Tf\coloneqq \sum_{\lambda\in 2^\Z} \langle \lambda\rangle^n  \big(\big(\tfrac{\xi_j}{\langle\lambda\rangle }\big)^n\psi_\lambda\big)(\sqrL) f
		\end{equation*}
		is bounded by Proposition~\ref{prop:L^2BoundednessOfAlmostOrthogonalOperators} and that by Proposition~\ref{prop:PropertiesPhillipsCalculus}~(b) and Proposition~\ref{prop:CalderonReproducingFormulaOnInhomogeneousSobolevSpaces}, we have $Tf=(L_j)^{\frac{n}{2}} f$ for $f\in \SL$. The second assertion follows from the identity $P(t)=\sum_{j=1}^d b_j(t)L_j$ and the first assertion with $n=2$.
	\end{proof}

	\section{Parametrix Construction}\label{section:ParametrixConstruction}
 We  collect the time-dependent coefficients $b_j$ in a matrix
	\begin{equation}\label{eq:DefinitionOfB(t)}
		B(t):=\diag(b_1(t),\dots, b_d(t))\quad (t\in \R).
	\end{equation}
	In this section, for fixed $s\in \R$ and $u_s\in \rH_L^\alpha$, we want to construct a parametrix for the equation
	\begin{equation}\label{eq:ParametrixConstructionI}
		(D_t^2-P(t))u(t)=0\quad (t\in \R),\quad u(s)=u_s.
	\end{equation}
	Recall that $P(t)=\sum_{j=1}^d b_j(t)L_j$. In view of Proposition~\ref{prop:PropertiesPhillipsCalculus}~(b), $P(t)$ is associated with the symbol $p(t,\xi)\coloneqq (B(t)\xi|\xi)=\sum_{j=1}^d b_j(t)\xi_j^2$,  $(t,\xi)\in \R\times \R^d$, and thus the ODE in \eqref{eq:ParametrixConstructionI} corresponds in the frequency space to the equation
	\begin{equation}\label{eq:ParametrixConstructionII}
		(D_t^2-p(t,\xi))v(t,\xi)=0\quad (t\in \R).
	\end{equation}
	For fixed $\xi\in \R^d$, \eqref{eq:ParametrixConstructionII} is a second-order linear ODE in $t$. In order to find an approximate solution, we make, as in Lax's parametrix construction (see \cite{Lax1957}, \cite[Chapter~VIII.64]{Eskin2011}), the ansatz $v(t,\xi)=\ee^{\ii \varphi_{t,s}(\xi)}$ with an appropriate one-homogeneous phase function $\varphi_{t,s}\colon \R^d\to \R$, whose partial derivatives $\partial_t^k \varphi_{t,s}$ $(k\in \{0,1,2\})$ are also one-homogeneous with respect to $\xi$. Formally, we then have
	\begin{equation*}
		\ee^{-\ii \varphi_{t,s}(\xi)}(D_t^2-p(t,\xi))\ee^{\ii \varphi_{t,s}(\xi)} =(\partial_t \varphi_{t,s}(\xi))^2-p(t,\xi))+\tfrac{1}{\ii}\partial^2_t \varphi_{t,s}(\xi).
	\end{equation*}
	Viewing the right-hand side as a function of the frequency variable $\xi$, we observe that $(\partial_t\varphi_{t,s}(\xi))^2, p(t,\xi)\simeq |\xi|^2$ and $|\partial^2_t \varphi_{t,s}(\xi)|\simeq |\xi|$ by the homogeneity of $\varphi_{t,s}$. Therefore, we might think of $\frac{1}{\ii}\partial^2_t \varphi_{t,s}(\xi)$ as a term of \textit{lower} order (at least in the relevant high-frequency regime $|\xi|\gg1$). Neglecting this term, we arrive at the approximation
	\begin{equation*}
		\ee^{-\ii \varphi_{t,s}(\xi)}(D_t^2-p(t,\xi))\ee^{\ii \varphi_{t,s}(\xi)} \approx(\partial_t \varphi_{t,s}(\xi))^2-p(t,\xi))\overset{!}=0,
	\end{equation*}
	which yields $\partial_t\varphi_{t,s}(\xi) =\pm p(t,\xi)^{\frac{1}{2}}=\pm (B(t)\xi|\xi)^{\frac{1}{2}}$. \Bk Choosing the initial condition $v(s,\xi)=1$ in \eqref{eq:ParametrixConstructionII} gives $\varphi_{s,s}(\xi)=0$, so that an integration with respect to $t$ naturally leads to the following definition.
	\begin{definition}[Definition of Phase Function]\label{def:PhaseFunction} Let $s,t\in \R$. We  define the phase function 
	\begin{equation}\label{eq:PhaseFunction}
		\varphi_{t,s}\colon \R^d\to \R,\quad \varphi_{t,s}(\xi)\coloneqq \int_s^t (B(\tau)\xi|\xi)^{\frac{1}{2}} \dd \tau.
	\end{equation} 	
	\end{definition} 
	For later reference, we state some useful properties of $\varphi_{t,s}$. % which are readily checked by a straightforward computation.
	\begin{lemma}[Regularity of the Phase Function]\label{lemma:DerivativesOfPhaseFunction} The following assertions hold true.
		\begin{enumerate}[label=(\alph*)]
			\item For each $s,t\in \R$, the phase function $\varphi_{t,s}\colon \R^d\to \R$ is smooth away from the origin and positively homogeneous of degree one. Moreover, we have the bounds
			\begin{equation}
				|\partial_\xi^\alpha 	\varphi_{t,s}(\xi)|\lesssim_{\alpha} |\xi|^{1-|\alpha|} |t-s|.\label{eq:XiDerivativesOfVarphi}\\
			\end{equation}
			\item For fixed $s\in \R$ and $\xi\in \R^d\setminus \{0\}$, the map $t\mapsto \varphi_{t,s}$ belongs to $C^{1,1}(\R)$ with
			\begin{align}
				&|\partial_t \varphi_{t,s}(\xi)| \simeq |\xi|,\quad |\partial_t^2 \varphi_{t,s}(\xi)|\lesssim \|B'(t)\|\,|\xi|.\label{eq:BoundTimeDetivativeVarphi}
			\end{align}
			Here, $\partial_t^2 \varphi_{t,s}(\xi)$ exists for almost every $t\in \R$. Moreover, $\varphi_{t,s}=-\varphi_{s,t}$ and thus $\partial_s^k \varphi_{t,s}=-\partial_\tau^k \varphi_{\tau,t}\mid_{\tau=s}$ for all $k\in \{0,1,2\}$.
		\end{enumerate}
	\end{lemma}
	\begin{proof} For $\tau\in \R$, consider the function $f_{\tau}(\xi)\coloneqq (B(\tau)\xi|\xi)^{\frac{1}{2}},$ $\xi\in \R^d$. Clearly, $f_\tau$ is positively homogeneous of order one and smooth away from the origin. Thus, for every $\alpha\in \N_0$,
	\[|\partial_\xi^\alpha f_\tau(\xi)|\leq C_{\alpha,\tau} |\xi|^{1-|\alpha|}\quad \text{with}\quad C_{\alpha,\tau}\coloneqq \sup_{|\xi|=1}|\partial_\xi^\alpha f_\tau(\xi)|.\]	
	By \eqref{eq:Ellipticity_b}, we have $\frac{1}{2}\id \leq B(\tau)\leq \frac{3}{2}\id$, which yields $C_\alpha\coloneqq \sup_{\tau\in \R}C_{\alpha,\tau}<\infty$. Therefore, if $s,t\in \R$ with $s\leq t$, then
	\[|\partial_\xi^\alpha \varphi_{t,s}(\xi)|\leq \int_s^t |\partial_\xi^\alpha f_\tau(\xi)|\dd\tau \leq C_\alpha |\xi|^{1-|\alpha|}\int_s^t \dd \tau=C_\alpha |\xi|^{1-|\alpha|}|t-s|. \]
	The same estimate holds if $s>t$, since $\varphi_{t,s}=-\varphi_{s,t}$. This proves (a). To prove (b), fix $s\in \R$ and $\xi\in \R^d\setminus\{0\}$, and observe that $t\mapsto \varphi_{t,s}(\xi)$ is differentiable with
	\[\partial_t \varphi_{t,s}(\xi)=f_t(\xi)\simeq |\xi|.\]
	Moreover, since $t\mapsto f_t(\xi)$ is Lipschitz continuous, we have for a.e. $t\in \R$
	\[|\partial_t^2 \varphi_{t,s}(\xi)|=|\partial_t f_t(\xi)|=\frac{|(B'(t)\xi|\xi)|}{2f_t(\xi)}\leq \frac{\|B'(t)\| |\xi|^2}{2f_t(\xi)}\simeq \|B'(t)\| |\xi|. \]
	\end{proof}
	
	Now, our natural candidate for a parametrix of \eqref{eq:ParametrixConstructionI} would be $T^\pm(t,s)u_s=(\ee^{\pm \ii  \varphi_{t,s}})(\sqrL)u_s$. However, this expression is not well-defined within the Phillips functional calculus for $\sqrL$ as $\ee^{\pm \ii \varphi_{t,s}}\notin \cF\bM$. To circumvent this subtlety, we define $T^{\pm}(t,s)$  as a limit of Littlewood--Paley sums as in Proposition~\ref{prop:L^2BoundednessOfAlmostOrthogonalOperators}. To this end, we fix a standard Littlewood--Paley partition of unity $(\psi)_{\lambda\in 2^\Z}$ as defined in \eqref{eq:LittlewoodPaleyPartitionOfUnity} for the rest of this article.
	\begin{proposition}[Boundedness of Parametrices I]\label{prop:L^2BoundednessOfParametrix} Let $\alpha\in \R$. Then, the linear operators 
	\begin{equation*}
		T^\pm(t,s)\colon \rH^\alpha_{L}\to \rH^\alpha_{L},\quad  T^\pm(t,s)f=\sum_{\lambda\in 2^{\Z}}(\ee^{\pm \ii \varphi_{t,s} }\psi_\lambda)(\sqrL)f
	\end{equation*}
	are bounded with operator norms uniformly bounded in $s,t\in \R$. 
	\end{proposition}
	\begin{proof} Fix $s,t\in \R$ and put $h_{\lambda}\coloneqq \ee^{\pm \ii\varphi_{t,s}}$ for $\lambda\in 2^\Z$. Since $|h_\lambda|=1$ for all $\lambda\in 2^\Z$, the assertion immediately follows from Proposition~\ref{prop:L^2BoundednessOfAlmostOrthogonalOperators}.
	\end{proof}
	\begin{remark}  One  may wonder why we use the Phillips functional calculus for $\sqrL$ in place of the Borel functional calculus for the self-adjoint operator $L$. There are at least two reasons. First, while $\rho(|\sqrL|)=\rho(L)$ for $\rho\in C_c^\infty([0,\infty))$, the phase function $\varphi_{t,s}$ is generally not radially symmetric, so that $T^{\pm}(t,s)$ cannot be meaningfully expressed within the Borel functional calculus for $L$. Second, the integral representation \eqref{eq:DefinitionPhillipsCalculus2} is concrete enough to deduce dispersive estimates for $T^{\pm}(t,s)$ which are foundational to prove Strichartz estimates.
	\end{remark}

Clearly, the operator $\Box_P\coloneqq  D_t^2 -P(t)$ is a differential operator of order two and therefore loses two derivatives in the $\rH^{\alpha}_L$-scale, i.e., $\Box_P\colon \rH^{\alpha}\to \rH^{\alpha-2}$ boundedly. In the following, we show that $T^\pm(t,s)$ are approximate solution operators to the equation $\Box_Pu=0$, in the sense that $\Box_P T^{\pm}(t,s)$ only loses one derivative in the $\rH^{\alpha}_L$-scale, i.e., $\Box_P T^{\pm}(t,s)\colon \rH_{L}^\alpha\to \rH_{L}^{\alpha-1}$ boundedly, see Theorem~\ref{thm:ParametrixProperty} below. 
% It is in this sense that the operators $T^{\pm}(t,s)$ can be thought of parametrices.
 The derivative gain allows us to construct a weak solution to \eqref{eq:WaveEquationWithTimeDependentLipschitzMetrics} by an iterative procedure. We need the following lemmas as a preparation.
\begin{lemma}[Strong Differentiability of Parametrices I]\label{lemma:TimeDerivativeOfParametrices} Let $s_0,t_0\in \R$ and assume $f\in \rH^\alpha_L$ for some $\alpha\in \R$. Then, the maps
	\begin{equation*}
		t\mapsto T^\pm(t,s_0)f\quad \text{and}\quad  s\mapsto T^\pm(t_0,s)f
	\end{equation*}
	belong to $ C_b(\R; \rH_L^{\alpha})\cap C_b^1(\R; \rH_L^{\alpha-1})\cap W^{2,\infty}(\R; \rH_L^{\alpha-2})$, and for all $k\in \{0,1,2\}$, we have
	\begin{align*}
		&D_t^k T^{\pm}(t,s_0)f={}\sum_{\lambda\in 2^\Z} \big(D_t^k\ee^{\pm \ii  \varphi_{t,s_0}}\psi_\lambda\big)(\sqrL)f,\\
		&D_s^k T^{\pm}(t_0,s)f={}\sum_{\lambda\in 2^\Z} \big(D_s^k\ee^{\pm \ii  \varphi_{t_0,s}}\psi_\lambda\big)(\sqrL)f,\quad \text{ and }\\
		&\|D_t^k T^{\pm}(t,s_0)f\|_{\alpha-k}+\|D_s^k T^{\pm}(t_0,s)f\|_{\alpha-k}\lesssim \|f\|_\alpha.
	\end{align*}
\end{lemma}
\begin{proof} Let $s_0,t_0,\alpha\in \R$ and $f\in \rH^\alpha_L$. Since $T^{\pm}(t,s)=T^{\mp}(s,t)$ for $s,t\in \R$, it suffices to consider the maps $u^\pm \colon t\mapsto T^{\pm}(t,s_0)f$. But in this case, the corresponding statement is just a straightforward consequence of Lemma~\ref{lemma:DifferentiationUnderTheSum}, Remark~\ref{rem:DifferentiationUnderTheSum}, and Lemma~\ref{lemma:DerivativesOfPhaseFunction}. Indeed, put $h_{t,\lambda}\coloneqq \ee^{\pm i\varphi_{t,s_0}}$ for all $t\in \R,\lambda\in 2^\Z$. Then, $h_{t,\lambda}$ is smooth away from the origin and by Lemma~\ref{lemma:DerivativesOfPhaseFunction}, we have the bounds
	\begin{align}\label{eq:BoundsForH}
		\begin{split}
			&|h_{t,\lambda}(\xi)|=1,\\ 
			&|D_t h_{t,\lambda}(\xi)|=|D_t\varphi_{t,s_0}(\xi)|\lesssim|\xi|\lesssim \langle \xi \rangle,\\
			&|D_t h_{t,\lambda}(\xi)-D_t h_{t_0,\lambda}(\xi)|\lesssim \|B(t)-B(t_0)\| |\xi|+|t-t_0| |\xi|^2,\\
			&|D_t^2 h_{t,\lambda}(\xi)|=\big| \big(D_t\varphi_{t,s_0}(\xi)\big)^2+\tfrac{1}{\ii}D_t^2 \varphi_{t,s_0}(\xi)\big|\lesssim |\xi|^2+\|B'(t)\||\xi|\lesssim_{m_4} \langle \xi\rangle^2.
		\end{split}
	\end{align}	
	Lemma \ref{lemma:DifferentiationUnderTheSum} and the first three bounds in \eqref{eq:BoundsForH} yield $u^\pm \in C_b^{k}(\R; \rH_L^{\alpha-k})$ for $k\in \{0,1\}$, while Remark~\ref{rem:DifferentiationUnderTheSum} and the last bound in \eqref{eq:BoundsForH} give $D_t u^\pm\in W^{1,\infty}(\R; \rH_L^{\alpha-2})$. The claim follows.
	\end{proof}
	\begin{lemma}[Derivative Gain on Dyadic Frequencies]\label{lemma:TimeDerivativeFrequencyLocalized} Let $s\in \R$, $\lambda\in 2^\Z$ and $f\in \rH^\alpha_L$  for some $\alpha\in \R$. Then, for a.e. $t\in \R$,
	\begin{equation*}
		\big(D_t^2-P(t)\big)\big(\ee^{\ii\varphi_{t,s}}\psi_\lambda\big)(\sqrL)f=\lambda \big(\ee^{\ii\varphi_{t,s}}\cdot r_{t,\lambda}\cdot \psi_\lambda\big)(\sqrL)f,
	\end{equation*}
	where $r_{t,\lambda}$ is smooth away from the origin with $\|r_{t,\lambda}\|_{\rL^\infty(K_\lambda)}\lesssim \|B'(t)\|$ and $K_\lambda\coloneqq \supp{\psi_\lambda}$.	
	\end{lemma}
	\begin{proof} By Lemma \ref{lemma:DerivativesOfPhaseFunction} and dominated convergence, we have for a.e. $t\in \R$	\begin{equation}\label{eq:TimeDerivativeDominatedConvergence}
			D_t^2(\ee^{\ii\varphi_{t,s}}\psi_\lambda)(\sqrL)f=(D_t^2\ee^{\ii\varphi_{t,s}}\psi_\lambda)(\sqrL)f.
		\end{equation}
		Recalling that $\varphi_{t,s}(\xi)=\int_s^t (B(\tau)\xi|\xi)^{\frac{1}{2}} \dd \tau$, we compute for a.e. $t\in \R$
		\begin{equation*}
			D_t^2 \ee^{\ii\varphi_{t,s}(\xi)}=D_t ((B(t)\xi|\xi)^{\frac{1}{2}}\cdot \ee^{\ii \varphi_{t,s}(\xi)})=\big((B(t)\xi|\xi)+D_t (B(t)\xi|\xi)^{\frac{1}{2}}\big)\ee^{\ii\varphi_{t,s}(\xi)}.
		\end{equation*}
		Plugging this into \eqref{eq:TimeDerivativeDominatedConvergence}, we get
		\begin{align*}
			&{}D_t^2(\ee^{\ii\varphi_{t,s}}\psi_\lambda)(\sqrL)f\\
			={}&((B(t)\xi|\xi)\ee^{\ii\varphi_{t,s}}\psi_\lambda)(\sqrL)f+\lambda(\ee^{\ii\varphi_{t,s}}D_t\big(B(t)\tfrac{\xi}{\lambda}|\tfrac{\xi}{\lambda}\big)^{\frac{1}{2}}\psi_\lambda)(\sqrL)f.
		\end{align*}
		By Proposition~\ref{prop:PropertiesPhillipsCalculus}~(b), the first term on the right-hand side is equal to
		\begin{equation*}
			\sum_{j=1}^d b_j(t)(\xi_j^2\ee^{\ii\varphi_{t,s}}\psi_\lambda)(\sqrL)f=\sum_{j=1}^d b_j(t)L_j(\ee^{\ii\varphi_{t,s}}\psi_\lambda)(\sqrL)f=P(t)(\ee^{\ii\varphi_{t,s}}\psi_\lambda)(\sqrL)f.
		\end{equation*} 
		We therefore conclude
		\begin{equation*}
			\big(D_t^2-P(t)\big)\big(\ee^{\ii\varphi_{t,s}}\psi_\lambda\big)(\sqrL)f=\lambda \big(\ee^{\ii\varphi_{t,s}}D_t(B(t)\tfrac{\xi}{\lambda}|\tfrac{\xi}{\lambda})^{\frac{1}{2}}\psi_\lambda \big)(\sqrL)f.
		\end{equation*}
		But by \eqref{eq:BoundTimeDetivativeVarphi}, we have
		\begin{equation*}
			|(D_t(B(t)\tfrac{\xi}{\lambda}|\tfrac{\xi}{\lambda})^{\frac{1}{2}}|=|D_t^2\varphi_{t,s}(\tfrac{\xi}{\lambda})|\lesssim \|B'(t)\|\, |\tfrac{\xi}{\lambda}|\simeq \|B'(t)\|\quad (\xi \in K_\lambda),
		\end{equation*}
		so the assertion follows setting $r_{t,\lambda}(\xi)\coloneqq D_t(B(t)\tfrac{\xi}{\lambda}|\tfrac{\xi}{\lambda})^{\frac{1}{2}}$.
	\end{proof}
	\begin{theorem}[Derivative Gain I]\label{thm:ParametrixProperty} Let $\alpha\in \R$. Then, the operator
	\begin{equation*}
		\big(D_t^2-P(t)\big)T^{\pm}(t,s)\colon \rH_L^{\alpha}\to  \rH_L^{\alpha-1}
	\end{equation*}
	is bounded for a.e. $s,t\in \R$ with an operator norm uniform in $s$ and $t$. Moreover, we have the  estimates
	\begin{align*}
		(i)\quad&\|\big(D_t^2-P(t)\big)T^{\pm}(t,s)f\|_{\rH^{\alpha-1}_L}\lesssim \|B'(t)\|\cdot  \|f\|_{\rH_L^\alpha},\\
		(ii)\quad&\|\big(D_t^2-P(t)\big)T^{\pm}(t,s)f\|_{\rL_t^1(\R;\rH^{\alpha-1}_L)}\lesssim \|B'\|_{\rL^1} \|f\|_{\rH_L^\alpha}. 
	\end{align*}
	\end{theorem}
	\begin{proof} It is enough to prove the estimates (i) and (ii). Let $s,t,\alpha\in \R$ and $f\in \rH^\alpha_L$. To ease notation, we just give the proof for $T\coloneqq T^+$ (the proof for $T^{-}$ is analogous). By Lemma~\ref{lemma:TimeDerivativeOfParametrices}, we have $D_t^2T(t,s)f\in \rH^{\alpha-2}_L$ almost everywhere. On the other hand, Proposition~\ref{prop:L^2BoundednessOfParametrix} and Lemma~\ref{lemma:ExtensionOfL_jToHomogeneousSobolevSpaces} yield $P(t)T(t,s)f\in  \rH^{\alpha-2}_L$. Thus,  $(D_t^2-P(t))T(t,s)f$ is a well-defined element in $\rH^{\alpha-2}_L$ (for a.e. $t$). Now, Lemma~\ref{lemma:TimeDerivativeOfParametrices} and Lemma~\ref{lemma:TimeDerivativeFrequencyLocalized} imply
	\begin{align*}
		\big(D_t^2-P(t)\big)T(t,s)f={}&\sum_{\lambda\in 2^\Z}\big(D_t^2-P(t)\big)\big(\ee^{\ii\varphi_{t,s}}\psi_\lambda\big)(\sqrL)f\\
		={}&\sum_{\lambda\in 2^\Z} \lambda \big(\ee^{\ii\varphi_{t,s}} r_{t,\lambda}\cdot \psi_\lambda\big)(\sqrL)f\\
		={}&\sum_{\lambda\in 2^\Z} \lambda (h_{t,s,\lambda}\psi_\lambda)(\sqrL)f
	\end{align*}
	with
	\begin{equation*}
		h_{t,s,\lambda}\colon \R^d\setminus\{0\}\to \C,\quad h_{t,s,\lambda}(\xi)=\ee^{\ii\varphi_{t,s}(\xi)}r_{t,\lambda}(\xi)
	\end{equation*}
	satisfying
	\begin{equation*}
		\|h_{t,s,\lambda}\|_{\rL^\infty(K_\lambda)}=\|r_{t,\lambda}\|_{\rL^\infty(K_\lambda)} \lesssim \|B'(t)\| 
	\end{equation*}
	uniformly in $\lambda\in 2^\Z$ and $s\in\R$. It therefore follows from Proposition~\ref{prop:L^2BoundednessOfAlmostOrthogonalOperators} that $\big(D_t^2-P(t)\big)T(t,s)f$ in fact belongs to $\rH^{\alpha-1}_L$ (for a.e. $t$) and that 
	\begin{equation*}
		\|\big(D_t^2-P(t)\big)T(t,s)f\|_{\rH^{\alpha-1}_L}\lesssim \|B'(t)\|\,  \|f\|_{\rH_L^\alpha}.
	\end{equation*}
	This proves (i). Estimate (ii) follows by integrating (i) with respect to $t\in \R$.
	\end{proof}

	Based on the operators $T^{\pm}(t,s)$, we define the following two families of operators $C(t,s)$ and $S(t,s)$ to adjust for the initial values $T(t,t)$ and $D_t T(t,s)|_{t=s}$. They play a similar role as $\mathrm{Cos}((t-s)\sqrt{L})$ and $\ii (t-s)\mathrm{Sinc}((t-s)\ssqrt{L})$ do in the time-independent case.
	\begin{definition}[Parametrices II]\label{def:Parametrices2} Let $s, t,\alpha\in \R$. For $f\in \rH_L^\alpha$, we define
	\begin{align}
		C(t,s)f\coloneqq{}& \hphantom{i}\sum_{\lambda\in 2^\Z} \big((\cos\circ\varphi_{t,s}) \cdot \psi_\lambda\big)(\sqrL)f,\notag\\
		S(t,s)f\coloneqq{}& \ii\sum_{\lambda\in 2^\Z} \bigg(\frac{\sin\circ\varphi_{t,s}}{\partial_s \varphi_{t,s}} \cdot \psi_\lambda\bigg)(\sqrL)f.\label{eq:DefintionSinc}
	\end{align}
	\end{definition}
	The properties established for $T^\pm(t,s)$ extend to $C(t,s)$ in a natural way. The situation for $S(t,s)$ is slightly different. In view of the factor $|\partial_s\varphi_{t,s}|^{-1}\simeq |\xi|^{-1}$ in \eqref{eq:DefintionSinc}, the mapping properties of $S(t,s)$ are even improved by one order in the $\rH^\alpha_L$-scale (see Lemma~\ref{lemma:RelationBetweenParametrices}~(b)~(iv) below). This will be crucial in Section~\ref{section:ExistenceOfWeakSolutions}. Another difference is that the very same factor reduces the strong differentiability of $s\mapsto S(t,s)$ by one order, which fortunately turns out to be irrelevant for our purposes. 
	\begin{lemma}[Relation between Parametrices]\label{lemma:RelationBetweenParametrices} Let $s,t,\alpha\in \R$.
		\begin{enumerate}[label=(\alph*)]
			\item We have
			\begin{equation}\label{eq:RelationBetweenParametricesI}
			C(t,s)f=\frac{1}{2}\big(T^{+}(t,s)+T^{-}(t,s)\big)f\quad (f\in \rH^\alpha_L).
			\end{equation}
			\item Define for $f\in \rH^\alpha_L$
			\begin{equation*}
			\tilde{T}(t,s)f\coloneqq\sum_{\lambda\in 2^{\Z}}(\tilde{h}_{t,s} \psi_\lambda\big)(\sqrL)f, \quad \tilde{h}_{t,s}(\xi)\coloneqq \langle \xi \rangle \frac{\sin(\varphi_{t,s}(\xi))}{\partial_s \varphi_{t,s}(\xi)}.
			\end{equation*}
			Then:
			\begin{enumerate}[label=(\roman*)]
				\item $\tilde{T}(t,s)\in \cL(\rH^\alpha_L)$  with $\|\tilde{T}(t,s)\|\lesssim \langle t-s \rangle$.
				\item The map $\tau\mapsto \tilde{T}(\tau,s)f$ belongs to $\bigcap_{k=0}^1C^k(\R;\rH^{\alpha-k}_L)\cap W_\loc^{2,\infty}(\R;\rH_L^{\alpha-2})$.
				\item The map $\tau\mapsto \tilde{T}(t,\tau)f$ belongs to $C(\R;\rH^{\alpha}_L)\cap W_\loc^{1,\infty}(\R; \rH^{\alpha-1}_L)$.
				\item We have the identity \begin{equation}\label{eq:RelationBetweenParametricesII}
					S(t,s)f=\ii\langle D_L\rangle^{-1}\,\tilde{T}(t,s)f\quad (f\in \rH_L^\alpha).	 
				\end{equation}
				\item We have
				\begin{equation*}
				\big(D_t^2-P(t)\big)(\tilde{h}_{t,s}\psi_\lambda)(\ssqrt{\mathbf{L}})f= \lambda  (\tilde{r}_{t,s,\lambda}\psi_\lambda)(\sqrL)f \quad (\lambda\in 2^\Z),
				\end{equation*}
				where $\tilde{r}_{t,s,\lambda}\colon \R^d\setminus\{0\}\to \C$ is smooth and satisfies the estimate\\ $\|\tilde{r}_{t,s,\lambda}\|_{\rL_\xi^\infty(K_\lambda)}\lesssim \|B'(t)\|$, $K_\lambda\coloneqq \supp{\psi_\lambda}$. 
			\end{enumerate}
		\end{enumerate}	
	\end{lemma}
	\begin{proof} Assertion (a) is obvious. To prove (b),  let us first show  the estimates
	\begin{equation}\label{eq:EstimatesforTildeT}
		\begin{aligned}
			|\tilde{h}_{t,s}(\xi)|&\lesssim \langle t-s\rangle, &\quad
			&|\partial_t \tilde{h}_{t,s}(\xi)|\lesssim \langle \xi \rangle, \\
			|\partial_s \tilde{h}_{t,s}(\xi)|&\lesssim_{m_4} \langle t-s\rangle \langle \xi\rangle ,&\quad & |\partial_t^2 \tilde{h}_{t,s}(\xi)|\lesssim \langle t-s\rangle \langle \xi\rangle^2
		\end{aligned}
	\end{equation}
	and
	\begin{align}\label{eq:BoundsContinuityofHTilde}
	|D_t \tilde{h}_{t,s}(\xi)-D_t \tilde{h}_{t_0,s}(\xi)|\lesssim \langle \xi\rangle^2\big( |t-t_0|+\|B(t)-B(t_0)\|\big)
	\end{align}
	for $s,t,s_0,t_0\in \R$, $\xi\in \R^d\setminus\{0\}$. Indeed, \eqref{eq:BoundTimeDetivativeVarphi} implies
	\begin{equation*}
		|\tilde{h}_{t,s}(\xi)|\leq \frac{\langle \xi\rangle }{|\partial_s \varphi_{t,s}(\xi)|}\simeq \frac{\langle \xi\rangle }{|\xi|}\lesssim 1\quad \text{for }|\xi|\geq 1,
	\end{equation*} 
	while for $|\xi|\leq 1$, \eqref{eq:XiDerivativesOfVarphi}, \eqref{eq:BoundTimeDetivativeVarphi} give
	\begin{equation*}
		|\tilde{h}_{t,s}(\xi)|\leq 2 \frac{|\sin(\varphi_{t,s}(\xi))|}{|\partial_s \varphi_{t,s}(\xi)|}\lesssim \frac{|\varphi_{t,s}(\xi)|}{|\xi|}\lesssim |t-s|.
	\end{equation*}
	This shows $|\tilde{h}_{t,s}(\xi)|\lesssim \langle t-s\rangle$. Using this estimate and \eqref{eq:BoundTimeDetivativeVarphi}, we deduce %\textcolor{red}{ASSUMPTION}
	\begin{align*}
		|\partial_t \tilde{h}_{t,s}(\xi)|&={} \frac{|\cos(\varphi_{t,s}(\xi))|\; |\partial_t \varphi_{t,s}(\xi)|}{|\partial_s \varphi_{t,s}(\xi)|} \langle \xi \rangle \lesssim \langle \xi \rangle,\\[0.1cm]
		|\partial_s \tilde{h}_{t,s}(\xi)|&={}\bigg|\cos(\varphi_{t,s}(\xi)) \langle \xi \rangle -\frac{\partial^2_s \varphi_{t,s}(\xi)}{\partial_s \varphi_{t,s}(\xi)} \tilde{h}_{t,s}(\xi)\bigg|\\
		&\lesssim_{}{} \langle \xi\rangle +\|B'(s)\|\,|\tilde{h}_{t,s}(\xi)|\\
		&\lesssim_{m_4} \langle  \xi \rangle+\langle  t-s\rangle \lesssim \langle t-s\rangle \langle \xi \rangle
	\end{align*}
	and finally
	\begin{align*}
		|\partial_t^2 \tilde{h}_{t,s}(\xi)|&={} \bigg|\tilde{h}_{t,s}(\xi)|\partial_t \varphi_{t,s}(\xi)|^2 -\frac{\partial_t^2\varphi_{t,s}(\xi) }{\partial_s\varphi_{t,s}(\xi)} \cos(\varphi_{t,s}(\xi))\langle \xi \rangle \bigg|     \\
		&\lesssim \langle t-s\rangle |\xi|^2 +\|B'(t)\| \langle\xi \rangle   \lesssim_{m_4} \langle t-s\rangle \langle \xi\rangle^2.
	\end{align*} 
	The bounds \eqref{eq:BoundsContinuityofHTilde} are shown similarly. Using \eqref{eq:EstimatesforTildeT} and \eqref{eq:BoundsContinuityofHTilde}, one may argue as in the proofs of Proposition~\ref{prop:L^2BoundednessOfParametrix} and Lemma~\ref{lemma:TimeDerivativeOfParametrices} to show assertions (i), (ii), and (iii). Furthermore, the boundedness of $\langle D_L\rangle^{-1}\colon \rH_L^{\alpha}\to \rH_L^{\alpha+1}$ and \eqref{eq:RuleForD_L} shows for each $f\in \rH^\alpha_L$ 
	\begin{align*}
		\ii\langle D_L\rangle^{-1}\tilde{T}(t,s)f={}&\ii \sum_{\lambda\in 2^\Z} \langle D_L\rangle^{-1}(\tilde{h}_{t,s}\psi_\lambda)(\sqrL)f\\
		={}&\ii \sum_{\lambda\in 2^\Z}  (\langle \xi \rangle^{-1} \tilde{h}_{t,s}\psi_\lambda)(\sqrL)f=S(t,s)f,
	\end{align*}
	proving (iv). Finally, (v) is exactly shown as Lemma~\ref{lemma:TimeDerivativeFrequencyLocalized}.
	\end{proof}
	The following two theorems are the main results of this section.
	\begin{theorem}[Boundedness and Strong Differentiability of Parametrices~II]\label{thm:GeneralizedCosineAndSineFuntions} Let $\alpha\in \R$. 
	\begin{enumerate}[label=(\alph*)]
		\item Let $s,t\in \R$. Then, the linear operators
		\begin{align}
			& C(t,s)\colon \rH_L^\alpha\to \rH_L^\alpha,\label{eq:L^2BoundednessOfCosineFunction}\\
			& S(t,s)\colon \rH_L^\alpha\to \rH_L^{\alpha+1}\label{eq:L^2BoundednessOfSineFunction}
		\end{align}
		are bounded with $\|C(t,s)\|\lesssim 1$ and $\|S(t,s)\|\lesssim \langle t-s\rangle $.
		\item  Let $f\in \rH^\alpha_L$ and $s_0,t_0\in \R$. Then:
		\begin{enumerate}[label=(\roman*)]
			\item The map $u_{s_0}\colon t\mapsto C(t,s_0)f$ belongs to $ \bigcap_{k=0}^1 C_b^k(\R; \rH_L^{\alpha-k})\cap W^{2,\infty}(\R; \rH_L^{\alpha-2}) $ with \[\|D_t^k u_{s_0}(t)\|_{\alpha-k}\lesssim \|f\|_\alpha,\quad k\in \{0,1,2\}.\] 
			\item The map $v_{s_0}\colon t\mapsto S(t,s_0)f$ belongs to $ \bigcap_{k=0}^1 C^k(\R; \rH_L^{\alpha+1-k}) \cap W_{\loc}^{2,\infty}(\R; \rH_L^{\alpha-1})$ with \[\|D_t^k v_{s_0}(t)\|_{\alpha+1-k}\lesssim \langle t-s_0\rangle^{|k-1|}  \|f\|_\alpha,\quad  k\in \{0,1,2\}.\]
			The map $w_{t_0}\colon s\mapsto S(t_0,s)f$ belongs to $ \dot{W}_{\loc}^{1,\infty}(\R; \rH_L^{\alpha})$.
		\end{enumerate}
	Moreover,
		\begin{alignat}{2}
			C(t_0,t_0)={}& \mathrm{Id}, \qquad&D_t C(t,s)|_{t=s_0}={}&0, \label{eq:InitialConditionsCosineFunction}\\
			S(t_0,t_0)={}& 0,  & D_t S(t,s_0)|_{t=s_0}={}&\mathrm{Id},\label{eq:InitialConditionsSineFunction}
		\end{alignat}
		where the time derivative is to be understood in the strong sense.
	\end{enumerate}
	\end{theorem}
	\begin{proof} In view of the identity \eqref{eq:RelationBetweenParametricesI}, the boundedness of $C(t,s)$ and its strong differentiability as stated in (b) immediately follow from  Proposition~\ref{prop:L^2BoundednessOfParametrix} and  Lemma~\ref{lemma:TimeDerivativeOfParametrices}, respectively. Combining the strong differentiability of $C(t,s)$ with Proposition~\ref{prop:CalderonReproducingFormulaOnInhomogeneousSobolevSpaces} then also yields \eqref{eq:InitialConditionsCosineFunction}. To prove the corresponding statements for $S(t,s)$, we use that $S(t,s)=\ii \langle D_L\rangle^{-1}\tilde{T}(t,s)$ by Lemma~\ref{lemma:RelationBetweenParametrices}~(b)~(iv). Thus, the assertions for $S(t,s)$ follow from the properties of $\tilde{T}(t,s)$ stated in Lemma~\ref{lemma:RelationBetweenParametrices}~(b)~(i)-(iii), and the fact that $\langle D_L\rangle^{-1}$ gains one derivative in the $\rH^\alpha_L$-scale, i.e., $\langle D_L\rangle^{-1}\colon \rH_L^\alpha\to \rH_L^{\alpha+1}$. 
	\end{proof}
	\begin{theorem}[Derivative Gain II]\label{thm:ParametrixPropertyOfSineAndCosineFuntions} Let $s,t,\alpha\in \R$. Then,
	\begin{align*}
		& \big(D_t^2-P(t)\big)C(t,s)\colon \rH_L^{\alpha}\to \rH_L^{\alpha-1},\\
		& \big(D_t^2-P(t)\big)S(t,s)\colon \rH_L^\alpha\to \rH_L^{\alpha}
	\end{align*}
	are bounded. More precisely,  uniformly in $s\in \R$, we have the estimates   %  the following estimates hold true, uniformly in $s\in \R$.
	\begin{align*}
		&\|\big(D_t^2-P(t)\big)C(t,s)f\|_{\rH^{\alpha-1}_L}\lesssim \|B'(t)\|  \|f\|_{\rH_L^\alpha},\\
		&\|\big(D_t^2-P(t)\big)S(t,s)f\|_{\rH^{\alpha}_L}\lesssim \|B'(t)\| \|f\|_{\rH_L^\alpha},
	\end{align*}
	and
	\begin{align*}
		&\|\big(D_t^2-P(t)\big)C(t,s)f\|_{\rL^1_t(\R;\rH^{\alpha-1}_L)}\lesssim \|B'\|_{\rL^1} \|f\|_{\rH_L^\alpha},\\
		&\|\big(D_t^2-P(t)\big)S(t,s)f\|_{\rL^1_t(\R;\rH^{\alpha}_L)}\lesssim \|B'\|_{\rL^1} \|f\|_{\rH_L^\alpha}.
	\end{align*}
	\end{theorem}
	\begin{proof} Once again, the properties for $C(t,s)$ follow from Theorem \ref{thm:ParametrixProperty} and the identity \eqref{eq:RelationBetweenParametricesI}. To prove the assertions for $S(t,s)$, we once again use $S(t,s)=\ii\langle D_L\rangle^{-1}\tilde{T}(t,s)$ as stated in Lemma~\ref{lemma:RelationBetweenParametrices}~(b)~(iv). Using (b)~(v) of the very same lemma in place of Lemma~\ref{lemma:TimeDerivativeFrequencyLocalized}, we find that Theorem~\ref{thm:ParametrixProperty} holds true with $\tilde{T}(t,s)$ in place of $T^\pm(t,s)$. Thus, the assertions for $S(t,s)$ follow again from the fact that $\langle D_L\rangle^{-1}\colon \rH_L^\alpha \to \rH_L^{\alpha+1}$. 
	\end{proof}

	\section{Existence of Weak Solutions in $\rH^\alpha_L$}\label{section:ExistenceOfWeakSolutions}
	In this section, we aim to establish the existence of weak solution to \eqref{eq:WaveEquationWithTimeDependentLipschitzMetrics} in $\rH^\alpha_L$. We prove existence
	following the approach in \cite{SmithParametrix1998}, which was subsequently also used by Hassell--Rozendaal in their $\rL^p$-theory for rough wave equations \cite{HassellRozendaal2022}. To motivate the idea, we first recall the easier time-independent case $B(t)=\id$ $(t\in \R)$. Then $P(t)=L$ and since $\ii\ssqrt{L}$ generates a $C_0$-group $(\ee^{\ii t\ssqrt{L}})_{t\in \R}$ on $\rL^2$ (by the Borel functional calculus for self-adjoint operators), we have the representation (cf. \cite[Corollary~3.14.8]{ArendtBattyHieberNeubrander2011} in the case $F=0$)
	\begin{equation}\label{eq:RepresentationFormulaInTheAutonomousCase}
	u(t)=\mathrm{Cos}(t\ssqrt{L})g+\ii t\,\mathrm{Sinc}(t\ssqrt{L})h+\int_0^t 
	\ii (t-s)\mathrm{Sinc}((t-s)\ssqrt{L})F(s) \dd s,
	\end{equation}
	where 
	\begin{equation*}
	\mathrm{Cos}(t\ssqrt{L})\coloneqq \frac{1}{2} \big(\ee^{\ii t\ssqrt{L}}+\ee^{-\ii t\ssqrt{L}}\big),\quad t\, \mathrm{Sinc}(t\ssqrt{L})\coloneqq \int_0^t \mathrm{Cos}(s\ssqrt{L})\dd s.
	\end{equation*}
	Now, the idea in the time-dependent case is to replace the \textit{exact} solution operators $\mathrm{Cos}((t-s)\sqrt{L})$ and $\ii (t-s)\, \mathrm{Sinc}((t-s)\sqrt{L})$ in \eqref{eq:RepresentationFormulaInTheAutonomousCase} by the \textit{parametrices} $C(t,s)$ and $S(t,s)$ constructed in Section~\ref{section:ParametrixConstruction}.  
	So we make the ansatz
	\begin{equation*}
	u(t)=C(t,0)g+S(t,0)h+\int_0^t S(t,s)G(s)\dd s
	\end{equation*} 
	with a suitable function $G\in \rL^1(\R;\rH_L^{\alpha-1})$. This function will turn out to be the solution to a Volterra equation. Recall that we know from Theorem~\ref{thm:ParametrixPropertyOfSineAndCosineFuntions} that
	\begin{align*}
	& R_1(t,s)\coloneqq (D_t^2-P(t))C(t,s)\colon \rH^\alpha_L\to \rH^{\alpha-1}_L,\\
	& R_2(t,s)\coloneqq (D_t^2-P(t))S(t,s)\colon \rH^{\alpha-1}_L\to \rH^{\alpha-1}_L
	\end{align*}
	boundedly.
	\begin{lemma}[Volterra Operator]\label{lemma:Volterra} Let $\alpha\in \R$ and $G\in \rL^1(\R; \rH_L^{\alpha-1})$. Put
	\begin{equation*}
		V(t)\coloneqq \int_0^t S(t,s)G(s)\dd s\quad (t\in \R).
	\end{equation*}	
	Then, $V\in C(\R;\rH^{\alpha}_L)\cap C^1(\R;\rH^{\alpha-1}_L)\cap W^{2,1}_\loc(\R;\rH^{\alpha-2}_L)$. Moreover, 
	\begin{equation*}
		(D_t^2-P(t))V(t)=\tfrac{1}{\ii}[(\mathbf{Id}+\ii \mathbf{R})G](t)
	\end{equation*}
	for a.e. $t\in \R$,
	where
	\begin{equation*}
		\mathbf{R}\colon \rL^1(\R;\rH_L^{\alpha-1})\to \rL^1(\R;\rH_L^{\alpha-1}),\quad (\mathbf{R}G)(t)=\int_0^t R_2(t,s)G(s)\dd s
	\end{equation*}
	has operator norm $\|\mathbf{R}\|\leq C \|B'\|_1$ with a constant $C$ depending only on $m_1,m_2$ and $m_4$.
	\end{lemma}
	\begin{proof} Let $\alpha\in \R$ and $G\in \rL^1(\R;\rH_L^{\alpha-1})$. We write
	\begin{equation*}
		V(t)= \int_0^t v(t,s) \dd s \quad (t\in \R),
	\end{equation*}	
	where we have set $v(t,s)\coloneqq S(t,s)G(s)$ for $t\in \R$ and a.e. $s\in \R$. Then,
	by Theorem~\ref{thm:GeneralizedCosineAndSineFuntions}, we have for a.e. $s\in \R$ 
	\begin{align}\label{eq:PreparationVolterraOperator}
		&v(\cdot,s)\in  C(\R;\rH^{\alpha}_L)\cap C^1(\R;\rH^{\alpha-1}_L)\cap  W_\loc^{2,\infty}(\R; \rH^{\alpha-k}_L) \quad\text{with } \notag\\
		&\|D_t^k v(t,s)\|_{\alpha-k}\lesssim \langle t-s\rangle^{|k-1|}  \|G(s)\|_{\alpha-1}, \quad k\in \{0,1,2\}.
	\end{align} 
	\textbf{First Step}: $\mathbf{V\in C(\R;\rH^{\alpha}_L)}$\\[0.2cm]
	Fix $t_0\in \R$ and let $t\in \R$. We have to show $V(t)\to V(t_0)$ in $\rH^{\alpha}_L$ as $t\to t_0$. We consider the limit $t\downarrow t_0$ (the limit $t\uparrow t_0$ can be treated similarly). Observe that for $t\geq t_0$,
	\begin{equation}\label{eq:ContinuityOfV}
		\|V(t)-V(t_0)\|_{\alpha}\leq{}\int_{t_0}^{t}\|v(t,s)\|_{\alpha}\dd s+\int_0^{t_0}\|v(t,s)-v(t_0,s)\|_{\alpha}\dd s.
	\end{equation}
	Now for the first integral on the right-hand side of \eqref{eq:ContinuityOfV}, we have
	\begin{equation*}
		\int_{t_0}^{t}\|v(t,s)\|_{\alpha}\dd s  \lesssim\langle t-t_0\rangle \int_{t_0}^{t}\|G(s)\|_{\alpha-1}\dd s \rightarrow 0\quad (t\downarrow t_0), 
	\end{equation*}
	where we used \eqref{eq:PreparationVolterraOperator} with $k=0$ and then dominated convergence. To estimate the second integral on the right-hand side of \eqref{eq:ContinuityOfV}, note that once again by \eqref{eq:PreparationVolterraOperator} (with $k=0$), the integrand converges to zero as $t\downarrow t_0$ for a.e. $s$ and is controlled by $\langle t_0-s\rangle \|G(s)\|_{\alpha-1}\lesssim \langle t_0\rangle \|G(s)\|_{\alpha-1}$ if $|t-t_0|\leq 1$. Thus, we conclude once again that
	\begin{equation*}
		\int_0^{t_0}\|v(t,s)-v(t_0,s)\|_{\alpha}\dd s\rightarrow 0 \quad (t\downarrow t_0)
	\end{equation*}
	by dominated convergence. This shows that $V\in C(\R;\rH^\alpha_L)$.
	\\[0.2cm]
	\textbf{Second Step}: $\mathbf{V\in C^1(\R;H^{\alpha-1}_L)}$\\[0.2cm]
	Let $t\in \R$. Then, for $h\neq 0$,
	\begin{align*}
		\frac{V(t+h)-V(t)}{\ii h}={}&\int_0^{t+h}\frac{v(t+h,s)-v(t,s)}{\ii h}\dd s+\frac{1}{\ii h}\int_t^{t+h} v(t,s)\dd s\\
		\eqqcolon{}& \Delta_{1,h}+\Delta_{2,h}.
	\end{align*}
	Now, it follows again from \eqref{eq:PreparationVolterraOperator} with $k=1$ and dominated convergence that
	\begin{equation*}
		\Delta_{1,h}\rightarrow \int_0^t D_t v(t,s)\dd s \quad (h\to 0)\quad \text{in }\rH^{\alpha-1}_L.
	\end{equation*}
	Turning to $\Delta_{2,h}$, we have to be a bit more careful since $G$ need not be continuous and thus we cannot use the fundamental theorem. But using $S(t,t)=0$ and Theorem~\ref{thm:GeneralizedCosineAndSineFuntions}~(b), we have 
	\begin{align*}
		\|\Delta_{2,h}\|_{\alpha-1}={}&\bigg\|\int_{\R} \mathbbm{1}_{(t,t+h)}(s) \bigg(\frac{s-t}{h}\bigg) \frac{S(t,s)-S(t,t)}{s-t} G(s) \dd s\bigg\|_{\alpha-1}\\
		\leq{}& \int_{\R} \mathbbm{1}_{(t,t+h)}(s) \int_0^1 \|\partial_s S(t,t+(s-t)\tau)G(s) \|_{\alpha-1} \dd \tau \dd s\\
		\lesssim{}& \int_{\R} \mathbbm{1}_{(t,t+h)}(s) \|G(s)\|_{\alpha-1} \dd s
		\rightarrow{} 0 \quad (h\downarrow 0).
	\end{align*}
	One argues similarly for $h\uparrow 0 $. This proves that $V\in C^1(\R;\rH^{\alpha-1}_L)$ with $D_t V(t)=\int_0^t D_t v(t,s)\dd s$.
	\\[0.2cm]
	\textbf{Third Step}: $\mathbf{D_t V\in W_\loc^{1,1}(\R;H_L^{\alpha-2})}$\\[0.2cm]
	Let $J\coloneqq (-T,T)$, $T>0$. We have to show that $D_t V\in W^{1,1}(J;\rH_L^{\alpha-2})$. Exactly as in the second step, one writes for $h\neq 0$,
	\begin{align*}
		&\frac{D_t V(t+h)-D_t V(t)}{\ii h}\\={}&\int_0^{t+h}\frac{D_t v(t+h,s)-D_t v(t,s)}{\ii h}\dd s+\frac{1}{\ii h}\int_t^{t+h} D_t v(t,s)\dd s\\
		\eqqcolon{}& \Delta_{1,h}(t)+\Delta_{2,h}(t).
	\end{align*}
	Using \eqref{eq:PreparationVolterraOperator} with $k=2$ and dominated convergence, we deduce $\Delta_{1,h}\rightarrow \int_0^t  D_t^2 v(t,s)\dd s$ in $\rL^\infty(\overline{J};\rH^{\alpha-2}_L)\hookrightarrow \rL^1(J;\rH^{\alpha-2}_L)$ as $h\rightarrow 0$. Turning to $\Delta_{2,h}(t)$, we further split for $h>0$
	\begin{align*}
		&\|\Delta_{2,h}(t)-\tfrac{1}{i}G(t)\|_{\alpha-2}\\
		\leq{}&\frac{1}{h}\int_t^{t+h} \|D_t S(t,s)G(s)-G(s)\|_{\alpha-2}\dd s+\frac{1}{h}\int_t^{t+h} \|G(s)-G(t)\|_{\alpha-2}\dd s\\
		\eqqcolon{}& \Delta^{1}_{2,h}(t)+\Delta^{2}_{2,h}(t),
	\end{align*}
	so that
	\begin{equation*}
		\|\Delta_{2,h}(t)-\tfrac{1}{\ii }G(t)\|_{L^1_t(J; \rH_L^{\alpha-2}(\R^d))}\leq \int_J \Delta^1_{2,h}(t)\dd t +\int_J \Delta^2_{2,h}(t)\dd t.
	\end{equation*}
	Arguing just as in the proof of \cite[Proposition 1.4.29]{CazenaveHaraux1998}), we find $\int_J \Delta_{2,h}^2(t)\dd t\rightarrow 0$ as $h\downarrow 0$. On the other hand, using \eqref{eq:InitialConditionsSineFunction} and \eqref{eq:PreparationVolterraOperator}, we estimate
	\begin{align*}
		\|D_t S(t,s)G(s)-G(s)\|_{\alpha-2}={}&\|D_t S(t,s)G(s)-D_tS(s,s)G(s)\|_{\alpha-2}\\
		\leq{}& \int_0^1 \|D_t^2S(s+(t-s)\tau,s)G(s)\|_{\alpha-2}\dd \tau \,|t-s|\\
		\lesssim{}& \|G(s)\|_{\alpha-1}|t-s|,
	\end{align*}
	which implies
	\begin{align*}
		\int_J \Delta^1_{2,h}(t)\dd t \leq{}& \int_J \frac{1}{h}\int_{t}^{t+h}  \|G(s)\|_{\alpha-1}|t-s|\dd s\dd t\\
		\leq{}& \int_{-T}^{T+h} \bigg(\frac{1}{h}\int_{s-h}^{s}  |t-s|\dd t\bigg) \|G(s)\|_{\alpha-1}\dd s\leq  \|G\|_{\rL^1(\R;\rH^{\alpha-1}_L)} h \to 0
	\end{align*}
	for $h\downarrow 0$. One argues similarly for the limit $h\uparrow 0$. Thus, we have shown that
		\begin{equation*}
			\frac{D_t V(t+h)-D_t V(t)}{\ii h}\rightarrow  \int_0^tD_t^2v(t,s) \dd s+\tfrac{1}{i}G(t)\quad (h\to 0)
		\end{equation*}
		in $\rL^1(J;\rH_L^{\alpha-2})$. It now follows (see e.g. \cite{CazenaveHaraux1998}) that $D_tV\in W^{1,1}(J;\rH_L^{\alpha-2})$ with
		\begin{equation}\label{eq:WeakDerivative}
			D_t^2 V(t)=\int_0^t D_t^2 v(t,s) \dd s+\tfrac{1}{i}G(t) \quad \text{in }\rH^{\alpha-2}_L
		\end{equation} 
		for a.e. $t\in \R$.
		\\[0.2cm]
		\textbf{Fourth Step}: \textbf{Conclusion}\\[0.2cm]
		By Theorem~\ref{thm:GeneralizedCosineAndSineFuntions} and Lemma~\ref{lemma:ExtensionOfL_jToHomogeneousSobolevSpaces}, we have 
		\begin{equation*}
			P(t)V(t)=\int_0^t P(t)v(t,s) \dd s \quad \text{in }\rH^{\alpha-1}_L.
		\end{equation*}
		Recalling that $(D_t^2-P(t))v(t,s)=R_2(t,s)G(s)$, it follows that
		\begin{align*}
			(D_t^2-P(t))V(t)={}&\tfrac{1}{i}G(t)+\int_0^t R_2(t,s)G(s)\dd s\\
			={}&\tfrac{1}{i}[(\mathbf{Id}+i\mathbf{R})G](t).
		\end{align*}
		Now by Theorem \ref{thm:ParametrixPropertyOfSineAndCosineFuntions}, we have for a.e. $t\in \R$
		\begin{align*}
			&\|(\mathbf{R}G)(t)\|_{\rH^{\alpha-1}_L}\\\leq{}& \int_0^t \|R_2(t,s)G(s)\|_{\rH^{\alpha-1}_L}\dd s\\
			\lesssim{}& \int_0^t \|B'(t)\| \|G(s)\|_{\rH^{\alpha-1}_L} \dd s \leq\|B'(t)\|\cdot \|G\|_{\rL^1(\R;\rH_L^{\alpha-1})},
		\end{align*}
		which yields
		\begin{equation*}
			\|\mathbf{R}G\|_{\rL^1(\R;\rH^{\alpha-1}_L)}\lesssim \|B'\|_{\rL^1}\cdot \|G\|_{\rL^1(\R;\rH_L^{\alpha-1})}
		\end{equation*}
		as desired. Finally, tracing back all constants arising in the preceeding proofs, one checks that the implicit constant only depends on $m_4$ and on the $\rL^2$-bounds for the Phillips functional calculus for $\sqrL$, the latter depending in turn on $M_2= \sup_{y\in \R^d}\|\ee^{-\ii y\cdot \sqrL}\|_{\cL(\rL^2)}$. But $M\simeq_{m_1,m_2} 1$ by the self-adjointness of $L_1,\dots, L_d$ with respect to \eqref{eq:EquivalentScalarProduct}.  
	\end{proof}
	We are now ready to prove the existence of a weak solution to \eqref{eq:WaveEquationWithTimeDependentLipschitzMetrics} in the $\rH^\alpha_L$-scale\footnote{Here, a weak solution in $\rH^\alpha_L$ is just defined as in Definition~\ref{def:WeakSolutions} with $\rH^\alpha_L$ in place of $\rH^\alpha$.}. 
	
	\begin{theorem}[Well-posedness in $\rH_L^\alpha$]\label{thm:WellposednessInAdaptedSpaces} 
	Let $\eps_1$ be the constant from \eqref{eq:SizeCondition_b} and suppose that $\eps_1\in (0,\frac{1}{C})$, 
	where $C$ is the constant from Lemma~\ref{lemma:Volterra}. Let $\alpha\in\R $ and suppose that $g\in \rH_L^{\alpha}$, $h\in \rH_L^{\alpha-1}$, and $F\in \rL^1(\R;\rH_L^{\alpha-1})$. Then, there exists a unique weak solution $u$ to \eqref{eq:WaveEquationWithTimeDependentLipschitzMetrics} in $\rH^\alpha_L$. Moreover, there exists $G\in \rL^1(\R;\rH_L^{\alpha-1})$ such that
		\begin{equation}\label{eq:RepresentationOfSolution}
			u(t)=C(t,0)g+S(t,0)h+\int_0^t S(t,s)G(s)\dd s\quad (t\in \R),
		\end{equation}
		and we have the estimate
		\begin{equation*}
			\|G\|_{\rL^1(\R;\rH^{\alpha-1}_L)}\lesssim \|g\|_{\rH_L^{\alpha}}+\|h\|_{\rH^{\alpha-1}_L}+\|F\|_{\rL^1(\R;\rH_L^{\alpha-1})}.
		\end{equation*}
	Furthermore, for each bounded interval $J\subseteq \R$ with $0\in J$, the map
	\begin{equation}\label{eq:ContinuousDependence}
		\Phi\colon \rH_L^\alpha \times \rH^{\alpha-1}_L\times \rL^1(J;\rH_L^{\alpha-1})\to C(\overline{J}; \rH_L^{\alpha}),\quad (g,h,F)\mapsto u
	\end{equation}
	is Lipschitz continuous.
	\end{theorem}
	\begin{proof} Let $\alpha\in \R$ and suppose that $g\in \rH_L^{\alpha}$, $h\in \rH_L^{\alpha-1}$, and $F\in \rL^1(\R;\rH_L^{\alpha-1})$. Uniqueness of weak solutions and the continuous dependence of the initial data and the driving force as in \eqref{eq:ContinuousDependence} is proved in Corollary~\ref{cor:UniquenessOfWeakSolutions} below. To prove existence, define for a given function $G\in \rL^1(\R;\rH_L^{\alpha-1})$ the function $u$ by the right-hand side of \eqref{eq:RepresentationOfSolution}. Then, by Theorem~\ref{thm:GeneralizedCosineAndSineFuntions} and Lemma~\ref{lemma:Volterra}, $u\in C(\R; \rH_L^{\alpha})\cap C^1(\R; \rH_L^{\alpha-1})\cap W^{2,1}_\loc(\R;\rH_L^{\alpha-2})$ with $u(0)=g$ and $D_tu(0)=h$. Now Theorem~\ref{thm:ParametrixPropertyOfSineAndCosineFuntions} and Lemma~\ref{lemma:Volterra} imply that $u$ is a weak solution to \eqref{eq:WaveEquationWithTimeDependentLipschitzMetrics} in  $\rH^\alpha_L$  if and only if for a.e. $t\in \R$
		\begin{equation*}
			F(t)=((D_t^2-P(t))u(t) = R_1(t,0)g+R_2(t,0)h+\tfrac{1}{\ii}\big[(\mathbf{Id}+\ii\mathbf{R})G\big](t),
		\end{equation*}
		which is equivalent to
		\begin{align*}
			\tfrac{1}{\ii}\big(\mathbf{Id}+\ii\mathbf{R}\big)G(t)=F(t)-R_1(t,0)g-R_2(t,0)h.
		\end{align*}
		Since $\|\mathbf{R}\|_{\cL(\rL^1(\R;\rH^{\alpha-1}_L))}\leq C \|B'\|_{\rL^1}\leq C \eps_1<1$ by Lemma \ref{lemma:Volterra} and the assumption, the operator $\mathbf{Id}+\ii\mathbf{R}$ is invertible. On the other hand, $R_1(t,0)g$ and $R_2(t,0)h$ belong to $\rL_t^1(\R;\rH^{\alpha-1}_L)$ with 
		\begin{align*}
			&\|R_1(t,0)g\|_{\rL_t^1(\R;\rH^{\alpha-1}_L)}\lesssim \|B'\|_{\rL^1} \|g\|_{\rH_L^\alpha},\\&  \|R_2(t,0)h\|_{\rL_t^1(\R;\rH^{\alpha-1}_L)}\lesssim \|B'\|_{\rL^1} \|h\|_{\rH_L^{\alpha-1}}
		\end{align*}
		by Theorem \ref{thm:ParametrixPropertyOfSineAndCosineFuntions}. Therefore, if we choose
		\begin{equation*}
			G=\ii\big(\mathbf{Id}+\ii\mathbf{R}\big)^{-1} \big(F-R_1(t,0)g-R_2(t,0)h\big)\in \rL_t^1(\R;\rH^{\alpha-1}_L),
		\end{equation*}
		it follows that $u$ given by \eqref{eq:RepresentationOfSolution} indeed defines a weak solution to \eqref{eq:WaveEquationWithTimeDependentLipschitzMetrics}. Moreover, we have the estimate
		\begin{align*}
			&\|G\|_{\rL_t^1(\R;\rH^{\alpha-1}_L)}\\\lesssim{}& \|R_1(t,0)g\|_{\rL_t^1(\R;\rH^{\alpha-1}_L)}+\|R_2(t,0)h\|_{\rL_t^1(\R;\rH^{\alpha-1}_L)}+\|F\|_{\rL_t^1(\R;\rH_L^{\alpha-1})}\\
			\lesssim{}& \|g\|_{\rH_L^{\alpha}}+\|h\|_{\rH^{\alpha-1}_L}+\|F\|_{\rL^1(\R;\rH_L^{\alpha-1})}.
		\end{align*}
		as desired. The proof is complete.
	\end{proof}
	By Proposition~\ref{prop:Kato}, we can identify the spaces $\rH^\alpha$ and $\rH_L^\alpha$ for $\alpha\in [-2,2]$, so Theorem~\ref{thm:WellposednessInAdaptedSpaces} immediately gives the proof of Theorem~\ref{thm:wellposedness}.
	
	\begin{corollary}[Existence and Uniqueness of Weak Solutions in $\rH^\alpha$]\label{cor:WellposednessInH} Let $\eps_1>0$ be as in Theorem~\ref{thm:WellposednessInAdaptedSpaces} and suppose additionally $-1\leq \alpha \leq 2$. Assume that $g\in \rH^{\alpha}$, $h\in \rH^{\alpha-1}$, and $F\in \rL^1(\R;\rH^{\alpha-1})$. Then, the function $u$ as in Theorem~\ref{thm:WellposednessInAdaptedSpaces} defines the unique weak solution to \eqref{eq:WaveEquationWithTimeDependentLipschitzMetrics} in $\rH^\alpha$.
	\end{corollary}
	
	\begin{remark}\label{rem:RelaxAssForWellposedness} Theorem~\ref{thm:WellposednessInAdaptedSpaces} and hence Corollary~\ref{cor:WellposednessInH} are also true under milder conditions on the coefficients $b_1,\dots , b_d$. First, observe that up to now, we did not use the smallness assumption of $\eps_0$ as in \eqref{eq:Ellipticity_b}. Thus, it would be enough if the coefficients $b_j$ are bounded from above and below by a positive constant, just as the coefficients $a_j$ are. Second, the smallness of $\eps_1>0$ was only needed in the proof of Theorem~\ref{thm:WellposednessInAdaptedSpaces} in order to invert the operator $\mathbf{Id}+\ii \mathbf{R}$ on $\rL^1(\R;\rH_L^{\alpha-1})$ (using a Neumann-series argument). However, \textit{without} assuming the smallness of $\eps_1$, we can invert this operator \textit{locally} in time, i.e., on the space $\rL^1(J;\rH_L^{\alpha-1})$, where $J\coloneqq (-T,T)$ for some fixed $T\in (0,\infty)$. Indeed, a straightforward induction on $k\in \N_0$ yields
		\begin{equation}
			\|\mathbf{R}^k G\|_{\rL^1(J;\rH_L^{\alpha-1})}\lesssim \frac{(C_J\,T)^k}{k!}\|G\|_{\rL^1(J;\rH_L^{\alpha-1})}\quad (k\in \N_0)
		\end{equation}
		with
		\begin{equation*}
			C_J\coloneqq \sup_{t,s\in J} \|R_2(t,s)\|_{\cL(\rH^{\alpha-1}_L)}\lesssim  \sup_{t,s\in J} \|B'(t)\| \leq m_4,
		\end{equation*}
		where we used Theorem~\ref{thm:ParametrixPropertyOfSineAndCosineFuntions}. Thus, the operator $A\coloneqq \sum_{k=0}^\infty (-1)^k\ii^k\mathbf{R}^k$ converges in $\cL(\rL^1(J;\rH^{\alpha-1}_L))$ (with operator norm bounded by a constant times $\ee^{m_4T}$) and defines the inverse of $\mathbf{Id}+\ii\mathbf{R}$. Following the proof of Theorem~\ref{thm:WellposednessInAdaptedSpaces}, we obtain a unique weak solution on $J$, i.e., a unique function $u_J\in  C(J; \rH_L^{\alpha})\cap C^1(J; \rH_L^{\alpha-1})\cap W^{2,1}(J;\rH_L^{\alpha-2})$ with $u(0)=g$, $D_tu(0)=h$ and $D_t^2u=P(t)u(t)+F(t)$ for a.e. $t\in J$. Exhausting $\R$ by bounded open intervals $J$ then gives the claim. Observe that this gives only $G\in \rL^1_\loc(\R;\rH_L^{\alpha-1})$ in the representation formula for $u$.
		 However, the smallness of $\eps_1$ ensures that $G\in \rL^1(\R;\rH^{\alpha-1}_L)$, which in turn means that the energy $\|u(t)\|_{\rH^1}$ stays globally bounded with respect to $t$. In general, this is nontrivial (see e.g. \cite{TataruGlobal}). However, it is a minimum requirement in order to prove global-in-time Strichartz estimates, since the energy estimates are a special case of the latter.
	\end{remark}

\section{Uniqueness of Weak Solutions}\label{section:UniquenessOfWeakSolutions}
	The classical energy inequalities imply the uniqueness of weak solutions to linear wave equations with $C^\infty_b(\R^d)$-coefficients in the scale of  standard Sobolev spaces $\rH^\alpha(\R^d)$, $\alpha\in \R$. 
 As the coefficients of our wave operator $\Box_P$ possess only limited regularity, these energy inequalities are applicable only in a restricted range of exponents $\alpha$. However, combining them with Proposition~\ref{prop:Kato} gives a first result.
	\begin{lemma}[Classical Energy Inequalities]\label{lemma:ClassicalEnergyInequalities} 
		Let $J\subseteq \R$ be a bounded open interval with $0\in J$. If
		\begin{equation*}
			u\in C(\overline{J};\rH^1_L)\cap C^1(\overline{J};\rL^2)\cap W^{2,1}(J;\rH^{-1}_L) 
		\end{equation*} 
		with $\Box_Pu= (D_t^2-P(t))u\in \rL^1(J;\rL^2)$, then
		\begin{equation}\label{eq:EnergyInequalityBabyVersion} 
			\|u\|_{C(\overline{J};\rH^1_L)}\lesssim_{|J|} \|u(0)\|_{\rH_L^1}+\|D_tu(0)\|_{\rL^2}+\|\Box_Pu \|_{\rL^1(J;\rL^2)}.
		\end{equation}
	\end{lemma}
	\begin{proof} We first observe that \eqref{eq:EnergyInequalityBabyVersion} holds with $\rH_L^1$ replaced by the classical Sobolev space $\rH^1$. Indeed, a careful inspection of the proof of \cite[Chapter~I, Theorem~3.1]{Sogge1995} shows that the latter is applicable with $s=0$ in the assumptions of that theorem (the coefficients $c_j(t,x)=b_j(t)a_j(x_j)$ of $P(t)$ are Lipschitz, so taking one time derivative and integrating by parts once with respect to\@ $x$ is justified to make the energy inequality work). But $\rH^{\pm 1}_L=\rH^{\pm 1}$ by Proposition~\ref{prop:Kato}, so the claim follows.
	\end{proof}
	Note that Lemma~\ref{lemma:ClassicalEnergyInequalities} in particular implies the uniqueness of weak solutions in $\rH^1_L$. In the following, we lift this result to general exponents $\alpha\in \R$. The reason why we can do this is that, while $\langle D_x\rangle^\alpha$ does not commute with $P(t)=\sum_{j=1}^d b_j(t)L_j$ for large $\alpha$ (by the limited regularity of the coefficients $a_j$), the fractional power $\langle D_L\rangle^\alpha$ in fact \textit{does} (on appropriate subsets of functions). In order to exploit this fact, we need the following approximation lemma. To state it, we introduce some notation first. Let $\alpha\in \R$, $p\in [1,\infty]$, $m\in \N_0$, and $J\subseteq \R$ be a bounded open interval. Put
	\begin{equation*}
		\mathcal{W}^{m,p}_{J,\alpha}\coloneqq \begin{cases}
			W^{m,p}(J;\rH^\alpha_L) & \text{for }p\in [1,\infty),\\
			C^m(\overline{J};\rH^\alpha_L)&\text{for }p=\infty.
		\end{cases}
	\end{equation*}
	Let $u\in \mathcal{W}^{m,p}_{J,\alpha}$ and $n\in \N_0$. Then, we set $\Z_n\coloneqq \{-n,\dots, n\}$ and define
	\begin{equation*}
		(\Pi_n u)(t)\coloneqq \sum_{\lambda\in 2^{\Z_n}} \psi_\lambda(\sqrL)u(t)\quad \text{for all } t\in \overline{J}
	\end{equation*}
	if $p=\infty$. If $p\in [1,\infty)$, we define $\Pi_n u$ by the same expression in an almost everywhere sense. If $I\Subset J$ is an open subinterval and $ 4^{-n}<  \mathrm{dist}(I, J^c)$, we define
	\begin{equation*}
		(S_n u)(t)\coloneqq (\varphi_n \ast u)(t)=\int_{\R}\varphi_n(t-s)u(s)\dd s\quad \text{for } t\in \overline{I},
	\end{equation*}
	where $\varphi_n(t)\coloneqq 4^n\varphi(4^nt)$ and $\varphi\in C_c^\infty(\R)$ satisfies $\supp{\varphi}\subseteq (-1,1)$ and $\int_\R \varphi(t)\dd t=1$.
	\begin{lemma}[Approximation Lemma]\label{lemma:Approximation} Let $I\Subset J$ be bounded open intervals in $\R$ and $m,\ell\in \N_0,p\in[1,\infty],\alpha,\gamma\in \R$.
		\begin{enumerate}[label=(\alph*)]
			\item For all $n\in\N_0$, we have $\Pi_n\in \cL(\mathcal{W}^{m,p}_{J,\alpha}, \mathcal{W}^{m,p}_{J,\alpha+\gamma})$ with $\|\Pi_n\|\lesssim_{\gamma} 2^{(n+1)\gamma_+}$. Moreover, $\Pi_n \to \id$ strongly on $\mathcal{W}^{m,p}_{J,\alpha}$ as $n\to \infty$.
			\item For all $n \geq n_0$, we have $S_n\in \cL(\mathcal{W}^{m,p}_{J,\alpha},\mathcal{W}^{m+\ell,p}_{I,\alpha})$ with $\|S_n \|\lesssim 4^{n\ell_+}$, where $n_0$ is such that $4^{-n_0}<\mathrm{dist}(I,J^c)$. Moreover, if $u\in \mathcal{W}^{m,p}_{J,\alpha}$ , then $(S_nu)^{(k)}=S_nu^{(k)}$ for $k\leq m$ and $S_n u\to u|_{I}$ in $\mathcal{W}^{m,p}_{I,\alpha}$ as $n\to \infty$.  
		\end{enumerate}
		\begin{proof}
			(a) Suppose first that $p\in [1,\infty)$. If $\gamma\leq 0$, then $\|\Pi_n\|\leq 1$ since $\rH^{\alpha}_L\hookrightarrow \rH^{\alpha+\gamma}_{L}$ in this case. Therefore, we may suppose $\gamma>0$. By  Lemma~\ref{lemma:ExtensionOfPhillipsFunctionalCalculus}~(a)   and Proposition~\ref{prop:PropertiesPhillipsCalculus}~(c), we then have 
			\begin{equation*}
				\|\Pi_n u(t)\|_{\rH^{\alpha+\gamma}_L}\lesssim \sum_{\lambda\in 2^{\Z_n}}\langle \lambda\rangle^{\gamma} \|u(t)\|_{\rH^\alpha_L} \lesssim_{\gamma} 2^{(n+1)\gamma}\|u(t)\|_{\rH^\alpha_L}
			\end{equation*}
			for a.e. $t\in J$. Raising this inequality to the $p$-th power and integrating it with respect to $t\in J$, we deduce the bound $\|\Pi_n\|\lesssim 2^{(n+1)\gamma}$. This proves the first claim. To prove the second claim, note that we have proved that $(\Pi_n)_n$ is in particular uniformly bounded on $\mathcal{W}_{J,\alpha}^{m,p}$. Let $v\in D\coloneqq C^\infty(\overline{J}; \rH^\alpha_L)$. Then, $\Pi_n v\to v$ in $\mathcal{W}_{J,\alpha}^{m,p}$ by  Proposition~\ref{prop:CalderonReproducingFormulaOnInhomogeneousSobolevSpaces}  and the uniform continuity of $v$. Now, since $D$ is dense in $\mathcal{W}_{J,\alpha}^{m,p}$ (see e.g. \cite[Corollary 1.4.37]{CazenaveHaraux1998}), the strong convergence of $(\Pi_n)_n$ to $\id$ on $\mathcal{W}^{m,p}_{J,\alpha}$ follows from a standard density argument using uniform boundedness. The case $p=\infty$ is proved similarly.
			
			(b) These assertions are proved exactly as in the case of scalar-valued functions (see e.g. \cite[Section 1.3]{Hoermander1990}).
		\end{proof}
	\end{lemma}
	
	\begin{proposition}[Energy Estimates in $\rH^{\alpha}_L$]\label{prop:EnergyEstimate} 
		Let $J\subseteq \R$ be a bounded open interval with $0\in J$. Suppose further that for some $\alpha\in\R$, the function $u$ belongs to $ C(\overline{J};\rH_L^\alpha)\cap C^1(\overline{J};\rH^{\alpha-1}_L)\cap W^{2,1}(J;\rH^{\alpha-2}_L)$. If $\Box_Pu\in \rL^1(J;\rH^{\alpha-1}_L)$, then
		\begin{equation*}
			\|u\|_{C(\overline{J}; \rH_L^\alpha)}\lesssim_{|J|}  \|u(0)\|_{\rH_L^\alpha}+\|D_tu(0)\|_{\rH_L^{\alpha-1}}+\|\Box_P u\|_{\rL^1(J;\rH_L^{\alpha-1})}.
		\end{equation*}
	\end{proposition}
	\begin{proof} Let $J\subseteq \R, \alpha\in \R$ and $u$ be as in the statement of the proposition. Fix some open interval $I$ with $0\in \overline{I}\subseteq J$. Now, for $n\in \N$ with $4^n\geq (\mathrm{dist}(I,J^c))^{-1}$, we define $u_n\colon \overline{I}\to \rH_L^{\alpha}$ by
		\begin{equation*}
			u_n(t)\coloneqq (S_n \Pi_n u)(t)=\int_{\R}\varphi_n(t-s)\Pi_n u(s) \dd s \quad (t\in \overline{I}).
		\end{equation*}
		Observe that $u_n$ belongs to $C^\infty(\overline{I};\rH^\infty_L)$. Indeed, since $u\in C(\overline{J};\rH^\alpha_L)$, it follows that $\Pi_n u\in C(\overline{J};\rH^\beta_L)$ for all $\beta\in \R$ by Lemma~\ref{lemma:Approximation}~(a) and therefore $u_n\in C^k(\overline{I};\rH^\beta_L)$ by Lemma~\ref{lemma:Approximation}~(b) for all $\beta\in \R$ and $k\in \N_0$. This proves $u_n\in C^\infty(\overline{I};\rH^\infty_L)$. But then also $v_n\coloneqq \langle D_L\rangle^{\alpha-1}u_n\in C^\infty(\overline{I};\rH^\infty_L)$ and $\Box_P v_n\in C^{0,1}(\overline{I}; \rH^\infty_L)\subseteq \rL^1(I;\rL^2)$ (note that the loss of regularity in $t$ arises from the fact that the $b_j$ are only Lipschitz). Applying Lemma~\ref{lemma:ClassicalEnergyInequalities} to $v_n$, we obtain
		\begin{align}\label{eq:EnergyEstimateOnV_n}
			\begin{split}
				\|u_n\|_{C(\overline{I};\rH_L^\alpha)}&=\hphantom{_{|I|}}{}\|v_n\|_{C(\overline{I};\rH_\rL^1)}\\
				&\lesssim_{|I|} \|v_n(0)\|_{\rH^1_L}+\|D_t v_n(0)\|_{\rL^2}+\| \Box_P v_n\|_{\rL^1(I;\rL^2)}\\
				&=\hphantom{_K}{}\|u_n(0)\|_{\rH^\alpha_L}+\|D_t u_n(0)\|_{\rH^{\alpha-1}_L}+\|\Box_Pu_n\|_{\rL^1(I;\rH^{\alpha-1}_L)},
			\end{split}
		\end{align}
		where, for the last equality, we used that $\Box_P$ and $\langle D_L\rangle^{\alpha-1}$ commute when applied to $u_n$. Now, we want to pass to the limit $n\to \infty$ in \eqref{eq:EnergyEstimateOnV_n}. By Lemma~\ref{lemma:Approximation}, we infer $u_n\to u $ in $C(\overline{I}; \rH^\alpha_L)$ and $D_t u_n\to D_t u$ in $C^1(\overline{I};\rH^{\alpha-1}_L)$.
		We claim that
		\begin{equation}\label{eq:EnergyEstimateInhomogeneity}
			\|\Box_Pu_n\|_{\rL^1(I;\rH^{\alpha-1}_L)}\lesssim_{|I|} \|\Box_Pu \|_{\rL^1(J;\rH^{\alpha-1}_L)}+2^{-n}\|u\|_{C(\overline{J};\rH^\alpha_L)}.
		\end{equation}
		Taking \eqref{eq:EnergyEstimateInhomogeneity} for granted, we may let $n\to \infty$ in \eqref{eq:EnergyEstimateOnV_n} to obtain
		\begin{equation}\label{eq:EnergyEstimateSubinterval}
			\|u\|_{C(\overline{I}; \rH_L^\alpha)}\lesssim_{|I|}  \|u(0)\|_{\rH_L^\alpha}+\|D_tu(0)\|_{\rH_L^{\alpha-1}}+\|\Box_P u\|_{\rL^1(J;\rH_L^{\alpha-1})},
		\end{equation} 
		and we infer the claim from \eqref{eq:EnergyEstimateSubinterval} by simply exhausting $J$ by open, compactly contained subintervals $I\Subset J$. So it remains to prove \eqref{eq:EnergyEstimateInhomogeneity}. First, note that since $u\in W^{2,1}(J;\rH^{\alpha-2}_L)$ by assumption, Lemma~\ref{lemma:Approximation}~(b) implies
		\begin{equation}\label{eq:TimeDetivativeOFInhomogeneity}
			D_t^2u_n=D_t^2(\varphi_n\ast \Pi_n u)=(\varphi_n\ast \Pi_n D_t^2u)=S_n\Pi_n D_t^2u
		\end{equation}
		in $\rL^1(I;\rH^{\alpha-2}_L)$. On the other hand, since $P(t)=\sum_{j=1}^d b_j(t) L_j$, one may verify that 
		\begin{equation}\label{eq:PAppliedToInhomogeneity}
			(Pu_n)(t)=(S_n\Pi_n Pu)(t)- Ru(t)  \quad \text{for }t\in\overline{I}
		\end{equation}
		with $Ru=\sum_{j=1}^d R_ju$ and 
		\begin{equation*}
			R_j u(t)=\int_{\R}\varphi_n(t-s)(b_j(t)-b_j(s)) (L_j\Pi_n u)(s) \dd s.
		\end{equation*}
		Note that by Lemma~\ref{lemma:ExtensionOfL_jToHomogeneousSobolevSpaces} and Lemma~\ref{lemma:Approximation}~(a)
		\begin{align*}
			\begin{split}
				\|R_j u(t)\|_{\rH^{\alpha-1}_L}
				\leq{}& \|b_j'\|_\infty \|\Pi_n u\|_{C(\overline{J};\rH^{\alpha+1}_L)} \int_{\R}|\varphi_n(t-s)||t-s|\dd s\\
				\lesssim{}&m_4\cdot \big(2^{(n+1)}\|\Pi_n u\|_{C(\overline{J};\rH^{\alpha}_L)}\big) \cdot 4^{-n}\|\varphi\|_1\\
				\lesssim{}& 2^{-n}  \|u\|_{C(\overline{J};\rH^\alpha_L)}
			\end{split}
		\end{align*}
		for each $t\in \overline{I}$ and thus 
		\begin{equation}\label{eq:EstimateForErrorInhomogeneity}
			\|Ru\|_{\rL^1(I;\rH^{\alpha-1}_L)}\lesssim_{|I|} 2^{-n} \textbf{}\|u\|_{C(\overline{J},\rH^\alpha_L)} .	
		\end{equation}
		%	Since $P(t)=\sum_{j=1}^d b_j(t) L_j$, this gives
		%	\begin{equation}\label{eq:PAppliedToInhomogeneity}
			%		(Pu_n)(t)=(S_n\Pi_n Pu)(t)- Ru(t)  \quad \text{for }t\in\overline{I}
			%	\end{equation}
		%	with $Ru=\sum_{j=1}^d R_ju$.
		Combining \eqref{eq:TimeDetivativeOFInhomogeneity} and \eqref{eq:PAppliedToInhomogeneity}, we conclude
		\begin{equation*}
			\Box_P u_n=(D_t^2-P)u_n= S_n\Pi_n \Box_P u-Ru \quad \text{in } \rL^1(I;\rH^{\alpha-1}_L),
		\end{equation*}
		and applying Lemma~\ref{lemma:Approximation} and \eqref{eq:EstimateForErrorInhomogeneity} gives
		\begin{align*}
			\|\Box_P u_n\|_{\rL^1(I;\rH^{\alpha-1}_L)}&\leq{}\|S_n\Pi_n\Box_P u_n\|_{\rL^1(I;\rH^{\alpha-1}_L)}+\|Ru \|_{\rL^1(I;\rH^{\alpha-1}_L)}\\
			&\lesssim_{|I|}{}\|\Box_P u_n\|_{\rL^1(I;\rH^{\alpha-1}_L)}+ 2^{-n}\|u\|_{C(\overline{J},\rH^\alpha_L)}
		\end{align*}
		as desired. This completes the proof.
	\end{proof}
	\begin{corollary}[Uniqueness of Weak Solutions in $\rH^\alpha_L$]\label{cor:UniquenessOfWeakSolutions} Let $\alpha\in \R$ and suppose that $g\in \rH^\alpha_L, h\in \rH^{\alpha-1}_L$, and $F\in \rL^1(\R;\rH^{\alpha-1}_L)$. If $u$ and $v$ are weak solutions to \eqref{eq:WaveEquationWithTimeDependentLipschitzMetrics} in $\rH^\alpha_L$, then $u=v$.
	\end{corollary}
	\begin{proof} Define $J_n\coloneqq (-n,n)$ for $n\in \N$. Applying Proposition \ref{prop:EnergyEstimate} to $w_n\coloneqq (u-v)|_{\overline{J_n}}$ gives $w_n=0$. Letting $n\to \infty$ yields $w=0$ and thus $u=v$ as desired.
	\end{proof}

\section{Global Strichartz Estimates}\label{section:GlobalStrichartzEstimates}
	In this section, we prove Theorem~\ref{mainthm:Strichartz} following the approach in the proof of \cite[Theorem~2]{FreySchippa2023}. Recall the parametrices from Section~\ref{section:ParametrixConstruction},
	\begin{equation*}
	T^{\pm}(t,s)= \sum_{\lambda\in 2^\Z} T^\pm_\lambda(t,s) =\sum_{\lambda\in 2^\Z} (\ee^{\pm \ii \varphi_{t,s}} \psi_\lambda)(\sqrL).
	\end{equation*}
	We first establish a Strichartz inequality for $T_1^\pm (t,s)$ in Subsection~\ref{subsection:StrichartzEstimatesForUnitFrequencies} using the Keel--Tao framework \cite{KeelTao1998}, see Proposition~\ref{prop:StrichartzEstimatesForUnitFrequencies} below. In Subsection \ref{subsection:GlobalStrichartzEstimatesforT_lambda(t,s)}, we obtain global Strichartz estimates for the dyadically localized operators $T_\lambda^{\pm}(t,s)$ by a scaling argument. Finally, in Subsection~\ref{subsection:GlobalStrichartzEstimatesforT(t,s)}, we patch these estimates together to deduce global Strichartz estimates for $T^{\pm}(t,s)$ (see Theorem~\ref{thm:StrichartzEstimatesForParametricesI}~(a) below). This then immediately translates to Strichartz estimates for the parametrices $C(t,s)$ and $S(t,s)$, which, when combined with Theorem~\ref{thm:WellposednessInAdaptedSpaces} and the Christ--Kiselev lemma, then yields the proof of Theorem~\ref{mainthm:Strichartz}. 
	\subsection{Global Strichartz Estimates for Unit Frequencies}\label{subsection:StrichartzEstimatesForUnitFrequencies}
	In this subsection, we prove global Strichartz estimates for $T^{\pm}_1(t,s)$. More precisely, we show:
	\begin{proposition}[Global Strichartz Estimates for $T_1^{\pm}(t,s)$]\label{prop:StrichartzEstimatesForUnitFrequencies} Let $\eps_0\in (0,\frac{1}{2})$ from \eqref{eq:Ellipticity_b} be sufficiently small and suppose that $(p,q,\alpha)$ is a wave-admissible Strichartz triple and $s\in \R$. Then,
	\begin{equation}\label{eq:StrichartzEstimatesForUnitFrequencies}
		\|L^{-\frac{\alpha}{2}}(\ee^{\pm \ii\varphi_{t,s}}\psi)(\sqrL)f\|_{\rL^p_t(\R;\rL_x^q(\R^d))}\lesssim \|f\|_2 \quad (f\in \rL^2),
	\end{equation}
	with an implicit constant independent of $s$.
	\end{proposition}
	Our proof is based on the celebrated result by Keel--Tao \cite{KeelTao1998}. For the convenience of the reader, we restate it here. Recall that, for given $\gamma>0$, an exponent pair $(p,q)\in [2,\infty]^2$ is called \textit{$\gamma$-admissible} if $(p,q,\sigma)\neq (2,\infty,1)$ and $\frac{1}{p}+\frac{\gamma}{q}\leq \frac{\gamma}{2}$. Note that if $(p,q,\alpha)$ is a wave-admissible Strichartz triple, then $(p,q)$ is $\frac{d-1}{2}$-admissible (cf. \eqref{eq:StrichartzTriple}).

	\begin{theorem}[\texorpdfstring{\cite[Theorem~1.2]{KeelTao1998}}{[KeelTao, Thm. 1.2]}]\label{thm:KeelTao} Let $(\Omega,\mu)$ be a measure space and $H$ be a Hilbert space. Suppose that $U(t)\colon H \to \rL^2(\Omega)$ is a linear operator for each $t\in \R$ satisfying the following properties:
	\begin{enumerate}[label=(\roman*)]
		\item \textit{Uniform Boundedness}: For all $t\in \R$ and $f\in H$
		\begin{equation*}
			\|U(t)f\|_{\rL_x^2(\Omega)}\lesssim \|f\|_H.
		\end{equation*}
		\item \textit{Truncated Decay}:  For some $\gamma>0$, we have
		\begin{equation*}
			\|U(s)U(t)^\ast g\|_{\rL_x^\infty(\Omega)} \lesssim (1+|t-s|)^{-\gamma}\|g\|_{\rL_x^1(\Omega)}\quad \text{for all } s,t\in \R, g\in S,
		\end{equation*}
		where $S\subseteq \rL^1(\Omega)\cap \rL^2(\Omega)$ denotes a dense subset in $\rL^1(\Omega)$.
	\end{enumerate}
	Then, the estimates
	\begin{equation*}
		\|U(t)f\|_{\rL^p_t(\R;\rL^q_x(\Omega))}\lesssim \|f\|_H\quad (f\in H)
	\end{equation*}
	and 
	\begin{equation*}
		\left\|\int_0^t U(t)U(s)^\ast F(s)\dd s \right\|_{\rL^p_t(\R;\rL^q_x(\Omega))}\lesssim \|f\|_H\quad (F\in \rL^{\tilde{p}'}_t(\R;\rL^{\tilde{q}'}_x(\Omega))),
	\end{equation*}
	hold true for any $\gamma$-admissible exponent pairs $(p,q), (\tilde{p},\tilde{q})$.
	\end{theorem}
	We want to apply Theorem~\ref{thm:KeelTao} with $U_\pm(t)\coloneqq L^{-\frac{\alpha}{2}}(\ee^{\pm \ii\varphi_{t,s}}\psi)(\sqrL)$ (for fixed $s\in \R$). Condition (i) in Theorem~\ref{thm:KeelTao} is straightforward to verify using the Phillips functional calculus. However, for condition (ii), we need an oscillatory integral estimate, whose proof we postpone to Section~\ref{section:OscillatoryIntegralEstimate}. In order to state it, we set for $\eps\in (0,\frac{1}{2})$
	\begin{equation*}
	\cB_\eps\coloneqq \{\tilde{B}\colon [0,1]\to \R^{d\times d} \mid \tilde{B}(t)\text{ is diagonal for all $t\in [0,1]$ and }\|\tilde{B}-\id\|_\infty\leq \eps\}
	\end{equation*}
	and associate with $\tilde{B}\in \cB_\eps$ the phase function
	\begin{equation}\label{eq:DefinitionOfGeneralPhaseFunction}
	\varphi_{\tilde{B}}\colon \R^d\to \R, \quad \varphi_{\tilde{B}}(\xi)=\int_0^1 (\tilde{B}(\tau)\xi|\xi)^{\frac{1}{2}}\dd \tau.
	\end{equation}
	To ease notation, we just write $\tilde{\varphi}\coloneqq \varphi_{\tilde{B}}$, if $\tilde{B}$ is clear from the context. 
	Observe that, if $\varphi_{t,s}$ denotes the phase function from Definition~\ref{def:PhaseFunction}, then a simple change of variables shows that 
	\begin{equation}\label{eq:ChangeOfVariablesPhaseFunction}
		\varphi_{t,s}=(t-s)\tilde{\varphi}_{t,s}, \,\text{where }\tilde{\varphi}_{t,s}\coloneqq \varphi_{\tilde{B}_{t,s}},\;\tilde{B}_{t,s}(\tau)\coloneqq B((t-s)\tau+s)\quad (\tau\in [0,1]),
	\end{equation} 
	and $\tilde{B}_{t,s}\in \cB_{\eps_0}$ by \eqref{eq:Ellipticity_b}.
	\begin{theorem}[Oscillatory Integral Estimate]\label{thm:OscillatoryIntegralEstimate} There exists $\eps_0\in (0,\frac{1}{2})$ such that the following holds. For any $\tilde{\psi}\in C_c^\infty(\R^d)$ supported away from the origin, there exists some constant $C=C(\tilde{\psi},d)>0$ such that we have the decay estimate
	\begin{equation*}
		|\cF^{-1}(\ee^{\ii t\tilde{\varphi}}\tilde{\psi})(y)|\leq C  (1+|t|)^{-\frac{d-1}{2}}
	\end{equation*}	 
	for all $t\in \R, y\in \R^d,$ and $\tilde{B}\in \cB_{\eps_0}$.
	\end{theorem}
	Equipped with Theorem~\ref{thm:KeelTao} and Theorem~\ref{thm:OscillatoryIntegralEstimate}, we are in the position to prove Proposition~\ref{prop:StrichartzEstimatesForUnitFrequencies}.
	\begin{proof}[Proof of Proposition~\ref{prop:StrichartzEstimatesForUnitFrequencies}]
	Let $(p,q,\alpha)$ be a wave-admissible Strichartz triple and fix $s\in \R$. Set $\tilde{\psi}\coloneqq |\cdot|^{-\alpha} \psi$. By Lemma \ref{lemma:ExtensionOfPhillipsFunctionalCalculus}~(b), we have for $f\in \rL^2$ that
	\begin{align*}
	L^{-\frac{\alpha}{2}}(\ee^{\pm  \ii\varphi_{t,s}}\psi)(\sqrL)f={}&\big(\ee^{\pm  \ii\varphi_{t,s}}\tilde{\psi}\big)(\sqrL)f\eqqcolon U_{\pm}(t)f.
	\end{align*}
	Thus, it remains to show
	\begin{equation}\label{eq:StrichartzEstimatesForUnitFrequencies2}
	\|U_\pm(t)f\|_{\rL^p_t(\R;\rL^q_x(\R^d))}\lesssim \|f\|_2\quad (f\in \rL^2),
	\end{equation}
	for which we invoke Theorem~\ref{thm:KeelTao}. To ease notation, we provide the proof for $U\coloneqq U_+$ (the proof for $U_-$ is analogous). We verify conditions (i) and (ii) from Theorem \ref{thm:KeelTao}. Property (i) is easy since 
	\begin{equation*}
	\|U(t)f\|_2\lesssim_{M_2} \|\ee^{\ii\varphi_{t,s}}\tilde{\psi}\|_\infty \|f\|_2\lesssim_{\psi,d,\alpha} \|f\|_2 \quad (f\in \rL^2,t \in \R)
	\end{equation*}
	by Proposition \ref{prop:PropertiesPhillipsCalculus}~(c). As $\sqrL$ and therefore also $\sqrL$ is self-adjoint w.r.t. \eqref{eq:EquivalentScalarProduct}, we have $(\ee^{\ii y\cdot \sqrL})^\ast=\ee^{-\ii y\cdot \sqrL}$ and therefore
	\begin{equation*}
	(U(\sigma))^\ast=\big(\overline{\ee^{\ii\varphi_{\sigma,s}}\tilde{\psi}}\big)(\sqrL)=\big(\ee^{-\ii\varphi_{\sigma,s}}\tilde{\psi}\big)(\sqrL).
	\end{equation*}
	Thus,
	\begin{equation}\label{eq:KeelTaoProperty2}
	U(\tau)(U(\sigma))^\ast =\big(\ee^{\ii(\varphi_{\tau,s}-\varphi_{\sigma,s})}\tilde{\psi}^2\big)(\sqrL)=\big(\ee^{\ii\varphi_{\tau,\sigma}}\tilde{\psi}^2\big)(\sqrL)
	\end{equation}
	for $\tau,\sigma\in \R$. Now, by Proposition~\ref{prop:L^inftyL^1-Estimate}~(b), we have for $g\in \rL^1\cap \rL^2$ that
	\begin{equation}\label{eq:Property2KeelTao}
	\|U(\tau)(U(\sigma))^{\ast}g\|_{\infty}=\big\|\big(\ee^{\ii\varphi_{\tau,\sigma} }\tilde{\psi}^2\big)(\sqrL)g\big\|_\infty\lesssim \|\cF^{-1}(\ee^{\ii\varphi_{\tau,\sigma} }\tilde{\psi}^2)\|_\infty \|g\|_1.
	\end{equation}
In view of \eqref{eq:ChangeOfVariablesPhaseFunction}, we have $\varphi_{\tau,\sigma}=(\tau-\sigma)\tilde{\varphi}_{\tau,\sigma}$ with $\tilde{\varphi}_{\tau,\sigma}=\varphi_{\tilde{B}_{\tau,\sigma}}$ and $\tilde{B}_{\tau,\sigma}\in \cB_{\eps_0}$.
 Hence, we obtain from Theorem~\ref{thm:OscillatoryIntegralEstimate}
	\begin{equation*}
	\|\cF^{-1}(\ee^{\ii\varphi_{\tau,\sigma} }\tilde{\psi}^2)\|_\infty= \|\cF^{-1}(\ee^{\ii(\tau-\sigma)\tilde{\varphi}_{\tau,\sigma} }\tilde{\psi}^2)\|_\infty\lesssim_{\tilde{\psi},d} (1+|\tau-\sigma|)^{-\frac{d-1}{2}}\|g\|_1 \quad (\tau,\sigma\in \R).
	\end{equation*}
	Using this in \eqref{eq:Property2KeelTao} shows condition (ii) in Theorem~\ref{thm:KeelTao}. Applying the latter theorem (separately in both components of $\rL^2=\rL^2(\R^d;\C^2)$), we obtain \eqref{eq:StrichartzEstimatesForUnitFrequencies2} as desired.
	\end{proof}
	\begin{remark}\label{rem:StrichartzEstimatesForUnitFrequencies} A careful inspection of the proof reveals that Propostion~\ref{prop:StrichartzEstimatesForUnitFrequencies} remains true (with same implicit constant) if one replaces the $\varphi$-defining matrix $B(t)$ by any matrix $B(\ell(t))$ with bijective transformation $\ell\colon \R\to \R$, as the argument only hinges on assumption \eqref{eq:Ellipticity_b}, which is invariant under this transformation.
	\end{remark}
	\subsection{Global Strichartz Estimates for $T_\lambda^{\pm}(t,s)$}\label{subsection:GlobalStrichartzEstimatesforT_lambda(t,s)}
	Let $p\in (1,\infty)$. For $\lambda>0$, we define the dilation operator $\delta_\lambda\colon \rL^p\to \rL^p$ by $(\delta_\lambda f)(x)\coloneqq f_\lambda(x)\coloneqq f(\frac{x}{\lambda})$ and the rescaled operator
	$\mathbf{L}_\lambda\coloneqq (L_{1,\lambda},\dots,L_{d,\lambda})$ by
	\begin{equation*}
	L_{j,\lambda}\coloneqq \begin{pmatrix}
		D_j a_{j,\lambda} D_j&\\
		0& a_{j,\lambda} D_j^2
	\end{pmatrix},\quad\quad a_{j,\lambda}(x_j)\coloneqq a_j\big(\tfrac{x_j}{\lambda}\big).
	\end{equation*}
	The following lemma shows that $T^\pm_\lambda(t,s)$ is obtained from $T^{\pm}_1(t,s)$ by rescaling.
	\begin{lemma}[Scaling]\label{lemma:RescalingDyadicFrequencies} Let $\lambda>0$. 
	\begin{enumerate}[label=(\alph*)]
		\item We have \begin{equation}\label{eq:ScalingSymmetryOfHalfWaveGroup}
			\delta_\lambda \ee^{\ii\frac{y}{\lambda} \sqrt{\mathbf{L}}} =\ee^{\ii y \sqrt{\mathbf{L}_\lambda}}\delta_\lambda.
		\end{equation}
		In particular,
		\begin{equation}\label{eq:LambdaIndependenceOfOperatorNormOfHalfWaveGroup}
			\sup_{\lambda>0, y\in \R^d} \|\ee^{\ii y \sqrt{\mathbf{L}_\lambda}}\|_{\cL(\rL^2)}\leq 1.
		\end{equation} 
		\item We have\begin{equation*}
			\big[\big(\ee^{ \pm \ii\varphi_{t,s}} \psi_\lambda)(\sqrL)f\big]_\lambda=\big(\ee^{\pm \ii \lambda\varphi_{t,s}}\psi\big)(\sqrt{\mathbf{L}_\lambda})f_\lambda \quad (f\in \rL^2).
		\end{equation*} 
	\end{enumerate}
	\end{lemma}
	\begin{proof} Let $p\in (1,\infty)$ and $\lambda>0$.
	\begin{enumerate}[label=(\alph*)]
		\item A straightforward computation reveals the identity $\bL_\lambda =\delta_\lambda (\lambda^{-2} \bL) \delta_{\lambda^{-1}}$, which implies $\ssqrt{\bL_\lambda} =\delta_\lambda (\lambda^{-1} \sqrL) \delta_{\lambda^{-1}}$ by the holomorphic functional calculus for sectorial operators. By standard semigroup theory (see e.g. \cite[Sections~II.2.1 and II.2.2]{EngelNagel2000}), it follows that $\ssqrt{\bL_\lambda}$ generates a $d$-parameter $C_0$-group on $\rL^p$ given by
		\begin{equation*}
			\ee^{\ii y\ssqrt{\bL_\lambda}}=\delta_\lambda \ee^{\ii \frac{y}{\lambda}\ssqrt{\bL}}\delta_{\lambda^{-1}}\quad (y\in \R^d).
		\end{equation*}
		Now, \eqref{eq:LambdaIndependenceOfOperatorNormOfHalfWaveGroup} is a straightforward consequence, since
		\begin{equation*}
			\big\|\ee^{\ii y\cdot \sqrt{\mathbf{L}_\lambda}}\big\|=\big\|\delta_\lambda\big(\ee^{\ii\frac{y}{\lambda}\cdot \sqrt{\mathbf{L}}}\big)\delta_{\lambda^{-1}}\big\|
			\leq \big\|\delta_\lambda\big\|\big\|\ee^{\ii\frac{y}{\lambda}\cdot \sqrt{\mathbf{L}}}\big\| \big\|\delta_{\lambda^{-1}}\big\|\leq \lambda^{\frac{d}{2}}\cdot 1 \cdot \lambda^{-\frac{d}{2}}=1,
		\end{equation*} 
		where in the above display $\|\cdot\|$ denotes the norm in $\cL(\rL^2)$.
		\item Let $s,t\in$ and $f\in \rL^2$. By Remark \ref{rem:PhillipsCalculusForSymmetricFunctions}, we have
	\begin{equation}\label{eq:RepresentingDyadicPieceByRescaling}
		\big[(\ee^{\pm \ii\varphi_{t,s}}\psi_\lambda)(\sqrL)f\big]_\lambda={} \int_{\R^d} \cF^{-1}(\ee^{\ii \varphi_{t,s}}\psi_\lambda)(y)\big(\mathrm{Cos}(y\sqrL)f\big)_\lambda \dd y.
	\end{equation}	
	Observe that $\cF^{-1}(\ee^{\ii \varphi_{t,s}}\psi_\lambda)(y)=\lambda^d \cF^{-1}(\ee^{\ii \lambda\varphi_{t,s}}\psi)(\lambda y)$ by the change of variables $\xi\mapsto \lambda \xi$ and the positive homogeneity of $\varphi_{t,s}$. Moreover, $\big(\mathrm{Cos}(y\sqrL)f\big)_\lambda=\mathrm{Cos}(\lambda y\ssqrt{\bL_\lambda}))f_\lambda$ by \eqref{eq:ScalingSymmetryOfHalfWaveGroup}. Using these relations in \eqref{eq:RepresentingDyadicPieceByRescaling} gives
	\begin{align*}
		\big[(\ee^{\pm \ii\varphi_{t,s}}\psi_\lambda)(\sqrL)f\big]_\lambda={}& \int_{\R^d} \lambda^d \cF^{-1}(\ee^{\ii \lambda \varphi_{t,s}}\psi)(\lambda y)\mathrm{Cos}(\lambda y\ssqrt{\bL_\lambda}))f_\lambda \dd y\\
		={}& \int_{\R^d} \cF^{-1}(\ee^{\ii \lambda \varphi_{t,s}}\psi)( y)\mathrm{Cos}( y\ssqrt{\bL_\lambda}))f_\lambda \dd y=(\ee^{\pm \ii \lambda\varphi_{t,s}}\psi)(\ssqrt{\bL_\lambda})f_\lambda.
	\end{align*}
	\end{enumerate}
	\end{proof}
	\begin{lemma}[Global Strichartz Estimates for $T_\lambda^{\pm}(t,s)$]\label{lemma:StrichartzEstimatesForDyadicFrequencies} Suppose that $(p,q,\alpha)$ is a wave-admissible Strichartz triple and $s\in \R$, $\lambda>0$. Then,
	\begin{equation}\label{eq:StrichartzForDyadicPieces}
		\|L^{-\frac{\alpha}{2}}(\ee^{\pm \ii\varphi_{t,s}}\psi_\lambda)(\sqrL)f\|_{\rL^p_t(\R;\rL_x^q(\R^d))}\lesssim \|f\|_2 \quad (f\in \rL^2),
	\end{equation}
	with an implicit constant independent of $s$ and $\lambda$.
	\end{lemma}
	\begin{proof} Let $(p,q,\alpha)$ be a wave-admissible Strichartz triple, $\lambda>0$, and $f\in \rL^2$. As in the proof of Proposition~\ref{prop:StrichartzEstimatesForUnitFrequencies}, set $\tilde{\psi}\coloneqq |\cdot|^{-\alpha}\psi$. Then, by Lemma~\ref{lemma:ExtensionOfPhillipsFunctionalCalculus}~(b) and Lemma~\ref{lemma:RescalingDyadicFrequencies}~(b), we have 
	\begin{align*}
		L^{-\frac{\alpha}{2}}(\ee^{\pm  \ii\varphi_{t,s}}\psi_\lambda)(\sqrL)f={}&\lambda^{-\alpha}(\ee^{\pm  \ii\varphi_{t,s}}\tilde{\psi}_\lambda)(\sqrL)f\\
		={}&
		\lambda^{-\alpha}\delta_{\lambda^{-1}}\big(\ee^{\pm  \ii\lambda \varphi_{t,s}}\tilde{\psi}\big)(\ssqrt{\mathbf{L}_\lambda})\delta_{\lambda}f\\
		={}&
		\lambda^{-\alpha}\delta_{\lambda^{-1}}\big(\ee^{\pm  \ii \varphi^\lambda_{\lambda t,\lambda s}}\tilde{\psi}\big)(\ssqrt{\mathbf{L}_\lambda})\delta_{\lambda}f,
	\end{align*}	
	where $\varphi^\lambda_{t,s}$ is just defined as $\varphi_{t,s}$, with $B(\tau)$ replaced by $B_\lambda(\tau)=B(\frac{\tau}{\lambda})$, see Definition~\ref{def:PhaseFunction}.
	Now, applying the change of variables $x\mapsto \frac{x}{\lambda}$, $t\mapsto \frac{t}{\lambda}$ first and then Remark~\ref{rem:StrichartzEstimatesForUnitFrequencies} as well as Proposition~\ref{prop:StrichartzEstimatesForUnitFrequencies} with $\ssqrt{\bL_\lambda}$ in place of $\sqrL$ (note that this is justified since Assumption \ref{ass:MainThm} is invariant under the scaling $a_j\mapsto a_{j,\lambda}$), we obtain
	\begin{align*}
		\|L^{-\frac{\alpha}{2}}(\ee^{\pm  \ii\varphi_{t,s}}\psi_\lambda)(\sqrL)f\|_{\rL^p_t(\R;\rL_x^q(\R^d))}={}&	\lambda^{-\alpha}\|\delta_{\lambda^{-1}}\big(\ee^{\pm  \ii \varphi^\lambda_{\lambda t,\lambda s}}\tilde{\psi}\big)(\ssqrt{\mathbf{L}_\lambda})\delta_{\lambda}f\|_{\rL^p_t(\R;\rL_x^q(\R^d))}\\
		={}&\lambda^{-\big(\alpha+\frac{1}{p}+\frac{d}{q}\big)}\|\big(\ee^{\pm  \ii \varphi^\lambda_{ t,\lambda s}}\tilde{\psi}\big)(\ssqrt{\mathbf{L}_\lambda})\delta_{\lambda}f\|_{\rL^p_t(\R;\rL_x^q(\R^d))}\\
		\lesssim{}& \lambda^{-\big(\alpha+\frac{1}{p}+\frac{d}{q}\big)}\|\delta_\lambda f\|_{2}\\
		={}&\lambda^{-\big(\alpha-\frac{d}{2}+\frac{1}{p}+\frac{d}{p}\big)} \|f\|_{2}=\|f\|_{2},
	\end{align*}
	where the last equality follows from the fact that $(p,q,\alpha)$ is a wave-admissible Strichartz triple. 
	\end{proof}
	\subsection{Proof of Theorem~\ref{mainthm:Strichartz}}\label{subsection:GlobalStrichartzEstimatesforT(t,s)}
	We recall the following lemma, which is due to Christ--Kiselev.
	\begin{lemma}[\texorpdfstring{\cite[Lemma~2.4]{Tao2006}}{[Tao, Lemma 2.4]}]\label{lemma:ChristKiselev} Let $X,Y$ be Banach spaces and let $K\in C(\R\times \R; \cL(X,Y))$. Suppose that $1\leq p<q\leq \infty$ is such that
	\begin{equation*}
		\bigg\|\int_\R K(t,s)F(s) \dd s\bigg\|_{\rL_t^q(\R;Y)}\lesssim \|F\|_{\rL_t^p(\R;X)}
	\end{equation*}
	for all $F\in \rL_t^p(\R;X)$. Then, one also has
	\begin{equation*}
		\bigg\|\int_{-\infty}^t K(t,s)F(s) \dd s\bigg\|_{\rL^q_t(\R;Y)}\lesssim \|F\|_{\rL_t^p(\R;X)}.
	\end{equation*}
	\end{lemma}

	\begin{theorem}[Strichartz Estimates for Parametrices I]\label{thm:StrichartzEstimatesForParametricesI} Suppose that $(p,q,\alpha)$ is a wave-admissible Strichartz triple. 
	\begin{enumerate}[label=(\alph*)]
		\item The estimate
		\begin{equation*}
			\|L^{-\frac{\alpha}{2}}T^{\pm}(t,s)f\|_{\rL^p_t(\R;\rL_x^q(\R^d))}\lesssim \|f\|_2 \quad (f\in \SL)
		\end{equation*}
		holds uniformly in $s\in \R$.
		\item If $(\tilde{p}, \tilde{q}, \tilde{\alpha})$ is another wave-admissible Strichartz triple with $\tilde{p}'<p$, then 
		\begin{equation*}
			\left\|\int_0^t L^{-\tfrac{\alpha+\tilde{\alpha}}{2}}T^\pm(t,s) G(s)\dd s \right\|_{\rL^p_t(\R;(\rL^q_x(\R^d)))}\lesssim \|G\|_{\rL_t^{\tilde{p}'}(\R;\rL_x^{\tilde{q}'}(\R^d))}
		\end{equation*}
		for all measurable functions $G\colon \R\to \SL$ of compact support.
	\end{enumerate}		
	\end{theorem}
	\begin{proof} Let $(p,q,\alpha)$ be a wave-admissible Strichartz triple.
	\begin{enumerate}[label=(\alph*)]
		\item Let $f\in \SL$ and $s\in \R$. Then, Proposition~\ref{prop:SquareFunctionCharacterizationOfL^pNorm} followed by an application of Minkowski's inequality (note that $p,q\geq 2$ in view of \eqref{eq:StrichartzTriple}) gives
	\begin{align}\label{eq:StrichartzLittlewoodPaley}
		\begin{split}
			\|L^{-\frac{\alpha}{2}}T^\pm(t,s)f\|^2_{\rL^p_t(\R;\rL^q_x(\R^d))}\simeq{}& \bigg\|\bigg(\sum_{\mu\in 2^\Z} \big| \psi_\mu (\sqrL)L^{-\frac{\alpha}{2}}T^\pm(t,s)f\big|^2\bigg)^{\frac{1}{2}}\bigg\|^2_{\rL^p_t(\R;\rL^q_x(\R^d))}\\
			\lesssim{}& \sum_{\mu\in 2^\Z} \big\| \psi_\mu (\sqrL)L^{-\frac{\alpha}{2}}T^\pm(t,s)f\big\|^2_{\rL^p_t(\R;\rL^q_x(\R^d))}.
		\end{split}
	\end{align}
	Note that by the Phillips functional calculus and the fact that the $\psi_\lambda$ are almost disjointly supported, we have
	\begin{align*}
		\psi_\mu (\sqrL)L^{-\frac{\alpha}{2}}T^\pm(t,s)f
		={}&\sum_{\lambda\in 2^{\Z}} (|\cdot|^{-\alpha}\ee^{\pm \ii\varphi_{t,s}}\psi_\lambda\psi_\mu) (\sqrL)f\\
		={}&\sum_{\lambda\in I_\mu} L^{-\alpha}(\ee^{\pm \ii\varphi_{t,s}}\psi_\lambda)(\sqrL) \psi_\mu(\sqrL)f,
	\end{align*}
	where $I_\mu\coloneqq \{\tfrac{\mu}{2}, \mu, 2\mu\}$. So the sum on the right-hand side of \eqref{eq:StrichartzLittlewoodPaley} can be estimated by
	\begin{align*}
		&{}3\sum_{\mu\in 2^\Z}\sum_{\lambda\in I_\mu} \big\| L^{-\frac{\alpha}{2}}(\ee^{\pm \ii\varphi_{t,s}}\psi_\lambda)(\sqrL)\psi_\mu (\sqrL)f\big\|^2_{\rL^p_t(\R;\rL^q_x(\R^d))}\\
		={}&3\sum_{\lambda\in 2^\Z}\sum_{\mu\in I_\lambda} \big\| L^{-\frac{\alpha}{2}}(\ee^{\pm \ii\varphi_{t,s}}\psi_\lambda)(\sqrL)\psi_\mu (\sqrL)f\big\|^2_{\rL^p_t(\R;\rL^q_x(\R^d))}\\
		\lesssim{}& 3\sum_{\lambda\in 2^\Z}\sum_{\mu\in I_\lambda} \big\|\psi_\mu (\sqrL)f\big\|^2_{2}= 9\sum_{\mu\in 2^\Z} \big\|\psi_\mu (\sqrL)f\big\|^2_{2}\simeq \|f\|_2^2,
	\end{align*}
	where we used Lemma~\ref{lemma:StrichartzEstimatesForDyadicFrequencies} and Proposition~\ref{prop:L^2BoundednessOfAlmostOrthogonalOperators} in the last line of the above display. This proves (a).\\[0.2cm]
	\item Let $(\tilde{p}, \tilde{q},\tilde{\alpha})$ be another wave-admissible Strichartz triple with $\tilde{p}'<p$. Fix $s\in \R$. Since $\SL$ is dense in $\rL^2(\R^d)$ by Proposition~\ref{prop:CalderonReproducingFormulaOnInhomogeneousSobolevSpaces}, it follows from (a) that
	\begin{equation*}
		\mathcal{T}_s\colon \rL^2(\R^d)\to \rL_t^{\tilde{p}}(\R;\rL^{\tilde{q}}_x(\R^d)), \quad (\mathcal{T}_s f)(t)=\rL^{-\frac{\tilde{\alpha}}{2}}T^\pm(t,s)f
	\end{equation*}
	is a bounded linear operator. Recalling the scalar product \eqref{eq:EquivalentScalarProduct} on $\rL^2$, we may define the equivalent dual pairing
	\begin{equation*}
		\langle F,G \rangle\coloneqq \int_{\R} \langle F(\tau), G(\tau)\rangle_A \dd \tau
	\end{equation*}
	on $\rL_\tau^{\tilde{p}}(\R;\rL^{\tilde{q}}_x(\R^d))\times \rL_\tau^{\tilde{p}'}(\R;\rL^{\tilde{q}'}_x(\R^d))$. Since $\mathcal{T}_s$ is bounded, the dual operator $\mathcal{T}_s^\ast% \colon  \rL_t^{\tilde{p}'}(\R;\rL^{\tilde{q}'}_x(\R^d))\to \rL^2(\R^d)
	$ and thus the composition $\mathcal{T}_s\mathcal{T}_s^\ast$ are bounded, too. Let $G\colon \R\to \SL$ be measurable with compact support. Then, a straightforward computation reveals
	\begin{equation*}
		\mathcal{T}_s^\ast G=\int_{\R} \mathcal{T}_{\tau}(s) G(\tau) \dd \tau,
	\end{equation*}
	and using that 
	\begin{equation*}
		T^\pm(t,s)T^\pm(s,\tau)=T^\pm(t,\tau), \quad  T^\pm(t,s)^\ast =T^\pm(s,t)\quad \text{for }\quad t,s,\tau\in \R
	\end{equation*}
	as identities on $\cL(\rL^2)$, we conclude
	\begin{equation*}
		(\mathcal{T}_s\mathcal{T}_s^\ast G)(t)= \int_{\R}L^{-\frac{\alpha+\tilde{\alpha}}{2}}T^\pm(t,\tau)G(\tau) \dd \tau
	\end{equation*}
	for a.e. $t\in \R$. Now, the claim follows from the boundedness of $\mathcal{T}_s\mathcal{T}_s^\ast$ and Lemma~\ref{lemma:ChristKiselev}.
	\end{enumerate}
	\end{proof}
	\begin{corollary}[Strichartz Estimates for Parametrices II]\label{cor:StrichartzEstimatesParatetricesII} Suppose that $(p,q,\alpha)$ is a wave-admissible Strichartz triple. 
	\begin{enumerate}[label=(\alph*)]
		\item We have \begin{equation}\label{eq:StrCos}
			\|L^{\frac{1-\alpha}{2}}C(t,s)f\|_{\rL^p_t(\R;\rL_x^q(\R^d))}\lesssim \|L^{\frac{1}{2}}f\|_2 \quad (f\in \rH^1_L)
		\end{equation}
		and
		\begin{equation}\label{eq:StrSin}
			\|L^{\frac{1-\alpha}{2}}S(t,s)f\|_{\rL^p_t(\R;\rL_x^q(\R^d))}\lesssim \|f\|_2 \quad (f\in \rL^2)
		\end{equation}  
		uniformly in $s\in \R$.
		\item Suppose that $(\tilde{p},\tilde{q},\tilde{\alpha})$ is another Strichartz triple with $\tilde{p}'<p$. Then, there holds the inhomogeneous estimate
		\begin{equation*}
			\left\|\int_0^t L^{\tfrac{1-\alpha-\tilde{\alpha}}{2}}S(t,s) G(s)\dd s \right\|_{\rL^p_t(\R;(\rL^q_x(\R^d)))}\lesssim \|G\|_{\rL_t^{\tilde{p}'}(\R;\rL_x^{\tilde{q}'}(\R^d))}
		\end{equation*}
		for $G\in \rL^{\tilde{p}'}_t(\rL^{\tilde{q}'}_x(\R^d))$.
	\end{enumerate}
	\end{corollary}
	\begin{proof} Let $(p,q,\alpha)$ be wave-admissible and let $f\in \SL$. By Lemma~\ref{lemma:RelationBetweenParametrices}~(a) and Lemma~\ref{lemma:ExtensionOfPhillipsFunctionalCalculus}~(b), we may write
	\[L^{\frac{1-\alpha}{2}}C(t,s)f=\tfrac{1}{2}L^{-\frac{\alpha}{2}}T^+(t,s)L^{\frac{1}{2}}f+\tfrac{1}{2}L^{-\frac{\alpha}{2}}T^-(t,s)L^{\frac{1}{2}}f,\]
	so Theorem~\ref{thm:StrichartzEstimatesForParametricesI} gives $\|L^{\frac{1-\alpha}{2}}C(t,s)f\|_{\rL^p_t(\R;\rL_x^q(\R^d))}\lesssim \|L^{\frac{1}{2}}f\|_2$ uniformly in $s\in \R$. Now, \eqref{eq:StrCos} follows as $\SL$ is dense in $\rH_L^1$ (in particular, the left-hand side $L^{\frac{1-\alpha}{2}}C(t,s)f$ is to be understood as a limit in $\rL^p_t(\R;\rL^q_x(\R^d))$). This proves \eqref{eq:StrCos}. Let $\rho\in C_c^\infty(\R^d)$ with $0\leq \rho\leq 1$ and $\rho=1$ on $B(0,1)$. For $\lambda\geq 1$, put $\chi_\lambda\in C_c^\infty(\R^d)$ by $\chi_\lambda(\xi)\coloneqq \rho(\frac{\xi}{\lambda})-\rho(\lambda \xi)$. Then, by the definition of $\SL$, we have $\chi_\lambda(\sqrL)f=f$ for $\lambda=\lambda(f)\geq 1$ sufficiently large. In view of \eqref{def:Parametrices2}, we may then write
	\[L^{\frac{1-\alpha}{2}}S(t,s)f=\tfrac{1}{2}L^{-\frac{\alpha}{2}}T^+(t,s)m_{s,\lambda}(\sqrL)f-\tfrac{1}{2}L^{-\frac{\alpha}{2}}T^-(t,s)m_{s,\lambda}(\sqrL)f\]
	with $m_{s,\lambda}(\xi)\coloneqq \frac{|\xi|}{(B(s)\xi|\xi)^{1/2}}\chi_\lambda(\xi)$. Since $\|m_{s,\lambda}(\sqrL)f\|_2\lesssim \|m_{s,\lambda}\|_\infty \|f\|_2\leq \|f\|_2$ by Proposition~\ref{prop:PropertiesPhillipsCalculus}~(c), we argue as above to infer \eqref{eq:StrSin}. Finally, (b) is proved similarly as Theorem~\ref{thm:StrichartzEstimatesForParametricesI}~(b), using that $\|m_{s,\lambda}(\sqrL)\|_{\cL(\rL_x^{\tilde{q}'}(\R^d))}\lesssim 1$ by Proposition~\ref{prop:PropertiesPhillipsCalculus}~(c).
	\end{proof}
	From the preceding corollary and Theorem~\ref{thm:WellposednessInAdaptedSpaces}, we obtain Strichartz estimates for weak solutions in $\rH^1_L$. 
	\begin{theorem}[Global-In-Time Strichartz Estimates for Weak Solutions in $\rH^1_L$]\label{thm:GlobalStrichartzEstimatesOnAdaptedSpaces} Let $(p,q,\alpha)$ be a wave-admissible Strichartz triple. Suppose that $g\in \rH_L^{1}$, $h\in \rL^2$, and $F\in \rL^1(\R;\rL^{2})$. Then, the weak solution to the wave equation \eqref{eq:WaveEquationWithTimeDependentLipschitzMetrics} satisfies the global-in-time Strichartz estimate
	\begin{equation*}
		\|L^{\frac{1-\alpha}{2}}u\|_{\rL_t^p(\R;\rL_x^q)}\lesssim \|g\|_{\rH^1_L(\R^d)}+\|h\|_{\rL^2}+\|F\|_{\rL^1(\R;\rL^2)}.
	\end{equation*}
	\end{theorem}
	\begin{proof}By Theorem \ref{thm:WellposednessInAdaptedSpaces}, we have the representation
	\begin{equation*}
		u(t)=C(t,0)g+S(t,0)h+\int_0^t S(t,s)G(s)\dd s \quad (t\in \R) 
	\end{equation*}
	with
	\begin{equation}\label{eq:EstimateForVolterraTerm}
		\|G\|_{\rL^1(\R;\rL^2)}\lesssim \|g\|_{\rH_L^1}+\|h\|_{\rL^2}+\|F\|_{\rL^1(\R;\rL^2)}.
	\end{equation}
	By Corollary \ref{cor:StrichartzEstimatesParatetricesII} (a), we have
	\begin{align*}
		&\|L^{\frac{1-\alpha}{2}}C(t,0)g(x)\|_{\rL_t^p(\rL_x^q(\R^d))}\lesssim \|L^{\frac{1}{2}}g\|_2\lesssim \|g\|_{\rH^1_L}, \\ &\|L^{\frac{1-\alpha}{2}}S(t,0)g(x)\|_{\rL_t^p(\rL_x^q(\R^d))}\lesssim \|h\|_2.
	\end{align*}
	Moreover, applying Corollary \ref{cor:StrichartzEstimatesParatetricesII} (b) with $(\tilde{p}, \tilde{q},\tilde{\alpha})=(\infty,2,0)$ (note that then $\tilde{p}'=1<2\leq p$) and \eqref{eq:EstimateForVolterraTerm}, we find 
	\begin{align*}
		\bigg\|\int_0^t L^{\frac{1-\alpha}{2}} S(t,s)G(s )\dd s\bigg\|_{\rL_t^p(\rL_x^q(\R^d))}{}&\lesssim \|G\|_{\rL^1(\R;\rL^2)}\\
		{}&\lesssim\|g\|_{\rH_L^1}+\|h\|_{\rL^2}+\|F\|_{\rL^1(\R;\rL^2)}.
	\end{align*}
	\end{proof}
	Now, in view of Proposition~\ref{prop:Kato}, Theorem~\ref{mainthm:Strichartz} is an immediate consequence of Theorem~\ref{thm:GlobalStrichartzEstimatesOnAdaptedSpaces}:
	\begin{corollary}[Global-In-Time Strichartz Estimates for Weak Solutions in $\rH^1$]\label{cor:StrichartzEstimates} Let $(p,q,\alpha)$ be a wave-admissible Strichartz triple and $\alpha\in [0,2]$. Suppose that $g\in \rH^{1}$, $h\in \rL^{2}$ and $F\in \rL^1(\R;\rL^{2})$. Then, the weak solution to the wave equation \eqref{eq:WaveEquationWithTimeDependentLipschitzMetrics} satisfies the global-in-time Strichartz estimate
	\begin{equation}
		\||D_x|^{1-\alpha}u\|_{\rL_t^p(\R;\rL_x^q(\R^d))}\lesssim \|g\|_{\rH^1}+\|h\|_{\rL^2}+\|F\|_{\rL^1(\R;\rL^2)}.
	\end{equation}
	\end{corollary}
	\section{Oscillatory Integral Estimate}\label{section:OscillatoryIntegralEstimate}
	In this section, we provide a proof of Theorem~\ref{thm:OscillatoryIntegralEstimate}, which was one of the key ingredients for the proof of Theorem~\ref{mainthm:Strichartz}. For the convenience of the reader, we restate it here. Recall that we set for $\eps\in (0,\frac{1}{2})$
	\begin{equation*}
	\cB_\eps\coloneqq \{\tilde{B}\colon [0,1]\to \R^{d\times d} \mid \tilde{B}(t)\text{ is diagonal for all $t\in [0,1]$ and }\|\tilde{B}-\id\|_\infty\leq \eps\}
	\end{equation*}
	and define for $\tilde{B}\in \cB_\eps$ the phase function
	\begin{equation}\label{eq:DefinitionOfGeneralPhaseFunction2}
	 \varphi_{\tilde{B}}\colon \R^d\to \R, \quad \varphi_{\tilde{B}}(\xi)=\int_0^1 (\tilde{B}(\tau)\xi|\xi)^{\frac{1}{2}}\dd \tau.
	\end{equation}
	To ease notation, we just write $\tilde{\varphi}= \varphi_{\tilde{B}}$ as in Section~\ref{section:GlobalStrichartzEstimates}.
	\begin{theorem}\label{thm:OscillatoryIntegralEstimate2} There exists $\eps_0\in (0,\frac{1}{2})$ such that the following holds. For any $\tilde{\psi}\in C_c^\infty(\R^d)$ supported away from the origin, there exists some constant $C=C(\tilde{\psi},d)>0$ such that we have the decay estimate
	\begin{equation*}
		|\cF(\ee^{\ii t\tilde{\varphi}}\tilde{\psi})(y)|\leq C  (1+|t|)^{-\frac{d-1}{2}}
	\end{equation*}	 
	 for all $t\in \R, y\in \R^d,$ and $\tilde{B}\in \cB_{\eps_0}$.
	\end{theorem}
	For $\tilde{B}\in \cB_\eps$ and associated phase function $\tilde{\varphi}$, let us define the hypersurface
	\begin{equation}\label{eq:DefinitionOfHypersurface}
	\mathbb{S}\coloneqq \{\xi\in \R^d\colon \tilde{\varphi}(\xi)=1\}.
	\end{equation}
	We also define the hypersurface
	\begin{equation}\label{eq:DefinitionOfSingularSet}
	\Sigma\coloneqq \{\nabla_\xi \tilde{\varphi}(\xi)\in \R^d \mid \xi\neq 0\}.
	\end{equation}
	Theorem~\ref{thm:OscillatoryIntegralEstimate2} can be reduced to decay estimates for the Fourier transform of the surface measure on $\mathbb{S}$. It is well-known that these are linked to the Gaussian curvature on $\mathbb{S}$ (see e.g. \cite{Herz1962}, \cite{Littman1963},   \cite[Theorem~7.7.14]{Hoermander1990}). Therefore, we first collect some important geometrical properties of $\mathbb{S}$, see Proposition~\ref{prop:PropertiesOfSubmanifold} below. We then use these to prove Theorem~\ref{thm:DecayEstimateForUnitFrequencies} below, from which Theorem~\ref{thm:OscillatoryIntegralEstimate2} follows.
	
	Observe that $\nabla_{\xi}\tilde{\varphi}(\xi)=\tilde{\sB}(\xi)\xi $ for $\xi\neq 0$, where
	\begin{equation}\label{eq:RepresentationOfGradientOfGeneralPhaseFunction}
	\tilde{\sB}(\xi)\coloneqq \mathrm{diag}(\tilde{\sB}_1(\xi),\dots, \tilde{\sB}_d(\xi)), \quad \tilde{\sB}_j(\xi)\coloneqq  \int_0^1\frac{\tilde{b}_j(\tau)}{(\tilde{B}(\tau)\xi|\xi)^{1/2}}\dd \tau.
	\end{equation}
	Since $\|\tilde{B}-\id\|_\infty\leq \eps_0\leq \frac{1}{2}$, it follows immediately from \eqref{eq:DefinitionOfGeneralPhaseFunction2} and \eqref{eq:RepresentationOfGradientOfGeneralPhaseFunction} that
	\begin{equation}\label{eq:BoundednessOfSubmanifold}
	\mathbb{S}\subseteq \big\{\xi\in \R^d\colon |\xi|\simeq 1\big\},\quad \tilde{\sB}_j(\xi)\simeq |\xi|^{-1}, \quad \Sigma\subseteq  \big\{\xi\in \R^d\colon|\xi|\simeq 1\big\}.
	\end{equation}

	\begin{proposition}[Properties of $\mathbb{S}$]\label{prop:PropertiesOfSubmanifold} Let $\eps_0\in (0,\frac{1}{2})$, $\tilde{B}\in \cB_{\eps_0}$, and $\mathbb{S}$ and $\Sigma$ as defined in \eqref{eq:DefinitionOfHypersurface} and \eqref{eq:DefinitionOfSingularSet}, respectively. 
	\begin{enumerate}[label=(\alph*)]
		\item $\mathbb{S}$ is a smooth, compact hypersurface in $\R^d$, with normal space at each $\omega\in \mathbb{S}$ given by
		\begin{equation}\label{eq:NormalSpaceOfSubmanifold}
			N_{\omega}(\mathbb{S})=\mathrm{span}\big\{\nabla_\xi \tilde{\varphi}(\omega)\big\}\quad (\nabla_\xi \tilde{\varphi}(\omega)\in \Sigma).
		\end{equation}
		Moreover, $\mathbb{S}$ is the boundary of the compact, strictly convex set \[K\coloneqq \{\xi\in \R^d\mid \tilde{\varphi}(\xi)\leq 1\}.\]
		\item The Gauss map on $\mathbb{S}$ given by
		\begin{equation*}
			n\colon \mathbb{S}\to S^{d-1},\quad n(\omega)\coloneqq \frac{\nabla_\xi \tilde{\varphi}(\omega)}{|\nabla_\xi \tilde{\varphi}(\omega)|}
		\end{equation*}
		is a diffeomorphism with $n(-\omega)=-n(\omega)$. In particular, for each $\nu_0  \in S^{d-1}$ there exists exactly one $\omega_0\in \mathbb{S}^{d-1}$ such that $\pm \nu_0= n(\pm \omega_0)$.
		
		\item Let $n$ be the Gauss map from (b) and $\eps_0$ sufficiently small. Then, there exists some $\kappa=\kappa(\eps_0)>0$ such that in each point $\omega\in \mathbb{S}$, the principal curvatures $\kappa_1(\omega),\dots, \kappa_{d-1}(\omega)$ with respect to $-n$ satisfy $\kappa_j(\omega)\geq \kappa$.
	\end{enumerate}
	\end{proposition}
	\begin{proof} The assertions in (a) are straightforward to prove (see e.g. \cite[Proposition~5.1]{Tu2017}). To prove (c), let $n$ be the Gauss map as defined in (b) and $\omega\in \mathbb{S}$. Recall that the principal curvatures $\kappa_1(\omega),\dots,\kappa_{d-1}(\omega)$ with respect to $-n$ are the eigenvalues of the self-adjoint shape operator
	\begin{equation*}
		L_\omega \colon T_\omega(\mathbb{S}) \to T_\omega(\mathbb{S}),\quad L_\omega(v)=D_v n(\omega)
	\end{equation*} 
	(see e.g. \cite[Section~I.9]{Tu2017}). Here, $D_v$ denotes the directional derivative along the tangent vector $v\in T_\omega(\mathbb{S})$. Thus, we have to show that there exists some $\kappa=\kappa(\eps_0)\in (0,1)$ such that $\kappa_j(\omega)~\geq~\kappa$ for all $j\in \{1,\dots,d-1\}$ and $\omega\in \mathbb{S}$. Equivalently, we need to show that the bilinear form associated with $L_\omega$, namely the second fundamental form in $\omega$, defined by
	\begin{equation*}
		\mathrm{II}_\omega \colon T_{\omega}(\mathbb{S})\times T_{\omega}(\mathbb{S}) \to \R, \quad  \mathrm{II}_\omega(v,\tilde{v})=(L_\omega v| \tilde{v}),
	\end{equation*}
	is positive definite uniformly in $\omega\in \mathbb{S}$. Unfortunately, the strict convexity of $K$, as shown in (a), only implies that $\mathrm{II}_\omega$ is positive \textit{semi}-definite. To prove strict positive definiteness, we argue by perturbation. Fix $\omega=(\omega_1,\dots,\omega_d)\in \mathbb{S}$. A computation yields for any $v\in T_\omega(\mathbb{S})$
	\begin{equation}\label{eq:SecondFundamentalForm}
		|\nabla_\xi \tilde{\varphi}(\omega)|\cdot \mathrm{II}_\omega(v,v)=(\nabla^2_\xi \tilde{\varphi}(\omega)v|v)=(\tilde{\sB}(\omega)v|v)-(\tilde{\sR}(\omega)v|v)
	\end{equation}
	with
	\begin{equation}
		\tilde{\sR}(\omega)=(\tilde{r}_{j,k})_{j,k=1}^d=(r_{jk}(\omega)\omega_j\omega_k)_{j,k=1}^d,\quad r_{jk}(\omega)\coloneqq \int_{0}^1 \frac{\tilde{b}_j(\tau)\tilde{b}_k(\tau)}{(\tilde{B}(\tau)\omega|\omega)^{3/2}}\dd \tau. 
	\end{equation}
	Therefore, in view of \eqref{eq:BoundednessOfSubmanifold}, 
	\begin{equation*}
		\mathrm{II}_\omega(v,v)\gtrsim (\tilde{\sB}(\omega)v|v)-(\tilde{\sR}(\omega)v|v)\gtrsim |v|^2-(\tilde{\sR}(\omega)v|v).
	\end{equation*}
	So it is enough to show that
	\begin{equation}\label{eq:PertubationTerm}
		|(\tilde{\sR}(\omega)v|v)|=\mathcal{O}(\eps_0)\cdot |v|^2 \quad (v\in T_\omega(\mathbb{S})).
	\end{equation}
	Put $\nu\coloneqq \nabla_{\xi}\tilde{\varphi}(\omega)=\tilde{\sB}(\omega)\omega$. Without restriction, we may assume $\nu_d\neq 0$. Then, $T_\omega(\mathbb{S})=\mathrm{span}\{ v_1,\dots, v_{d-1}\}$ with $v_k\coloneqq \nu_d e_k-\nu_k  e_d$. Now, if $v\in T_\omega(\mathbb{S})$, then there exists $\lambda\in \R^{d-1}$ such that $v=\sum_{k=1}^{d-1}\lambda_k v_k$ and thus
	\begin{equation*}
		(\tilde{\sR}(\omega)v|v)=\sum_{j,k=1}^{d-1}\lambda_j\lambda_k (\tilde{\sR}(\omega)v_k|v_j)
	\end{equation*}
	with
	\begin{align*}
		(\tilde{\sR}(\omega)v_k|v_j)={}&\big[\big(r_{jk}\tilde{\sB}_d-r_{jd}\tilde{\sB}_k\big)\tilde{\sB}_d+\big(r_{dd}\tilde{\sB}_k-r_{dk}\tilde{\sB}_d\big)\tilde{\sB}_j\big]\omega_j\omega_k|\omega_d|^2\\
		={}&\big[\big(r_{jk}-r_{jd}\big)\tilde{\sB}_d^2+r_{jd}\big(\tilde{\sB}_d-\tilde{\sB}_k\big)\tilde{\sB}_d\big]\omega_j\omega_k|\omega_d|^2\\
		+{}&\big[\big(r_{dd}-r_{dk}\big)\tilde{\sB}_j\tilde{\sB}_k+r_{dk}\big(\tilde{\sB}_k-\tilde{\sB}_d\big)\tilde{\sB}_j\big]\omega_j\omega_k|\omega_d|^2.
	\end{align*}
	By assumption, we have $\|\tilde{b}_k-\tilde{b}_j\|_\infty\leq \|\tilde{b}_k-1\|_\infty+\|\tilde{b}_j-1\|_\infty\leq 2\eps_0$ for all $j,k\in \{1,\dots,d\}$. Hence, since $|\omega|\simeq 1$ by \eqref{eq:BoundednessOfSubmanifold},
	\begin{align*}
		|(r_{jk}-r_{jd})(\omega)|\leq{}& \bigg(\int_0^1 \frac{\tilde{b}_j(\tau)}{(\tilde{B}(\tau)\omega|\omega)^{3/2}}\dd \tau\bigg)\cdot \|\tilde{b}_k-\tilde{b}_d\|_{\infty}\lesssim \eps_0.
	\end{align*}
	and similarly
	\begin{equation*}
		|\tilde{\sB}_d(\omega)-\tilde{\sB}_k(\omega)|\leq \bigg(\int_0^1 \frac{ 1}{(\tilde{B}(\tau)\omega|\omega)^{1/2}}\dd \tau\bigg)\cdot \|\tilde{b}_k-\tilde{b}_d\|_{\infty}  \lesssim \eps_0. 
	\end{equation*}
	We conclude that $|(\tilde{\sR}(\omega)v_k|v_j|\lesssim \eps_0 |\omega_j||\omega_k| |\omega_d|^2$ and since $|\omega_d|\simeq |\nu_d|$ by \eqref{eq:BoundednessOfSubmanifold}, we obtain
	\begin{equation*}
		|(\tilde{\sR}(\omega)v|v)|\lesssim  \eps_0 |\nu_d|^2 \sum_{j,k=1}^{d-1}\lambda_j\lambda_k |\omega_j| |\omega_k| \leq \eps_0|\nu_d|^2 |\lambda|^2|\omega|^2\simeq \eps_0 |\nu_d|^2 |\lambda|^2 \leq \eps_0 |v|^2
	\end{equation*}
	(the last inequality follows from the identity $|v|^2=|(\lambda|\nu')|^2+|\nu_d|^2|\lambda|^2$, which is readily verified).
	This proves \eqref{eq:PertubationTerm} and therefore the claim.
	
	Finally, we prove (b). Note that by (c), $\dd n_\omega\colon T_\omega (\mathbb{S})\to T_{n_\omega}(S^{d-1})\simeq T_\omega(\mathbb{S})$ has nonvanishing determinant for each $\omega\in \mathbb{S}$, so $n$ is a local diffeomorphism. Also, $n(-\omega)=-n(\omega)$ for $\omega\in \mathbb{S}$ follows immediately from the definition of $n$. Therefore, it suffices to show that $n$ is bijective.\\[0.2cm]
	\textit{Injectivity}: Suppose that $\nu\coloneqq n(\omega_1)=n(\omega_2)$ for some $\omega_1,\omega_2\in \mathbb{S}$. By (a), the set $K$ is strictly convex and $n$ is the outer unit normal vector field on the boundary of $K$. Therefore, we must have $K\backslash\{\omega_j\}\subseteq H_j$, where $H_j$ is the open supporting hyperplane at $\omega_j$ defined by
	\begin{equation*}
		H_j\coloneqq \{\xi \in \R^d\colon (\xi-\omega_j|\nu)< 0 \} \quad (j\in \{1,2\}).
	\end{equation*}
	But this is only possible if $\omega_1=\omega_2$.\\[0.2cm]
	\textit{Surjectivity}: Let $\nu \in S^{d-1}$. We consider the linear function
	\begin{equation*}
		f\colon \R^d\to \R,\quad f(\omega)=\nu\cdot \omega.
	\end{equation*}
	By the extreme value theorem and the compactness of $K$, there are extremal points $\omega_1,\omega_2\in K$ in which $f$ attains a maximum and a minimum. The function $f$, being linear, cannot attain extremal values in the interior of $K$ and thus $\omega_1$ and $\omega_2$ have to lie on $\partial K=\mathbb{S}$. In particular, $f(\omega_1)$ and $f(\omega_2)$ are local extremas on $\mathbb{S}$. Therefore, $\nu=\nabla_\omega f(\omega_j)$ must belong to $N_{\omega_j}(\mathbb{S})$ by the Lagrange multiplier theorem. Combining this with (a), we conclude $\pm n(\omega_1)=\nu=\pm n(\omega_2)$. If $n(\omega_1)=n(\omega_2)$, then $\omega_1=\omega_2$ by the already shown injectivity. But then $f$ would have to be constant (since minimum and maximum of $f$ would be equal), contradicting the fact that $f\neq 0$ is linear. So either $\nu=n(\omega_1)$ or $\nu=n(\omega_2)$, which shows that $n$ is surjective.
	\end{proof}
	
		\begin{lemma}\label{lemma:SeparationOfNormals} Let $\eps_0\in (0,\frac{1}{2})$  and $\kappa=\kappa(\eps_0)$ be as in Proposition~\ref{prop:PropertiesOfSubmanifold}~(c). Then, there are constants $c_\kappa$ and $c_{d,\kappa}>0$ such that the following holds. For all $\delta\in (0,c_\kappa^{-1})$ and $\omega,\omega_0\in \mathbb{S}$ with $|\omega\pm \omega_0|\geq \delta$, there exists some $v\in T_{\omega}(\mathbb{S})$ with $|v|=1$ and
		\begin{equation*}
			|(v|n(\omega_0))|\geq c_{d,\kappa}\delta.
		\end{equation*}
	\end{lemma}
	
	\begin{proof} Let $\omega\in \mathbb{S}$. Recall that the principal curvatures $\kappa_1(\omega),\dots,\kappa_{d-1}(\omega)$ with respect to $-n$ are the eigenvalues of the self-adjoint shape operator $L_\omega \colon T_\omega(\mathbb{S})\to T_\omega(\mathbb{S})$, $L_\omega = D_v n$. Under the identification $T_\omega(\mathbb{S})\simeq T_{n(\omega)}(S^{d-1})$, we have $\dd n_\omega =L_\omega$ and thus by the inverse function theorem that $ n^{-1}$ is smooth with $\|\dd (n^{-1})_\nu\|=\|(\dd n_{n^{-1}(\nu)})^{-1}\|\leq \kappa^{-1}$ for all $\nu\in S^{d-1}$. Thus,
		\begin{equation*}
			|n^{-1}(\nu)-n^{-1}(\nu_0)| \leq  \|\dd n^{-1}\|_\infty \dd_{S^{d-1}}(\nu,\nu_0) \leq \tfrac{\pi}{2\kappa } |\nu-\nu_0| 
		\end{equation*}	
		for $\nu,\nu_0\in S^{d-1}$. In particular, $n^{-1}$ is Lipschitz continuous with Lipschitz constant $\tfrac{\pi}{2\kappa}$. Thus, if $\delta>0$ and $\omega,\omega_0\in \mathbb{S}$ with $|\omega\pm \omega_0|\geq \delta$, then
		\begin{align*}
			\big(\tfrac{2\kappa \delta}{\pi}\big)^2\leq (\tfrac{2\kappa}{\pi}|\omega-\omega_0|)^2\leq{}& |n(\omega)-n(\omega_0)|^2\\
			={}&|n(\omega)|^2+|n(\omega_0)|^2-2(n(\omega)|n(\omega_0))\\
			={}&2\big(1-(n(\omega)|n(\omega_0))\big),
		\end{align*}
		which implies $(n(\omega)|n(\omega_0))\leq 1-\tfrac{1}{2}(\tfrac{2\kappa\delta}{\pi})^2$. Replacing $\omega_0$ by $-\omega_0$ and using that $n(-\omega_0)=-n(\omega_0)$, we also get $-(n(\omega)|n(\omega_0))\leq 1-\tfrac{1}{2}(\tfrac{2\kappa\delta}{\pi})^2$ and thus
		\begin{equation*}
			|(n(\omega)|n(\omega_0))|\leq 1-\tfrac{1}{2}\big(\tfrac{2\kappa \delta}{\pi}\big)^2\eqqcolon  r
		\end{equation*}
		 Put $c_{\kappa}\coloneqq \sqrt{2}\pi^{-1}\kappa>0$ so that $r\in (0,1)$ for $\delta\in (0,\tfrac{1}{c_\kappa})$. Now choose an orthonormal basis $(v_1,\dots, v_{d-1})$ of $T_{\omega}(\mathbb{S})$. Then,
		\begin{equation*}
			\sum_{j=1}^{d-1}|(v_j|n(\omega_0))|^2=|n(\omega_0)|^2- |(n(\omega)|n(\omega_0)|^2\geq 1-r^2\geq 1-r.
		\end{equation*}
		Thus, there must exist some $j\in \{1,\dots,d-1\}$ such that
		\begin{equation*}
			|(v_j|n(\omega_0))|^2\geq \tfrac{1}{d-1}(1-r)=\tfrac{1}{2(d-1)}\big(\tfrac{2\kappa \delta}{\pi}\big)^2 \eqqcolon c_{d,\kappa}^2\delta^2,
		\end{equation*}
		proving the claim.
	\end{proof}
	
	Let $\tilde{\psi}\in C_c^\infty(\R^d)$ be supported away from the origin. For $\eps\in (0,\frac{1}{2})$, $\tilde{B}\in \cB_\eps$ and associated phase function $\tilde{\varphi}$ as in \eqref{eq:DefinitionOfGeneralPhaseFunction}, let us define for $t\in \R$ the kernel $K_t$ by
	\begin{equation*}
		K_t (y)=\cF^{-1}(\ee^{-\ii t\tilde{\varphi}}\tilde{\psi})(y)=\frac{1}{(2\pi)^d}\int_{\R^d} \ee^{\ii (y\cdot \xi-t\tilde{\varphi}(\xi))} \tilde{\psi}(\xi) \dd \xi\quad (y\in \R^d).
	\end{equation*}
	
	Theorem~\ref{thm:OscillatoryIntegralEstimate2} is then a consequence of the following result.
	\begin{theorem}[Decay Estimate for $K_{t}$]\label{thm:DecayEstimateForUnitFrequencies}Let $\eps_0\in (0,\frac{1}{2})$ as in Proposition~\ref{prop:PropertiesOfSubmanifold}~(c) be sufficiently small. Then, there are constants $C_N=C_N(\tilde{\psi})>0$ $(N\in \N_0)$ such that the following estimates hold true for all $s,t\in \R, y\in \R^d$, and all $N\in \N_0$:
		\begin{equation}\label{eq:DecayEstimateForUnitFrequencies}
			\big|K_{t}(y)\big|\leq 		C_N(1+|y|)^{-\frac{d-1}{2}} (1+\mathrm{dist}(y,t\Sigma))^{-N}.
		\end{equation}
		In particular, 	
		\begin{equation}\label{eq:L^inftyEstimateForHalfWaveKerlUnitFrequencies}
			\|K_{t}\|_{\rL^\infty}\lesssim_{\tilde{\psi}} 	(1+|t|)^{-\frac{d-1}{2}} \quad (t\in \R).
		\end{equation}
	
	\end{theorem} 
	\begin{proof} Let $t\in \R$. The claim is trivial if $t=0$, for then $K_{t}=\mathcal{F}^{-1}\tilde{\psi}\in \Sc(\R^d)$ and $t\Sigma=\{0\}$. We may therefore suppose $t\neq 0$ in the following. We divide the proof into three steps.\vspace{0.2cm}
		
	\textbf{Step 1: Rescaling}\vspace{0.2cm}
		
	Define the full phase function $\tilde{\Phi}(z,\xi)\coloneqq z\cdot \xi-\tilde{\varphi}(\xi)$ for $(z,\xi)\in \R^d\times \R^d$. Then,
		\begin{equation}\label{eq:RescalingKernel}
			K_{t}(y)={}\frac{1}{(2\pi)^d}\int_{\R^d} \ee^{\ii t \tilde{\Phi}\big(\frac{y}{t},\,\xi\big)}\tilde{\psi}(\xi)\dd \xi\eqqcolon \frac{1}{(2\pi)^d}I\big(t, \tfrac{y}{t}\big),
		\end{equation}
		where
		\begin{equation*}
			I(\lambda,z)\coloneqq  \int_{\R^d} \ee^{\ii \lambda\tilde{\Phi}(z,\xi)}\tilde{\psi}(\xi)\dd \xi \quad (z\in \R^d, \lambda\in \R),
		\end{equation*}
		and \eqref{eq:DecayEstimateForUnitFrequencies} would follow from the uniform estimates
		\begin{equation}\label{eq:StationaryPhaseRescaled}
			\big|I(\lambda,z)\big|\lesssim_{N, \tilde{\psi}}(1+|\lambda z|)^{-\frac{d-1}{2}} (1+|\lambda| \mathrm{dist}(z,\Sigma))^{-N} \quad (z\in \R^d,\lambda\in \R).
		\end{equation}
		This is what we will show next.\vspace{0.2cm}
		
		\textbf{Step 2 : The Proof of Estimate \eqref{eq:StationaryPhaseRescaled}}
		
		We deal first with the easier case $z=0$ and then proceed with the case $z\neq 0$.\vspace{0.2cm}

		\textit{Case 1 }: $z= 0$.\vspace{0.2cm}
		
		Put $\delta\coloneqq \mathrm{dist}(0,\Sigma)>0$. Then, $\delta\simeq 1$ by \eqref{eq:BoundednessOfSubmanifold}. Now, observe that 
		\begin{equation}\label{eq:NonStationaryPhasePhaseFunctionEstimate1}
			|\nabla_\xi\tilde{\Phi}(0,\xi)|=|0-\nabla_{\xi}\tilde{\varphi}(\xi)|\geq \mathrm{dist}(0, \Sigma)= \delta \simeq 1
		\end{equation}
		and that by the positive homogeneity of $\tilde{\varphi}$, we have for any $\alpha\in \N_0^d$ with $|\alpha|\geq 2$
		\begin{equation}\label{eq:NonStationaryPhasePhaseFunctionEstimate2}
			|\partial_\xi^\alpha \tilde{\Phi}(y,\xi)|=|\partial_\xi^\alpha \tilde{\varphi}(\xi)|\lesssim_{\alpha}|\xi|^{1-|\alpha|}\lesssim 1 \quad \text{on $\supp{\tilde{\psi}}$}.
		\end{equation}
		Therefore, nonstationary phase (see e.g. \cite[Theorem~7.7.1]{Hoermander1990}) implies
		\begin{equation*}
			|I(\lambda,0)|\lesssim_{N,d, \tilde{\psi}}(1+|\lambda|)^{-N}\lesssim (1+|\lambda|\, \delta)^{-N}\quad (\lambda\in \R)
		\end{equation*}
		which is exactly \eqref{eq:StationaryPhaseRescaled}.\vspace{0.2cm}

		\textit{Case 2 }: $z\neq 0$.\vspace{0.2cm}
		
		Let $z\in \R^d\setminus \{0\}$ and set $\nu_z\coloneqq \frac{z}{|z|}\in S^{d-1}$. Using \cite[Theorem~3.11]{EvansGariepy2015} (see also \cite[Proposition~A.0.1~(iv)]{Mesfun2026}), we may integrate the integrand of $I(\lambda,z)$ along the level sets of $\tilde{\varphi}$. This gives
		\begin{equation*}
			I(\lambda,z)={}\int_{0}^{\infty} \ee^{-\ii r\lambda} \bigg(\int_{\tilde{\varphi}^{-1}(r)} \ee^{\ii\lambda z \cdot \omega} \tilde{\psi}(\omega) \dd \omega\bigg)\dd r,
		\end{equation*}
		where $\dd \omega \coloneqq |\nabla_\xi \tilde{\varphi}|^{-1} \dd \mathcal{H}^{d-1}$ (the latter denoting the $d-1$-dimensional Hausdorff measure on the hypersurface $\tilde{\varphi}^{-1}(r)$). Since $\tilde{\varphi}^{-1}(r)=r\tilde{\varphi}^{-1}(1)=r\mathbb{S}$ by the positive homogeneity of $\tilde{\varphi}$, the integral in the above display is equal to 
		\begin{align*}
			I(\lambda,z)={}&\int_{0}^{\infty} \ee^{-\ii r\lambda} \bigg(\int_{\mathbb{S}} \ee^{\ii r\lambda z \cdot \omega} \tilde{\psi}(r\omega) \dd \omega\bigg)r^{d-1}\dd r\\
			={}& \int_{0}^{\infty} \ee^{-\ii r\lambda} J(r\lambda| z|,\nu_z)\tilde{\chi}(r) \dd r,
		\end{align*}
		where $\tilde{\chi}(r)\coloneqq \chi(r)r^{d-1}$ and $\chi\in C_c^\infty((0,\infty))$ such that $\chi=1$ on $\mathrm{supp}(r\mapsto \tilde{\psi}(r\omega))$, and for fixed $r>0$,
		\begin{equation*}
			J(\mu,\nu)\coloneqq \int_{\mathbb{S}} \ee^{\ii \mu \nu \cdot \omega} \psi(r\omega) \dd \omega  \quad (\mu\in \R,\nu\in S^{d-1}).
		\end{equation*}
		By Proposition \ref{prop:PropertiesOfSubmanifold}~(b), there is exactly one $\omega_z\in \mathbb{S}$ such that $\pm \nu_z=n(\pm \omega_z)$. We localize $J(\mu,\nu_z)$ around $\pm \omega_z$: To this end, we choose $\rho\in C_c^\infty (\R)$ such that $\rho=1$ on $(-\tfrac{1}{2},\tfrac{1}{2})$, $\supp{\rho}\subseteq (-1,1)$ and for some $\delta>0$, we split $J(\mu,\nu_z)={}J_1(\mu,\nu_z)+J_2(\mu,\nu_z)+J_3(\mu,\nu_z)$, where
		\begin{align*}
			J_{1}(\mu,\nu_z)\coloneqq{}& \int_{\mathbb{S}} \ee^{\ii\mu \nu_z \cdot \omega}  \rho\big(\tfrac{|\omega-\omega_z|}{\delta}\big)\psi(r\omega) \dd \omega ,\\
			J_{2}(\mu,\nu_z)\coloneqq{}& \int_{\mathbb{S}} \ee^{\ii  \mu \nu_z \cdot \omega}  \rho\big(\tfrac{|\omega+\omega_z|}{\delta}\big)\psi(r\omega) \dd \omega ,\\
			J_{3}(\mu,\nu_z)\coloneqq{}&  \int_{\mathbb{S}} \ee^{\ii \mu \nu_z\cdot \omega} \psi(r\omega)\bigg(1-\rho\big(\tfrac{|\omega-\omega_z|}{\delta}+\rho\big(\tfrac{|\omega+\omega_z|}{\delta}\big)\bigg) \dd \omega.
		\end{align*}
		 Then, we obtain the splitting $I(\lambda,z)=I_1(\lambda,z)+I_2(\lambda,z)+I_3(\lambda,z)$ with
		\begin{equation*}
			I_k(\lambda,z)=\int_{0}^{\infty} \ee^{-\ii r\lambda} J_k(r\lambda| z|,\nu_z)\tilde{\chi}(r) \dd r
		\end{equation*}
		for $k=\{1,2,3\}$. We estimate each of these terms separately. Let $k\in \{1,2\}$. Note that by Euler's relation for $\tilde{\varphi}$,
		\begin{align*}
			\nu_z \cdot \omega_z =  n(\omega_z)\cdot  \omega_z=\frac{\nabla_\xi\tilde{\varphi}(\omega_z)\cdot\omega_z}{|\nabla_\xi\tilde{\varphi}(\omega_z)|}=\frac{\tilde{\varphi}(\omega_z)}{|\nabla_\xi\tilde{\varphi}(\omega_z)|}=\frac{1}{|\nabla_\xi\tilde{\varphi}(\omega_z)|}\eqqcolon \frac{1}{\sigma}
		\end{align*}
		and therefore 
		\begin{equation*}
			I_k(\lambda,z)=\int_0^\infty \ee^{-\ii r\lambda(1-\frac{ |z|}{\sigma})}\cdot  \ee^{-\ii r\lambda\frac{| z|}{\sigma}} J_k(r\lambda| z|,\nu_z)\tilde{\chi}(r) \dd r.
		\end{equation*}
		Hence, integrating by parts and applying stationary phase (see \cite[Theorem~7.7.5, Theorem~7.7.14]{Hoermander1990}) with $\delta=\delta(\kappa(\eps_0))>0$ sufficiently small, we obtain
		\begin{align*}
			&\big|\big(\lambda\big(1-\tfrac{ |z|}{\sigma}\big)\big)^N I_k(\lambda,z)\big|\\={}&\bigg|\int_{0}^{\infty}\bigg( (-D_r)^N\ee^{-\ii r\lambda(1-\frac{|z|}{\sigma})}\bigg) \ee^{-\ii r\lambda\frac{| z|}{\sigma}} J_k(r\lambda| z|,\nu_z)\tilde{\chi}(r)  \dd r\bigg|\\
			={}&\bigg|\int_{0}^{\infty} \ee^{-\ii r\lambda(1-\frac{ |z|}{\sigma})} D_r^N\bigg( \ee^{-\ii r\frac{\lambda| z|}{\sigma}} J_k(r\lambda| z|,\nu_z)\tilde{\chi}(r) \bigg) \dd r\bigg|\\
			\leq{}&\int_{0}^{\infty}\bigg|D_r^N \bigg(\ee^{-\ii r\lambda\frac{| z|}{\sigma}} J_k(r\lambda| z|,\nu_z)\tilde{\chi}(r)\bigg)\bigg| \dd r\\
			\lesssim_{N}{}& \sum_{\ell=0}^N\int_{0}^{\infty} (1+r|\lambda  z|)^{-\frac{d-1}{2}} |\partial_r^{N-\ell}\tilde{\chi}(r)| \dd r\lesssim_{N,\tilde{\psi}} (1+ |\lambda z|)^{-\frac{d-1}{2}}.
		\end{align*}
		This yields
		\begin{equation*}
			|I_k(\lambda,z)|\lesssim_{N,\tilde{\psi}} (1+ |\lambda z|)^{-\frac{d-1}{2}} \big(1+\big|\lambda\big(1-\tfrac{ |z|}{\sigma}\big)\big|\big)^{-N}
		\end{equation*}
		for $k\in \{1,2\}$. For $k=3$, we use the principle of nonstationary phase. Indeed, Lemma~\ref{lemma:SeparationOfNormals} implies (for $\delta<c_\kappa^{-1}$, which we may assume without restriction) that $f\colon \mathbb{S}\to \R,$ $f(\omega)=\nu_z \cdot \omega$ satisfies
		\begin{equation*}
			\|\dd f_\omega\|\geq c_{d,\kappa}\,  \delta
		\end{equation*}
		uniformly $z\in \R^d\setminus \{0\}$, $\omega_z\in \mathbb{S}$ and $\omega\in \mathbb{S}\setminus (B_\delta(\omega_z)\cup B_\delta(-\omega_z))$. Thus, by nonstationary phase (see e.g. \cite[Theorem~7.7.1]{Hoermander1990}),
		\begin{equation*}
			|D_r^\ell J_3(\mu  , \nu_z)|\lesssim_{\ell, M} (1+|\mu|)^{-M}\quad (M,\ell\in \N_0).
		\end{equation*}
		Following the above argument in the case $k\in \{1,2\}$, we therefore obtain
		\begin{align*}
			&\big|\big(\lambda\big(1-\tfrac{|z|}{\sigma}\big)\big)^N I_3(\lambda,z)\big|\\
			\leq&\int_{0}^{\infty}\bigg|D_r^N \bigg(\ee^{-\ii r\lambda\frac{| z|}{\sigma}} J_3(r\lambda| z|,\nu_z)\tilde{\chi}(r)\bigg)\bigg| \dd r\\
			\lesssim&_{N,d} \sum_{\substack{\ell \in \N_0^3,\\ |\ell|=N}}\int_{0}^{\infty} \big|D_r^{\ell_1}\ee^{-\ii r\lambda\frac{| z|}{\sigma}}\big|\cdot \big|D_r^{\ell_2}J_3(r\lambda|  z|,\nu_z) \big|\cdot\big| D_r^{\ell_3}\tilde{\chi}(r)\big| \dd r\\
			\lesssim&_{N,M} \sum_{\substack{\ell \in \N_0^3,\\ |\ell|=N}}\int_{0}^{\infty} |\lambda z|^{\ell_1+\ell_2} (1+r|\lambda z|)^{-M}\cdot\big| D_r^{\ell_3}(\tilde{\chi}(r))\big| \dd r\\
			\lesssim{}& (1+|\lambda z|)^{-(M-N)}.
		\end{align*}
		Choosing $M=2N$ gives 
		\begin{equation*}
			|I_3(\lambda,z)|\lesssim_{N} (1+|\lambda z|)^{-N}\big(1+\big|\lambda\big(1-\tfrac{ |z|}{\sigma}\big)\big|\big)^{-N}.
		\end{equation*}
		Since $|I(\lambda,z)|\leq |I_1(\lambda,z)|+|I_2(\lambda,z)|+|I_3(\lambda,z)|$, we have proved that
		\begin{equation*}
			|I(\lambda,z)|\lesssim_N (1+ |\lambda z|)^{-\frac{d-1}{2}} \big(1+\big|\lambda\big(1-\tfrac{ |z|}{\sigma}\big)\big|\big)^{-N}.
		\end{equation*}
		
		Now it just remains to recall that $\sigma=|\nabla_\xi \tilde{\varphi}( \omega_z)|\simeq 1$ by \eqref{eq:BoundednessOfSubmanifold} as well as $ z=\frac{|z|}{\sigma}\nabla_\xi \tilde{\varphi}( \omega_z)$, which gives
		\begin{equation*}
			\mathrm{dist}(z,\Sigma)\leq \big|z-\nabla_\xi \tilde{ \varphi}(\omega_z)\big|=\sigma \big|1-\tfrac{ |z|}{\sigma}\big|  \simeq \big|1-\tfrac{ |z|}{\sigma}\big|.
		\end{equation*}
		This proves \eqref{eq:StationaryPhaseRescaled} as desired.\vspace{0.2cm}
		
		\textbf{Step 3 : The Proof of \eqref{eq:DecayEstimateForUnitFrequencies}}\vspace{0.2cm}
		
		Recall from \eqref{eq:BoundednessOfSubmanifold} that
		\begin{equation}\label{eq:BoundsForSingularSet1}
			\Sigma\subseteq \{\xi\in \R^d\colon \tilde{c}_1 \leq |\xi|\leq \tilde{c}_2\} 
		\end{equation}
		for some $\tilde{c}_1, \tilde{c}_2>0$. 
		Therefore, we deduce from \eqref{eq:StationaryPhaseRescaled} 
		the uniform bounds
		\begin{equation}\label{eq:LInftyEstimateOfI}
			|I(\lambda,z)| \lesssim (1+|\lambda|)^{-\frac{d-1}{2}} \quad (\lambda\in \R, z\in \R^d).
		\end{equation}
		Indeed, if $\mathrm{dist}(z, \Sigma)\leq \tfrac{\tilde{c}_1}{2}$, then \eqref{eq:BoundsForSingularSet1} implies $|z|\geq \frac{\tilde{c}_1}{2}$, so \eqref{eq:StationaryPhaseRescaled} with $N=0$ yields
		\begin{equation*}
			|I(\lambda,z)|\lesssim (1+|\lambda z|)^{-\frac{d-1}{2}}\lesssim_{\tilde{c}_1,d}  (1+|\lambda |)^{-\frac{d-1}{2}} \quad (\lambda\in \R, z\in \R^d).
		\end{equation*}
		On the other hand, if $\mathrm{dist}(z, \Sigma)\geq \frac{\tilde{c}_1}{2}$, then choosing some $N\in \N$ with $N\geq \frac{d-1}{2}$ in \eqref{eq:StationaryPhaseRescaled} yields
		\begin{equation*}
			|I(\lambda,z)|\lesssim_N (1+|\lambda |\mathrm{dist}(z, \Sigma))^{-N}\lesssim_{\tilde{c}_1,N}(1+|\lambda |)^{-\frac{d-1}{2}} \quad (\lambda\in \R, z\in \R^d).
		\end{equation*}
		This proves \eqref{eq:LInftyEstimateOfI}. Now, \eqref{eq:DecayEstimateForUnitFrequencies} is an immediate consequence of \eqref{eq:LInftyEstimateOfI} in view of \eqref{eq:RescalingKernel}. The proof is complete.
	\end{proof}


\begin{thebibliography}{10}
	

	
	\bibitem{ArendtBattyHieberNeubrander2011}
	W. Arendt, C. J. K. Batty, M. Hieber and F. Neubrander, {\em Vector-valued Laplace transforms and Cauchy problems}, second edition, 
	Monographs in Mathematics, 96, Birkh\"auser/Springer Basel AG, Basel, 2011.
	
	
	\bibitem{AuscherMcInstoshTchamitchian1998}
	P. Auscher, A.~G.~R. McIntosh and P. Tchamitchian, Heat kernels of second order complex elliptic operators and applications, J. Funct. Anal. {\bfseries 152} (1998), no.~1, 22--73.

	\bibitem{BahouriDanchinChemin2011}
	H. Bahouri, J.-Y. Chemin and R. Danchin, {\em Fourier analysis and nonlinear partial differential equations}, Grundlehren der mathematischen Wissenschaften, 343, Springer, Heidelberg, 2011.
	
	\bibitem{BeliIgnatZuaZua2016}
	C.~N. Beli, L.~I. Ignat and E. Zuazua, Dispersion for 1-D Schr\"odinger and wave equations with BV coefficients, Ann. Inst. H. Poincar\'e{} C Anal. Non Lin\'eaire {\bfseries 33} (2016), no.~6, 1473--1495.
	
	\bibitem{CazenaveHaraux1998}
	T. Cazenave and A. Haraux, {\em An introduction to semilinear evolution equations}, translated from the 1990 French original by Yvan Martel and revised by the authors, 
	Oxford Lecture Series in Mathematics and its Applications, 13, Oxford Univ. Press, New York, 1998.
	
	
	\bibitem{ColombiniFerruccioSpagnolo1979}
	F. Colombini, E. De~Giorgi and S. Spagnolo, Sur les \'equations hyperboliques avec des coefficients qui ne d\'ependent que du temps, Ann. Scuola Norm. Sup. Pisa Cl. Sci. (4) {\bfseries 6} (1979), no.~3, 511--559.
	
	\bibitem{EngelNagel2000}
	K.-J. Engel and R.~J. Nagel, {\em One-parameter semigroups for linear evolution equations}, Graduate Texts in Mathematics, 194, Springer, New York, 2000.
	
	\bibitem{Eskin2011}
	G. Eskin, {\em Lectures on linear partial differential equations}, Graduate Studies in Mathematics, 123, Amer. Math. Soc., Providence, RI, 2011.
	
	\bibitem{EvansGariepy2015}
	L.~C. Evans and R.~F. Gariepy, {\em Measure theory and fine properties of functions}, revised edition, 
	Textbooks in Mathematics, CRC Press, Boca Raton, FL, 2015.
	
	\bibitem{FangWang2006}
	D. Fang and C. Wang, Some remarks on Strichartz estimates for homogeneous wave equation, Nonlinear Anal. {\bfseries 65} (2006), no.~3.
	


	
	\bibitem{FreyMcPortal2018}
	D. Frey, A.~G.~R. McIntosh and P. Portal, Conical square function estimates and functional calculi for perturbed Hodge-Dirac operators in $L^P$, J. Anal. Math. {\bfseries 134} (2018), no.~2, 399--453.
	
		\bibitem{FrePortal2020}
	D. Frey and P. Portal, $L^p$ estimates for wave equations with specific $C^{0,1}$ coefficients, 
Ann. Inst. Fourier (Grenoble) (2026), Online first, 46 p.	DOI: 10.5802/aif.3763.
	
	
	\bibitem{FreySchippa2023}
	D. Frey and R. Schippa, Strichartz estimates for equations with structured Lipschitz coefficients, J. Evol. Equ. {\bfseries 23} (2023), no.~3, Paper No. 45, 34 pp.
	
	\bibitem{GinibreVelo1995} 
	J. Ginibre and G. Velo, Generalized Strichartz inequalities for the wave equation, J. Funct. Anal. {\bfseries 133} (1995), no.~1, 50--68.
		
	\bibitem{Guo2018}
	Z.~H. Guo, J. Li, K. Nakanishi and L. Yan, On the boundary Strichartz estimates for wave and Schr\"odinger equations, J. Differential Equations {\bfseries 265} (2018), no.~11, 5656--5675.
	
	\bibitem{Haase2006}
	M. Haase, {\em The functional calculus for sectorial operators}, Operator Theory: Advances and Applications, 169, Birkh\"auser, Basel, 2006.
	
	\bibitem{HassellRozendaal2022}
	A. Hassell and J. Rozendaal, $L^p$ and $\mathcal{H}_{FIO}^p$ regularity for wave equations with rough coefficients, Pure Appl. Anal. {\bfseries 5} (2023), no.~3, 541--599.
	
	\bibitem{Herz1962}
	C.~S. Herz, Fourier transforms related to convex sets, Ann. of Math. (2) {\bf 75} (1962), 81--92.
	
	\bibitem{Hoermander1990}
	L.~V. H\"ormander, {\em The analysis of linear partial differential operators. I}, second edition, 
	Grundlehren der mathematischen Wissenschaften, 256, Springer, Berlin, 1990.
	
	
	\bibitem{HytonenVanNeervenVeraarWeis2017}
	T.~P. Hyt\"onen, J. van Neerven, M. Veraar and L. Weis, {\em Analysis in Banach spaces. Vol. II. Probabilistic methods and operator theory}, Ergebnisse der Mathematik und ihrer Grenzgebiete. 3. Folge. A Series of Modern Surveys in Mathematics, 67, Springer, Cham, 2017.
		
	\bibitem{Kapitanski1989_1}
	L. Kapitanski, Leningrad Math. J. {\bfseries 1} (1990), no.~3, 693--726; translated from Algebra i Analiz {\bfseries 1} (1989), no.~3, 127--159.
	
	\bibitem{Kapitanski1989_2}
	L. Kapitanski, J. Soviet Math. {\bfseries 56} (1991), no.~2, 2348--2389; translated from Zap. Nauchn. Sem. Leningrad. Otdel. Mat. Inst. Steklov. (LOMI) {\bfseries 171} (1989), 106--162, 185--186. 
	
	\bibitem{KeelTao1998}
	M. Keel and T.~C. Tao, Endpoint Strichartz estimates, Amer. J. Math. {\bfseries 120} (1998), no.~5, 955--980.
	
	\bibitem{KunstmannWeis2004}
	P.~C. Kunstmann and L.~W. Weis, Maximal $L_p$-regularity for parabolic equations, Fourier multiplier theorems and $H^\infty$-functional calculus, in {\em Functional analytic methods for evolution equations}, 65--311, Lecture Notes in Math., 1855, Springer, Berlin.
	
	\bibitem{Lax1957}
	P.~D. Lax, Asymptotic solutions of oscillatory initial value problems, Duke Math. J. {\bfseries 24} (1957), 627--646.
	

	
	\bibitem{LindbladSogge1995}
	H. Lindblad and C.~D. Sogge, On existence and scattering with minimal regularity for semilinear wave equations, J. Funct. Anal. {\bfseries 130} (1995), no.~2, 357--426.
	
	\bibitem{Littman1963}
	W. Littman, Fourier transforms of surface-carried measures and differentiability of surface averages, Bull. Amer. Math. Soc. {\bfseries 69} (1963), 766--770.
	
		\bibitem{Lunardi2009}
	A. Lunardi, {\em Interpolation theory}, second edition, 
	Appunti. Scuola Normale Superiore di Pisa (Nuova Serie), Ed. Norm., Pisa, 2009.

	\bibitem{McIntoshNahmod2000}
	A.~G.~R. McIntosh and A.~R. Nahmod, Heat kernel estimates and functional calculi of $-b\Delta$, Math. Scand. {\bfseries 87} (2000), no.~2, 287--319.
	
	\bibitem{Mesfun2026}
	Y. Mesfun, On an Operator-Theoretic Approach to Strichartz Estimates for Rough Wave Equations, PhD dissertation, 2026, KIT. \url{https://publikationen.bibliothek.kit.edu/1000193300}. DOI: 10.5445/IR/1000193300.
	
	\bibitem{MetcalfeTataru2012}
	J.~L. Metcalfe and D. Tataru, Global parametrices and dispersive estimates for variable coefficient wave equations, Math. Ann. {\bfseries 353} (2012), no.~4, 1183--1237.
	
	\bibitem{MockenhauptSeegerSogge1993}
	G. Mockenhaupt, A. Seeger and C.~D. Sogge, Local smoothing of Fourier integral operators and Carleson-Sj\"olin estimates, J. Amer. Math. Soc. {\bfseries 6} (1993), no.~1, 65--130.
	
	
	\bibitem{Pazy2011}
	A. Pazy, {\em Semigroups of linear operators and applications to partial differential equations}, Applied Mathematical Sciences, 44, Springer, New York, 1983.
	
	\bibitem{SmithParametrix1998}
	H.~F. Smith, A parametrix construction for wave equations with $C^{1,1}$ coefficients, Ann. Inst. Fourier (Grenoble) {\bfseries 48} (1998), no.~3, 797--835.
	
	\bibitem{SmithTataru2002} 
	H.~F. Smith and D. Tataru, Sharp counterexamples for Strichartz estimates for low regularity metrics, Math. Res. Lett. {\bfseries 9} (2002), no.~2-3, 199--204.
	
	
	\bibitem{Sogge1995}
	C.~D. Sogge, {\em Lectures on non-linear wave equations}, second edition, 
	Int. Press, Boston, MA, 2008.
	
	\bibitem{Stein1993}
	E.~M. Stein, {\em Harmonic analysis: real-variable methods, orthogonality, and oscillatory integrals}, Princeton Mathematical Series Monographs in Harmonic Analysis, 43 III, Princeton Univ. Press, Princeton, NJ, 1993.
	

	
	\bibitem{Strichartz1977}
	R.~S. Strichartz, Restrictions of Fourier transforms to quadratic surfaces and decay of solutions of wave equations, Duke Math. J. {\bfseries 44} (1977), no.~3, 705--714.
	
	\bibitem{Tao2006}
	T.~C. Tao, {\em Nonlinear dispersive equations}, CBMS Regional Conference Series in Mathematics, 106, Conf. Board Math. Sci., Washington, DC, 2006 Amer. Math. Soc., Providence, RI, 2006.
	
	\bibitem{Tataru2001}
	D. Tataru, Strichartz estimates for second order hyperbolic operators with nonsmooth coefficients. II, Amer. J. Math. {\bfseries 123} (2001), no.~3, 385--423.
	
	\bibitem{Tataru2002}
	D. Tataru, Strichartz estimates for second order hyperbolic operators with nonsmooth coefficients. III, J. Amer. Math. Soc. {\bfseries 15} (2002), no.~2, 419--442.
	
	\bibitem{TataruGlobal}
	D. Tataru, Global Strichartz estimates for variable coefficient second order hyperbolic operators, in {\em S\'eminaire: \'Equations aux D\'eriv\'ees Partielles, 1999--2000}, Exp. No. XI, 17 pp., S\'emin. \'Equ. D\'eriv. Partielles, \'Ecole Polytech., Palaiseau.
	
	\bibitem{Tu2017}
	L.~W. Tu, {\em Differential geometry}, Graduate Texts in Mathematics, 275, Springer, Cham, 2017.

	
\end{thebibliography}
\end{document}